\documentclass[11pt]{amsart}

\usepackage[margin=.83in]{geometry}
\usepackage{tikz,tkz-euclide}

\makeatletter
\renewenvironment{abstract}
 {\small
  \begin{center}
  {\bfseries Abstract}
  \end{center}
  \list{}{\leftmargin=0pt \rightmargin=0pt}
  \item\relax}
 {\endlist}
\makeatother

\usepackage{amsmath,amssymb,amsthm,amsfonts,tikz}

\usepackage{xcolor}
\usepackage[
    colorlinks=true,
    linkcolor=blue,
    citecolor=blue,
    urlcolor=blue
]{hyperref}
\usepackage[capitalize]{cleveref}

\usepackage{caption}

\usepackage[normalem]{ulem} 

\numberwithin{equation}{section}

\usepackage{aliascnt}
\newtheorem{theorem}{Theorem}[section]

\newaliascnt{proposition}{theorem}
\newtheorem{proposition}[proposition]{Proposition}
\aliascntresetthe{proposition}

\newaliascnt{prop}{theorem}
\newtheorem{prop}[prop]{Proposition}
\aliascntresetthe{prop}

\newaliascnt{lemma}{theorem}
\newtheorem{lemma}[lemma]{Lemma}
\aliascntresetthe{lemma}

\newaliascnt{corollary}{theorem}
\newtheorem{corollary}[corollary]{Corollary}
\aliascntresetthe{corollary}

\theoremstyle{definition}

\newaliascnt{definition}{theorem}
\newtheorem{definition}[definition]{Definition}
\aliascntresetthe{definition}

\newaliascnt{remark}{theorem}
\newtheorem{remark}[remark]{Remark}
\aliascntresetthe{remark}

\newaliascnt{example}{theorem}

\aliascntresetthe{example}

\newcommand{\Eff}{\operatorname{Eff}}

\newcommand{\Nef}{\operatorname{Nef}}
\newcommand{\Mov}{\operatorname{Mov}}
\newcommand{\Pic}{\operatorname{Pic}}

\newcommand{\Fan}{\operatorname{Fan}}

\newcommand{\ex}{\operatorname{ex}}

\newcommand{\Cl}{\operatorname{Cl}}

\newcommand{\Proj}{\operatorname{Proj}}
\newcommand{\relint}{\operatorname{relint}}
\DeclareMathOperator{\Cone}{Cone}

\newcommand{\Bdiv}{\mathbf{B}^{\mathrm{div}}}                   
\newcommand{\Bigc}{\operatorname{Big}}                          
\newcommand{\rank}{\operatorname{rank}}

\crefname{theorem}{Th.}{Ths.}
\Crefname{theorem}{Theorem}{Theorems}
\crefname{lemma}{Lem.}{Lems.}
\Crefname{lemma}{Lemma}{Lemmas}
\crefname{corollary}{Cor.}{Cors.}
\Crefname{corollary}{Corollary}{Corollaries}
\crefname{proposition}{Prop.}{Props.}
\Crefname{proposition}{Proposition}{Propositions}
\crefname{prop}{Prop.}{Props.}
\Crefname{prop}{Proposition}{Propositions}
\crefname{definition}{Def.}{Defs.}
\Crefname{definition}{Definition}{Definitions}
\crefname{remark}{Rem.}{Rems.}
\Crefname{remark}{Remark}{Remarks}
\crefname{example}{Ex.}{Exs.}
\Crefname{example}{Example}{Examples}

\usepackage{cite}
\usepackage[utf8]{inputenc}
\usepackage{algorithm}
\usepackage{algpseudocode}

\usepackage{array}
\newcolumntype{L}[1]{>{\raggedright\arraybackslash}p{#1}}

\title[Global Sections and Birational Geometry of Calabi--Yau Type Varieties]{Global Sections and Birational Geometry\\of Calabi--Yau Type Varieties}

\author{Andrei Constantin}
\address{School of Mathematics, University of Birmingham, Watson Building, Edgbaston, Birmingham B15 2TT,~UK}
\email{a.constantin@bham.ac.uk}

\author{Andre Lukas}
\address{Rudolf Peierls Centre for Theoretical Physics, University of Oxford, Parks Road, Oxford OX1 3PU,~UK}
\email{lukas@physics.ox.ac.uk}

\author{Elijah Sheridan}
\address{Department of Physics, Cornell University, Ithaca, NY 14853, USA}
\email{es888@cornell.edu}

\begin{document}
\maketitle

\begin{abstract}
We investigate aspects of the relationship between global sections of divisors and birational geometry. We show that for varieties $X$ of klt Calabi--Yau type, global sections are governed by the same birational data that control the $D$-minimal model program ($D$-MMP). As a consequence, global sections of arbitrary big divisors can be systematically reduced to Euler characteristics $\chi(Y,\mathcal O_Y(D))$ on suitable birational models $Y$ using vanishing theorems, extending classical positivity-based methods beyond the movable cone. The reduction proceeds by removing fixed divisorial components, which do not contribute to the space of global sections, either directly by subtraction or birationally by contraction to a suitable $D$-minimal model. We obtain piecewise quasipolynomial formulae for global sections on the big cone, whose domains are the Mori chambers of the $D$-MMP. Restricting to the subclass of Fano type varieties extends such formulae to the entire effective cone. These formulae can be constructed from just the small birational class of~$X$, yet encode properties of all birational contractions of~$X$. Alongside the computation of global sections from birational geometry, we study the inverse process, wherein finite computations of $h^0(X,\mathcal O_X(D))$ determine aspects of the chamber decomposition and the associated birational models.

\end{abstract}

\section{Introduction}

Global sections of divisors lie at the heart of birational geometry. Indeed, the rational maps determined by linear systems of divisors are one of the fundamental tools employed in birational geometry. The purpose of this paper is to develop a more systematic understanding of how this relationship can be inverted: how global sections are themselves governed --- and can be counted --- by birational geometry.

We illustrate how for certain normal $\mathbb{Q}$-factorial projective varieties $X$ on which the $D$-minimal model program ($D$-MMP) can be run, the structure of global sections is governed by the geography of those $D$-minimal models \cite{shokurov19963, kaloghiros2016finite}, that is, the decomposition of the effective cone into Mori chambers of divisor classes with a common $D$-minimal model. In particular, the Mori chambers are the domains of a natural piecewise quasipolynomial formula for the function $D \mapsto h^0(X,\mathcal O_X(D))$, achieved by reducing to Euler characteristics $\chi(Y,\mathcal O_Y(D_Y))$ of nef divisors on suitable birational models $Y$ using vanishing theorems. This reduction has long been understood for movable divisors: we illustrate how divisors with fixed divisorial components can themselves be reduced to the movable case through systematic operations that remove these fixed components --- birational contraction or direct subtraction. In this work we focus on varieties of Calabi--Yau type (paying special attention to the subclass of varieties of Fano type), using the results of \cite{gachet2024effective} regarding $D$-minimal models and their geography in this setting, though our methods are relevant in any settings where $D$-minimal models exist.

Our work is motivated in part by recent progress on explicit formulae for line bundle cohomology more broadly. In applications to particle spectra in string compactifications, large-scale computations revealed striking piecewise polynomial and quasipolynomial behaviour of the functions $D\mapsto h^i(X,\mathcal O_X(D))$~\cite{Constantin:2018hvl, Klaewer:2018sfl, Larfors:2019sie, Brodie:2019dfx, Brodie:2019pnz, Brodie:2021zqq, Brodie:2020fiq, Brodie:2021nit}. Subsequent mathematical work established and explained such formulae for several classes of smooth projective surfaces~\cite{Brodie:2019ozt, Brodie:2020wkd} and complete intersection Calabi--Yau threefolds~\cite{Wang:2026nfz}. More recently, generating-function approaches have suggested that, in a range of examples, a single rational function can encode both zeroth and higher cohomologies through different expansions~\cite{Constantin:2024ulu}. 

Closely related is the theory of asymptotic invariants of divisors, 
which studies numerical functions obtained from the asymptotic behaviour of line bundles.
Examples of these invariants include the volume of a divisor, i.e.~the normalized leading-order growth of global sections along rays in $N^1(X)_{\mathbb R}$, as well as asymptotic higher cohomology functions; for a review, see \cite{ein2005asymptotic}. On projective $\mathbb Q$-factorial toric varieties these functions admit piecewise polynomial descriptions on finite decompositions of the effective cone, namely the Gelfand--Kapranov--Zelevinsky (equivalently, Mori chamber) decomposition~\cite{hering2006asymptotic}, or refinements thereof. Similar results hold more broadly for Mori dream spaces \cite{ein2006asymptotic}.

In the present paper we focus on exact global sections, rather than their asymptotic growth and show that the $D$-MMP and geography of models provide a general framework for explaining and extending the piecewise polynomial and quasipolynomial formulae observed in earlier examples.

\subsection{Methodology}

\begin{figure}
\begin{center}
\resizebox{\textwidth}{!}{%
\begin{tikzpicture}[line width=0.6pt]


\definecolor{bigcol}{RGB}{31,78,121}
\definecolor{effcol}{RGB}{176,74,10}
\tikzset{
  setbox/.style   = {draw, line width=0.6pt},          
  cybox/.style    = {draw=bigcol, line width=1.1pt},                    
  mdsbox/.style   = {draw=effcol, line width=1.1pt, densely dashed},    
  fanobox/.style  = {draw=effcol, line width=1.1pt},                    
  keytext/.style  = {font=\footnotesize, anchor=west, inner sep=1pt},
  keyhead/.style  = {font=\footnotesize\itshape, inner sep=1pt},
  keyrow/.style   = {font=\footnotesize\itshape, inner sep=1pt},
  keyrule/.style  = {draw=black!35, line width=0.4pt},
  keynone/.style  = {font=\footnotesize, text=black!40, inner sep=1pt},
  universe/.style = {draw, line width=0.4pt},           
  setname/.style     = {font=\footnotesize, inner sep=2pt},
  wit/.style      = {font=\scriptsize, align=center,    
                     inner sep=1pt},
}

\coordinate (Ull) at (-.8, 0.0);   \coordinate (Uur) at (15.8, 6.75);
\coordinate (All) at ( 0.0, 1.25);   \coordinate (Aur) at (11.0, 6.40);
\coordinate (Bll) at ( 4.3, 0.80);   \coordinate (Bur) at (15.0, 5.90);
\coordinate (Cll) at ( 1.2, 4.70);   \coordinate (Cur) at ( 7.4, 5.60);
\coordinate (Dll) at ( 5.3, 1.55);   \coordinate (Dur) at (10.7, 4.40);
\coordinate (Ell) at ( 5.7, 1.55+0.40);   \coordinate (Eur) at ( 8.9, 3.30);
\coordinate (Fll) at ( 7.1, 2.45);   \coordinate (Fur) at (10.3, 4.40-0.40);
\coordinate (Gc)  at ( 8.0, 3.4+0.7);

\draw[universe] (Ull) rectangle (Uur);
\draw[cybox]    (All) rectangle (Aur);   
\draw[mdsbox]   (Bll) rectangle (Bur);   
\draw[setbox]   (Cll) rectangle (Cur);   
\draw[fanobox]  (Dll) rectangle (Dur);   
\draw[setbox]   (Ell) rectangle (Eur);   
\draw[setbox]   (Fll) rectangle (Fur);   

\node[setname, anchor=south west] at ( 0.15, 0.15) {normal $\mathbb{Q}$-factorial projective varieties};
\node[setname, anchor=north west, text=bigcol] at ( 0.15, 6.25) {Calabi--Yau type};
\node[setname, anchor=north east, text=effcol] at (14.85-.5, 5.75) {Mori dream spaces};
\node[setname, anchor=north west] at ( 1.55, 5.40) {Calabi--Yau};
\node[setname, anchor=north west, text=effcol] at ( 5.40, 4.00) {Fano type};
\node[setname, anchor=south west, align=left]
                               at ( 5.80, 2.05) {toric\\varieties};
\node[setname, anchor=north east] at (10.20, 3.90) {Fano};
\node[keyhead] at ( 6.20,-0.65) {$D$-MMP};
\node[keyhead] at (11.40,-0.65) {$D$-MMP and suitable KV vanishing};
\draw[keyrule] ( 4.45,-0.90) -- (14.20,-0.90);
\draw[keyrule] ( 4.45,-0.40) -- ( 4.45,-2.20);
\node[keyrow, anchor=east] at ( 4.25,-1.35) {on $\Bigc(X)$};
\node[keyrow, anchor=east] at ( 4.25,-1.95) {on $\Eff(X)$};
\node[keynone] at ( 6.20,-1.35) {---};
\draw[cybox]   ( 9.70,-1.35) -- (10.50,-1.35);
\node[keytext] at (10.65,-1.35) {Calabi--Yau type};
\draw[mdsbox]  ( 4.70,-1.95) -- ( 5.50,-1.95);
\node[keytext] at ( 5.65,-1.95) {Mori dream spaces};
\draw[fanobox] ( 9.70,-1.95) -- (10.50,-1.95);
\node[keytext] at (10.65,-1.95) {Fano type};

\end{tikzpicture}%
}
\end{center}
\caption{Relations among the main classes of varieties considered in this paper, emphasizing the overlap between Calabi--Yau type, Fano type, and Mori dream spaces that underlies our birational framework.}
\label{fig:variety_containment}
\end{figure}
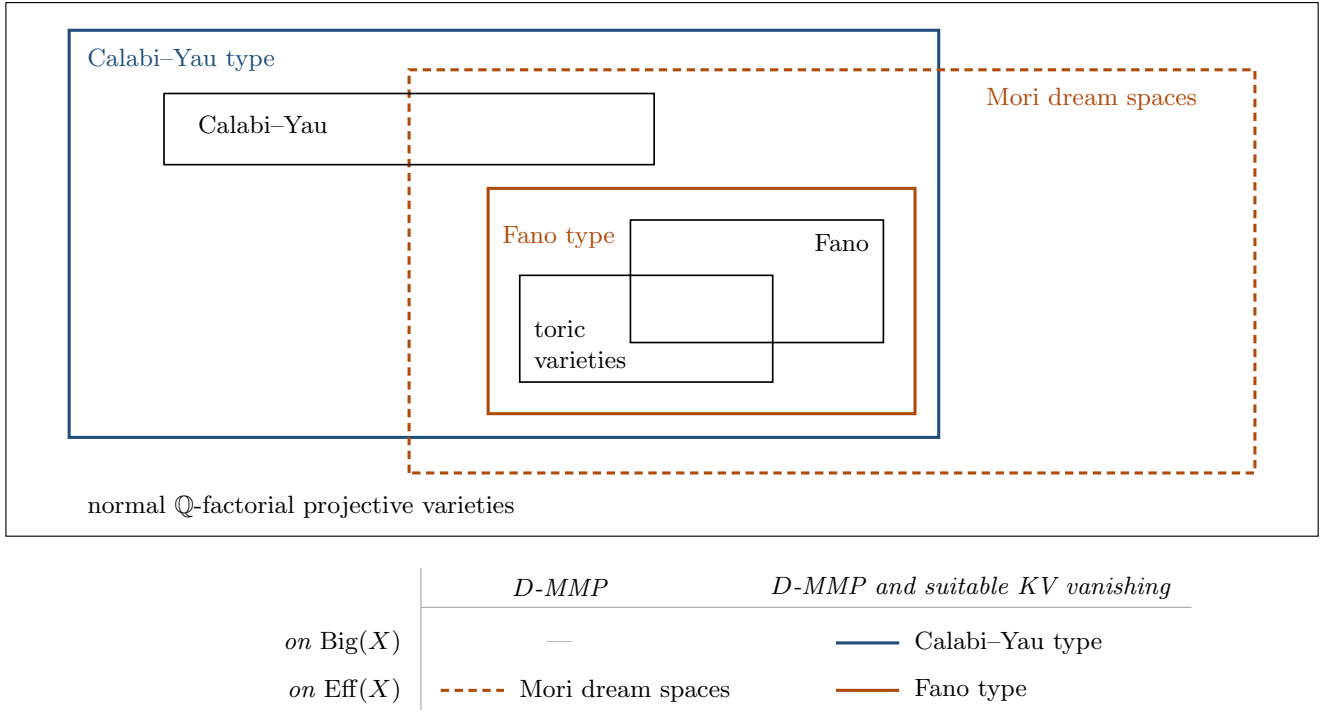

In this work we reduce the study of global sections to an Euler characteristic, which is a quasipolynomial function on $\mathbb{Q}$-factorial varieties by Snapper's theorem (\cref{th:snapper}). This method hinges on specific applications of the $D$-MMP and Kawamata--Viehweg (KV) vanishing (\cref{th:KV}). For varieties of Calabi--Yau type, these applications are available for all big divisors. Mori dream spaces provide a distinct setting in which the $D$-MMP can be run for every effective divisor, but the required vanishing need not hold in general; restricting to varieties of Fano type, which lie in the intersection of these two classes as illustrated in \cref{fig:variety_containment}, ensures that both ingredients are available throughout the effective cone. The main result of this paper is then that for varieties of Calabi--Yau type (Fano type) the Mori chamber decomposition of the big cone (effective cone) furnishes the domains of a piecewise quasipolynomial formula for global sections.

\begin{definition}
\label{def:CY_fano_type}
A normal projective variety $X$ is of \emph{klt Calabi--Yau type} if there exists an effective $\mathbb Q$-divisor $\Delta$ such that $(X,\Delta)$ is klt and $K_X+\Delta \equiv 0$. It is of \emph{Fano type} if there exists an effective $\mathbb Q$-divisor $\Delta$ such that $(X,\Delta)$ is klt and $-(K_X+\Delta)$ is ample.
\end{definition}

Throughout this paper, varieties of klt Calabi--Yau type and Fano type are additionally assumed to be $\mathbb Q$-factorial, to have finitely generated class group, and to satisfy $\Pic^0(X)=0$. Unless stated otherwise, the terms ``Calabi--Yau type'' and ``Fano type'' will refer to varieties satisfying these additional hypotheses. 

For simplicity, we sketch the approach for a Fano type variety $X$, thereby avoiding the more delicate, potentially infinite birational geometry of Calabi--Yau type varieties. As has long been understood, positive divisors readily admit formulae for global sections: for Fano type varieties, we may proceed as follows. The moving cone $\Mov(X)$ decomposes as a union of nef cones $\varphi_i^* \Nef(X_i)$ over the small $\mathbb{Q}$-factorial modifications (compositions of flips) $\varphi_i : X \dashrightarrow X_i$ which $X$ admits. Consequently, any movable divisor $D$ belongs to some $\varphi_i^* \Nef(X_i)$, hence $\varphi_{i*} D$ is nef on $X_i$ and a basic application of Kawamata--Viehweg vanishing (\cref{cor:MAIN-VANISHING}) identifies
\begin{equation}
    \label{eq:intro}
    h^0(X, \mathcal{O}_X(D)) = h^0(X_i, \mathcal O_{X_i}(\varphi_{i*}D)) =\chi(X_i, \mathcal{O}_{X_i}(\varphi_{i*} D))~.
\end{equation}
As the Euler characteristic on $X_i$ is quasipolynomial in the divisor class, and the birational model $X_i$ is constant on the chamber $\varphi_i^*\Nef(X_i)$, this yields a quasipolynomial formula for $h^0(X,\mathcal O_X(D))$ on that chamber. As the movable cone is covered by these chambers, $h^0$ is piecewise quasipolynomial on $\Mov(X)$.

The challenge is achieving formulae for the divisors in $\Eff(X) \setminus \Mov(X)$, where vanishing theorems are inapplicable. As illustrated schematically in \cref{fig:Eff_schematic}, these divisors lie in Mori chambers outside the movable cone and carry fixed divisorial components. Our strategy is to exploit that such fixed components do not affect the number of global sections. We present two strategies for removing these fixed components and reducing the global sections of any effective divisor to the Euler characteristic of a movable divisor. Both depend on the existence of $D$-minimal models and the associated Mori chamber decomposition.

\begin{figure}
\centering
\scalebox{0.8}{\begin{tikzpicture}[scale=2.6, line join=round, line cap=round]
\coordinate (O)  at (0,0);
\coordinate (r0) at (0.04,2.05);
\coordinate (r1) at (1.60,1.72);
\coordinate (r2) at (2.42,0.38);
\coordinate (r3) at (2.06,-1.14);
\fill[gray!30] (O) -- (r0) -- (r1) -- cycle;
\fill[gray!30] (O) -- (r1) -- (r2) -- cycle;
\fill[gray!10] (O) -- (r2) -- (r3) -- cycle;
\draw[thick] (O) -- (r0);
\draw[thick] (O) -- (r1);
\draw[thick] (O) -- (r2);
\draw[thick] (O) -- (r3);
\path (O) -- (r2)
  node[pos=1.15, sloped, inner sep=1.5pt]
  {$f_2^*\Nef(X_2)$};
\path (O) -- (r3)
  node[pos=1.1, sloped, inner sep=1.5pt]
  {$\mathbb{R}_{\geq0}E$};
\path (O) -- (r3) node[pos=0.3333, circle, fill, inner sep=1.2pt] (Ept) {};
\node[below=6pt, align=center] at (Ept) {$E$\\[0pt] ($f_2$-exceptional)};
\node[align=center] at (0.64,1.44)
{ $\mathcal{C}_1 = f_1^*\Nef(X_1)$\\[0pt]
 ($f_1$ an SQM)};
\node[align=center] at (1.41,0.70)
{$\mathcal{C}_0 = \Nef(X)$\\[-2pt]};
\node[align=center] at (1.4,-0.10)
{$\mathcal{C}_2 = $ classes with stable \\[0pt]
 fixed component $E$};
\end{tikzpicture}}
\caption{Schematic picture of the Mori chamber decomposition of the effective cone underlying the two methods of the paper. The movable cone, indicated by shading, is covered by pullbacks of nef cones of SQMs, on which vanishing theorems identify global sections with Euler characteristics. Outside the movable cone, fixed divisorial components can be removed either directly by subtraction or birationally by contraction to a $D$-minimal model. For varieties of Calabi--Yau type this yields piecewise quasipolynomial formulae on the big cone, while for varieties of Fano type the same mechanism extends to the entire effective cone.}
\label{fig:Eff_schematic}
\end{figure}
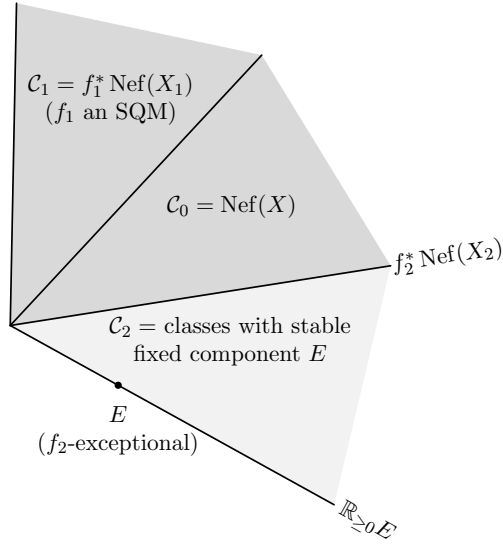

First, we consider the \textit{contraction} of fixed components. 
In our setting, one can always run a $D$-MMP and render an effective divisor nef. We employed this above in the special case of a movable divisor, for which only SQMs were necessary, but the $D$-MMP also renders non-movable divisors nef: it is just that divisorial contractions are required in addition to SQMs. The decomposition of the movable cone into nef cones above is a special case of the decomposition of the effective cone into Mori chambers. 
In particular, for any effective divisor, one can pass to the $D$-minimal model $f : X \dashrightarrow Y$ associated to its Mori chamber, at which point we can again apply \cref{eq:intro} (for $f$ rather than $\varphi_i$).

Second, we consider the \textit{subtraction} of fixed components of the complete linear system of a divisor.
Naively, determining the fixed components essentially requires knowing the divisor's global sections to begin with. However, a divisor's stable base locus is actually determined by its Mori chamber, as we will formalize through the Zariski decomposition (\cref{prop:zariski}). This means that the geography of models is also a geography of (stable) fixed components, enabling the efficient iterative subtraction of such components. Subtraction results in a divisor on $X$ with the same number of global sections but which resides in $\Mov(X)$, making the logic of \cref{eq:intro} again applicable.

\subsection{Summary of Results}

Using the contraction approach introduced above, we can reduce global sections to a quasipolynomial Euler characteristic on a $D$-minimal model.
\begin{theorem}
    \label{MAIN:contract}
    Let $X$ be a variety of Calabi--Yau type (Fano type). Then $D \mapsto h^0(X, \mathcal{O}_X(D))$ is a piecewise quasipolynomial on $\Cl(X)$ for big (respectively, effective) $D$, given by composing 
    \begin{enumerate}
        \item the pushforward to a $D$-minimal model $f : X \dashrightarrow Y$, on which $D$ is nef, with
        \item the Euler characteristic on $Y$. 
    \end{enumerate}
    The domains of quasipolynomiality are the Mori chambers, pulled back to $\Cl(X)$.
\end{theorem}
On the other hand, the subtraction method, formalized in \cref{alg}, reduces global sections to an Euler characteristic on an SQM after recursively subtracting components of the stable fixed locus, which involves only finitely many integer linear computations with the birational data of $X$ (i.e., the data determining the Mori chamber decomposition of the effective cone, formalized in \cref{def:birational_data}) and evaluations of Euler characteristics on SQMs of $X$.
\begin{theorem}
    \label{MAIN:subtract}
    Let $X$ be a variety of Calabi--Yau type. Then \cref{alg} computes $h^0(X, \mathcal{O}_X(D))$ for every $D \in \Bigc(X)$ with big Zariski movable subtraction (see \cref{def:zariski_movable_subtraction}). If $X$ is in particular Fano type, \cref{alg} computes $h^0(X, \mathcal{O}_X(D))$ for every $D \in \Eff(X)$.
\end{theorem}
We see that subtraction is often easier to apply than contraction --- one need only work with the SQMs of a variety, rather than the larger and more complicated set given by all of its birational contractions --- but naively doesn't yield a formula (only a recursive algorithm), and is less comprehensive for Calabi--Yau type varieties. The strongest result arises from combining subtraction and contraction.
\begin{theorem}
    \label{MAIN:together}
    Let $X$ be a variety of Calabi--Yau type (Fano type). Then \cref{alg} is asymptotically (in the sense of \cref{def:asymp}) a piecewise quasipolynomial function on $\Cl(X)$ for divisors in $\Bigc(X)$ (respectively, $\Eff(X)$) with domains given the Mori chambers, pulled back to $\Cl(X)$.
    Moreover, that asymptotic piecewise quasipolynomial function agrees with $D \mapsto h^0(X, \mathcal{O}_X(D))$ for all $D \in \Bigc(X)$ (respectively, $D \in \Eff(X)$).
\end{theorem}
In this result,  subtraction furnishes an explicit global section quasipolynomial on each Mori chamber which holds asymptotically; then contraction requires that on each entire Mori chamber, global sections must be quasipolynomial. In conjunction, the asymptotic quasipolynomial must hold everywhere (even where the algorithm naively failed or required extra iterations).

These results exhibit how birational geometric data determine formulae for global sections: the geography of models organizes the effective cone into Mori chambers, and \cref{MAIN:contract} entails that on any chamber global sections are determined by the Euler characteristic of a nef divisor on the associated $D$-minimal model. In this way, birational geometry imprints in a distinctive way on the formulae. Beyond the agreement between the domains of the formulae and of the $D$-MMP, we have that, for example, invariant directions in the formulae are generated by exceptional divisors. More generally, we see that collectively, the data of global sections of all effective divisors is equivalent to that of the global sections/Euler characteristics of all nef (and big, in the CY type case) divisors on all birational contractions.

Moreover, \cref{MAIN:subtract} and \cref{MAIN:together} illustrate how one doesn't need to construct every birational contraction --- it suffices to compute the nef cones and Euler characteristics of each SQM, and merely the prime exceptional divisors of their remaining birational contractions (see \cref{def:birational_data} and the surrounding discussion). In light of the previous paragraph, this means that this limited data actually determines a considerable amount about the complete collection of $D$-minimal models. For example, Euler characteristics $\chi(X,\mathcal{O}_X(D))$ on a variety $X$ encode important aspects of its intersection theory, characteristic classes, asymptotic divisorial invariants, lattices of divisor classes, and singularities. Our main theorems entail that these aspects can be determined for all birational contractions merely from the aforementioned data of the SQMs. In particular, a crucial role is played by how the Mori chambers intersect the lattice of divisor classes.

This can be taken one step further through a process we denote \emph{birational tomography}, which inverts our primary objective of ``cohomology from birational geometry.'' Any individual quasipolynomial is determined, in principle, by its evaluation at finitely many values --- consequently, one can make inferences about the birational geometry of a variety by performing a finite number of global section computations and numerically fitting the resulting data to a piecewise quasipolynomial. Such a process determines all of the rich information discussed in the previous paragraph --- the birational and Euler characteristic data of the SQMs we were previously treating as an input as well as the information encoded by the Euler characteristics of all birational contractions, hence ``birational geometry from cohomology.''
This is analogous in spirit to~\cite{Gendler:2022ztv}, where birational geometric information is extracted from finite enumerative geometry computations in the case of smooth Calabi--Yau threefolds.

From a computational perspective, the correspondence can therefore be exploited in whichever direction is more accessible. When the birational geometry is tractable, it may be computed directly: for Mori dream subvarieties of Mori dream spaces, Cox ring methods such as those of~\cite{herrera2024cox} can determine the relevant birational data, with Mori dream subvarieties of toric varieties providing a particularly concrete setting. Conversely, in geometries where individual cohomology computations are cheaper, one may instead use them as input for birational tomography; for example, line-bundle cohomology on toric hypersurfaces can be computed using packages such as \texttt{StringTorics}~\cite{stringtorics}, while complete intersections in products of projective spaces can be treated by Koszul-based methods implemented in \texttt{CIPro}~\cite{Anderson:2026eyl}. At present, these two computational approaches are most effective on substantially overlapping classes of explicitly presented varieties, but their relative cost can differ considerably in individual examples. Importantly, the tomography perspective is not restricted to Mori dream spaces: it can also probe Calabi--Yau type varieties with potentially infinitely many birational chambers, recovering individual regions without requiring a finite Cox ring computation or a complete global chamber decomposition.

Although our results concern Calabi--Yau type and Fano type varieties specifically, the underlying mechanisms of subtraction and contraction depend more generally on the existence of $D$-minimal models, sufficient control of stable base loci, and vanishing theorems. We therefore expect many aspects of the present framework to extend beyond the settings considered here. In particular, even when the $D$-MMP is only available --- or only practical --- on restricted subsets of the effective cone, the correspondence may still be exploited locally or for individual divisors, furnishing methods for computing global sections or extracting birational information even when a complete global description is unavailable.

\subsection{Organisation of the paper} 

In \cref{sec:prelim} we establish notation and review the ingredients required for our results, including aspects of birational geometry, the minimal model program, Mori dream spaces (especially Fano type varieties), Calabi--Yau type varieties, Zariski decomposition, and cohomological theorems. In \cref{sec:contraction} we develop our first approach, contraction, yielding the proof of \cref{MAIN:contract}. In \cref{sec:subtraction} we develop the complementary approach of subtraction and prove \cref{MAIN:subtract}, as well as the more powerful \cref{MAIN:together} which follows from contraction and subtraction together. In \cref{sec:non-big} we consider the marginal case of divisors that are nef but not big, explaining how passage to the base of an Iitaka fibration can furnish global section formulae. In \cref{sec:cy_hyp} we briefly comment on a result for Calabi--Yau hypersurfaces in toric varieties that is of practical utility, and in \cref{sec:birational_tomography} we further discuss the birational tomography approach. We conclude in \cref{sec:examples} with four illustrative examples.

\subsection{Acknowledgements}

We are grateful to Mike Stillman for many helpful conversations as well as for providing early access to the Macaulay2 \cite{M2} package StringTorics \cite{stringtorics}, which enabled line bundle cohomology computations for Calabi--Yau hypersurfaces in toric varieties. We thank the Erwin Schr\"odinger Institute's Unreasonable Effectiveness of Toric Geometry workshop (June-July 2026) for providing a productive research environment for the advancement of this project. We are also grateful to Bal\'azs Szendr\"oi for helpful comments on a draft of this paper. ES additionally thanks Sebastian Vander Ploeg Fallon, Chen Jiang, Daniel Halpern-Leistner, Liam McAllister, Jakob Moritz, and Rachel Webb for insightful conversations. 
The authors acknowledge the use of Claude Opus and Claude Fable for verification and correction of statements and computations, as well as for literature search, ideation and figure generation. All AI-assisted content was checked by the authors, who take full responsibility for the accuracy of the work.
ES is supported in part by NSF grant PHY-2309456, and this material is based upon work supported by the National Science Foundation (NSF) Graduate Research Fellowship under Grant No. 2139899; any opinions, findings, and conclusions or recommendations expressed in this material are those of the author(s) and do not necessarily reflect the views of the NSF. AC~is supported by the Royal Society grant DHF/R1/231142. AL is supported by the STFC consolidated grant ST/X000761/1.

\section{Preliminaries}\label{sec:prelim}

We collect here the basic notation and standard facts used in the paper. Here and throughout the remainder of the paper, all varieties $X$ are normal, $\mathbb{Q}$-factorial, and projective over~$\mathbb C$. Let $\mathrm{Cl}(X)$ denote the class group of linear equivalence classes of Weil divisors and $\Pic(X)$ the Picard group of linear equivalence classes of Cartier divisors. We write $N^1(X)$ for the
group of Cartier divisors modulo numerical equivalence.
For any abelian group $A$, we define $A_\mathbb{F} := A \otimes_{\mathbb Z} \mathbb F$ for $\mathbb{F} \in \{\mathbb{Q}, \mathbb{R}\}$. The cones of nef, movable, and effective divisor classes are denoted by
\begin{equation*}
    \Nef(X), \qquad \Mov(X), \qquad \Eff(X)\subset N^1(X)_{\mathbb R}.
\end{equation*}
We note that the latter two of these cones are neither open nor closed in general. In many contexts it is useful to consider their closures and/or interiors, but this largely will not be necessary for our purposes, with the exception of the big cone $\Bigc(X)$ (or cone of big divisors), which is the interior of $\Eff(X)$. 

For $D$ a divisor, we let $\mathcal O_X(D)$ denote the associated rank-one reflexive sheaf.
We also write $[D] \in \mathrm{Cl}(X)$ for its linear equivalence class, $[D]_\text{num} \in N^1(X)_\mathbb{R}$ for its numerical equivalence class, $|D|$ for its complete linear system, and $\Bdiv(D)$ for the closed set given by the union of the prime divisorial components of the stable base locus of $D$. If $D'$ is another divisor, we write linear equivalence and numerical equivalence as $D \sim D'$ and $D \equiv D'$, respectively. We have natural maps $\Cl(X) \xrightarrow{} \Pic(X)_\mathbb{Q}$ and $\Pic(X)_\mathbb{Q} \to N^1(X)_\mathbb{R}$, which we will implicitly employ throughout this work when we write, e.g., $D \in \Nef(X)$ (instead of $[D]_\text{num} \in \Nef(X)$). In particular, many functions we consider --- such as $h^0(X,-)$ and $\chi(X,-)$ --- will act on elements of $\Cl(X)$, but we will talk about properties of such functions on regions in $N^1(X)_\mathbb{R}$, by which we always mean the preimage of such regions in $\Cl(X)$.

We let $\nu_A(D)$ denote the degree of vanishing of $D$ along a prime divisor $A$ (e.g., if $D = \sum_i a_i D_i$ for $D_i$ prime divisors, $\nu_{D_i}(D) = a_i$). Furthermore, following \cite{Nakayama04}, for an effective divisor~$D$, we define
\begin{equation}
    \sigma_A(D) = \inf \{\nu_A(D') \; | \; D' \in |D|\}~.
\end{equation}
In \cite{Nakayama04} this is denoted $\sigma_A(D)_\mathbb{Z}$ --- we truncate the notation for brevity. We will also use the following terminology, which is convenient for stating several of our results.
\begin{definition}
    \label{def:asymp}
    We say that divisors in a cone $\mathcal{C} \subset N^1(X)_\mathbb{R}$ asymptotically have a property $P$ if for every Weil divisor $D \in \mathcal{C}$, there exists $m_0 \in \mathbb{Z}_{\geq 0}$ such that $mD$ satisfies $P$ for every $m \geq m_0$. The integer $m_0$ will in general be $D$-dependent.
\end{definition}
The following elementary observation underlies much of the present work: removing fixed components does not change the space of global sections.
\begin{lemma}
\label{lem:exceptional_preserve}
Let $X$ be an irreducible normal variety with Weil divisor $D$. Assume that a prime divisor $E$ is a fixed component of the linear system $|D|$. Then
\begin{equation*}
H^0(X,\mathcal O_X(D)) \;=\; H^0(X,\mathcal O_X(D-E)).
\end{equation*}
\end{lemma}
\begin{proof}
Clearly $H^0(X,\mathcal O_X(D-E))\subset H^0(X,\mathcal O_X(D))$. Conversely, every section $s \in H^0(X,\mathcal O_X(D))$ vanishes along $E$, since $E$ is fixed in $|D|$, so $\mathrm{div}(s)+D-E\ge 0$, hence $s\in H^0(X,\mathcal O_X(D-E))$.
\end{proof}
The only notion from the theory of singularities that we will need is that of a klt pair, which is required to state the definition of the varieties of interest (Calabi--Yau type and Fano type) as well as the Kawamata--Viehweg vanishing theorem. 
\begin{definition}
    Let $X$ be a normal variety with \(\Delta = \sum_i d_i \Delta_i \ge 0\) an effective
    $\mathbb{Q}$--divisor such that \(K_X + \Delta\) is $\mathbb{Q}$--Cartier, and let
    \(f \colon Y \to X\) be a log resolution of the pair \((X,\Delta)\); that is, a proper
    birational morphism with \(Y\) smooth such that the union of the exceptional locus
    and the support of the strict transform \(f^{-1}_*\Delta\) is a
    divisor with simple normal crossings. Then there are uniquely determined rational
    numbers \(a(E;X,\Delta)\), the \emph{discrepancies}, such that
    \begin{equation}
      K_Y + f^{-1}_*\Delta \;\sim_{\mathbb{Q}}\; f^*(K_X+\Delta)
          + \sum_E a(E;X,\Delta)\, E,
    \end{equation}
    where the sum runs over all \(f\)--exceptional prime divisors \(E\) on \(Y\).
    Each \(a(E;X,\Delta)\) depends only on \(E\) as a divisor over \(X\), not on the
    choice of \(f\). The pair $(X,\Delta)$ is \emph{Kawamata log terminal (klt)} if
    \(a(E;X,\Delta) > -1\) for every prime divisor \(E\) over \(X\) and every
    coefficient satisfies \(0 \le d_i < 1\).
\end{definition}

The methods developed in this paper also involve passing between birational models. If $f \colon X\dashrightarrow Y$ is a birational map, $D$ is a divisor on $X$, and $D'$ is a divisor on $Y$, we write $f_*D$ for the strict transform of $D$ on $Y$ (which we will also refer to as the pushforward) and $f^*D'$ for the pullback of $D'$ in $X$, $\mathbb{Q}$-linearly extended from Cartier divisors to Weil divisors. We also recall the following standard terminology.
\begin{definition}
    A birational map $f \colon X\dashrightarrow Y$ is a \emph{contraction} if its inverse contracts no divisors.
\end{definition}
\begin{definition}
    A birational map $\varphi \colon X\dashrightarrow Y$ with $\mathbb{Q}$-factorial image is a \emph{small $\mathbb Q$-factorial modification} (SQM) if it is an isomorphism in codimension one. Equivalently, it is an SQM if it is a contraction and its inverse is a contraction.
\end{definition}
We say a birational map is a divisorial contraction if it is a contraction but not an SQM.
\begin{definition}
    \label{def:D-nonpositive}
    A birational map $f \colon X \dashrightarrow Y$ is \emph{$D$-nonpositive} if on a common resolution $(\alpha, \beta) : W \to X \times Y$, $\alpha^* D - \beta^* f_* D$ is effective and $\beta$-exceptional.
\end{definition}
We will occasionally let contraction or SQM denote both the birational map itself as well as the target. Finally, for some fixed variety $X$, distinct birational contractions can yield the same variety. To distinguish between such maps when some reference $X$ is fixed --- e.g., to distinguish distinct Mori chambers --- we use the notion of \emph{marked} contractions (or marked varieties), which store the data of both the image $Y$ as well as a particular choice of birational map $f : X \dashrightarrow Y$. This marked contraction is said to equal $g : X \dashrightarrow Z$ if $g \circ f^{-1}$ is an isomorphism.

Finally, we will use the negativity lemma a few times.
\begin{lemma}[Negativity lemma, {\cite[Lem.~3.39]{KM}}]
    \label{lem:negativity}
    Let $f \colon X \to Y$ be a proper birational morphism between normal varieties and let $B$ be a $\mathbb{Q}$-Cartier $\mathbb{Q}$-divisor on $X$ such that $-B$ is nef over $Y$, or $f$-nef (i.e., considering only curves contracted by $f$). Then $B$ is effective if and only if $h_* B$ is effective. In particular, an $h$-exceptional divisor which is numerically trivial over $Y$ vanishes.
\end{lemma}

\subsection{Minimal Model Program}

Given an effective divisor \(D\) on a variety \(X\), one may attempt to construct a birational model on which the pushforward of $D$ is nef and preserves as much of the structure of $D$ as possible.

\begin{definition}[$D$-minimal model]
    Given a variety $X$ with divisor $D$, a $D$-minimal model is a variety $Y$ together with a $D$-nonpositive birational contraction $f : X \dashrightarrow Y$ such that $f_* D$ is a nef divisor on $Y$.
\end{definition}
The special case \(D=K_X\) corresponds to the standard minimal model program.
An $(X,D)$-log minimal model is a $(K_X + D)$-minimal model (e.g., for a discussion in the Fano-type case, see \cite[Cor.~2.7]{prokhorov2009towards}).
In favourable situations, $D$-minimal models can be constructed by a sequence of extremal contractions 
(SQMs and divisorial contractions) \[X = X_0 \dashrightarrow X_1 \dashrightarrow \dots \dashrightarrow X_n = Y~.\] This procedure is often referred to as the $D$-minimal model program ($D$-MMP). Each step is determined by a \(D\)-negative extremal ray of the Mori cone $R \subset \overline{NE}(X)$, that is a ray satisfying $D \cdot R < 0$. For a thorough review, see \cite{KM}.

The existence of $D$-minimal models
is not assured in general. By \cite[Th. 1.2]{BCHM10}, a projective klt pair $(X,\Delta)$ admits a log terminal model whenever $\Delta$ is big and $K_X+\Delta$ is pseudoeffective. We restrict our attention to settings in which substantially stronger results are available: Calabi--Yau type varieties, where every big divisor admits a $D$-minimal model (\cref{prop:cy_mmp}) and Mori dream spaces (of which Fano type varieties are a special case), where every effective divisor admits a $D$-minimal model (\cref{prop:MDS_chambers}).

In the settings considered below, birational contractions determine regions of the effective cone on which the resulting $D$-minimal model is constant. These are the Mori chambers.
\begin{definition}
    \label{def:mori_chamber}
    The Mori chamber $\mathcal{C} \subset \Eff(X)$ associated to a birational contraction $f : X \dashrightarrow Y$ with $\mathbb{Q}$-factorial image and prime exceptional divisors $E_1, \dots, E_r$ is the cone
    \begin{equation}
        \mathcal{C} = f^* (\Nef(Y) \cap \Eff(Y)) + \mathrm{ex}(f)~,
    \end{equation}
    where $\mathrm{ex}(f)$ is the simplicial cone generated by the exceptional divisors $E_1, \ldots, E_r$.
\end{definition}
In particular, in the notation of the definition $\mathrm{rank} \, N^1(X) = \mathrm{rank} \, N^1(Y) + r$ holds in our setting. If $D$ belongs to the Mori chamber of a birational map $f$, then the codomain of $f$ is a $D$-minimal model, so divisors in a common Mori chamber share a $D$-minimal model. 

The through-line uniting Mori dream spaces and Calabi--Yau varieties will be that the existence of $D$-minimal models induces a decomposition of the effective cone (or big cone, in the Calabi--Yau case) into Mori chambers, as we will discuss in \cref{sec:mori} and \cref{sec:cy}. The Mori chamber decomposition will organize both of our approaches for producing algorithms and formulae for global sections. Certainly for the contraction method, the Mori chambers prescribe the birational model on which a vanishing theorem should be applied and an Euler characteristic should be evaluated. Moreover, as we will see in \cref{sec:zariski}, the decomposition also encodes the stable base loci of effective divisors in a manner that enables efficient subtraction of fixed components. Following \cite{shokurov19963, kaloghiros2016finite}, we refer to the Mori chamber decomposition of the big cone/effective cone as the geography of models on \(X\).

The contraction method is enabled by the fact that passing to a $D$-minimal model preserves the global sections of $D$.
\begin{proposition}[{\negthinspace\negthinspace\negthinspace\cite[Rem. 2.4]{kaloghiros2016finite}}]
    \label{prop:mmp_preserves_sections}
    Let $X$ be a normal $\mathbb{Q}$-factorial projective variety with divisor $D$. If a $D$-minimal model $f : X \dashrightarrow Y$ exists, then
    \begin{equation*}
        h^0(X, \mathcal{O}_X(D)) = h^0(Y, \mathcal{O}_Y(f_* D)).
    \end{equation*}
\end{proposition}

\subsection{Quasipolynomials and Hilbert Series}\label{sec:mori1}

Quasipolynomials will play a foundational role in this work. There are two primary reasons for this: first, $\mathbb{Q}$-factoriality entails that Euler characteristics are quasipolynomial, as we recall in \cref{th:snapper} and \cref{cor:euler_char_is_quasipoly}, and reduction to Euler characteristics will be a primary tool of our methods. Second, as we will review in this section, the Hilbert function of a module over a multigraded ring is also (piecewise) quasipolynomial. If one considers in particular Cox rings of Mori dream spaces, this straightforwardly allows one to conclude that global sections on Mori dream spaces admit piecewise quasipolynomial formulae, as we state in \cref{prop:cox_ring_quasipoly}. However, in this paper we will pursue more computable and informative formulae.
\begin{definition}
    \label{def:quasipoly}
    A \emph{quasipolynomial} is a function $q:G\to\mathbb Z$, where $G\cong\mathbb{Z}^n\oplus H$ is a finitely generated Abelian group (with $H$ finite), for which there exists a finite-index sublattice $\Lambda \subset \mathbb{Z}^n$ and polynomials $q_{i,h}$ indexed by cosets $t_i + \Lambda$ of $\mathbb{Z}^n/\Lambda$ and $h \in H$ such that $q(x) = q_{i,h}(t)$ whenever $x = (t, h)$ for $t \in t_i + \Lambda$. If each $q_{i,h}$ is linear, we say that $q$~is \emph{quasilinear}.
    Similarly, for $A$ a finitely generated Abelian group, we call $f : G \to A$ \emph{quasilinear} if, for $\Lambda$ and the indexing as above, there are homomorphisms $\phi_{i,h} : \Lambda \to A$ and constants $c_{i,h} \in A$ such that $f(x) = \phi_{i,h}(t - t_i) + c_{i,h}$ whenever $x = (t,h)$ for $t \in t_i + \Lambda$.
\end{definition}
Now we turn our attention to modules over finitely generated multigraded rings. For details, we refer the reader to \S8 of \cite{miller2005combinatorial} --- we will just summarize the basic result here. Recall that a polynomial ring $S=\mathbb C[x_1,\ldots,x_n]$ is multigraded by a finitely generated Abelian group $G$ if each variable $x_i$ is assigned a degree $\deg(x_i)\in G$. Equivalently, this determines a group homomorphism $\deg:\mathbb Z^n\to G$, where $\mathbb Z^n$ is identified with the lattice of monomial exponents. For a monomial $x^a=\prod_i x_i^{a_i}$ with
$a=(a_1,\ldots,a_n)\in\mathbb N^n$, its degree is $\deg(x^a)=\sum_i a_i\deg(x_i)$.
The multigrading is positive if the only monomials of degree $0$ are the constants.
\begin{proposition}
\label{prop:gen_func_piecewise_quasi}
    Let $S$ be a positively multigraded polynomial ring with free grading group $G$ and $M$ a finitely generated multigraded $S$-module. Then the Hilbert series of $M$ is rational. Consequently, the Hilbert function of $M$ is piecewise quasipolynomial.
\end{proposition}
\begin{proof}
    Let $Q \subset G$ denote the image of $\mathbb{N}^n \subset \mathbb{Z}^n$ under $\deg$ and write $g_i := \deg(x_i)$ for $i = 1, \dots, n$. The multigraded Hilbert series of $M$ is the formal sum
    \begin{equation}
        HS(M;t) = \sum_{g \in G} \dim M_g \, t^g,
    \end{equation}
    and by \cite[Th. 8.20]{miller2005combinatorial} it lies in the ring $\mathbb{Z}[[Q]][G]$ of series supported on finitely many translates of $Q$, such that there is a unique Laurent polynomial $K(M;t) \in \mathbb{Z}[G]$ satisfying
    \begin{equation}
        \label{eq:hilbert_series_rational}
        HS(M;t) = K(M;t) \cdot \frac{1}{\prod_{i=1}^n (1 - t^{g_i})}~.
    \end{equation}
    This proves the first assertion. For the second, write $q : G \to \mathbb{Z}$, $q(g) = \dim M_g$, for the Hilbert function of $M$. By \cite[Lem. 8.16]{miller2005combinatorial} the second factor of \cref{eq:hilbert_series_rational} is a power series in $\mathbb{Z}[[Q]]$, namely the Hilbert series of $S$ itself, so its Hilbert function
    \begin{equation}
        q_S(g) = |\{a \in \mathbb{N}^n \; | \; \textstyle\sum_i a_i g_i = g\}|
    \end{equation}
    is a vector partition function, and is therefore piecewise quasipolynomial by \cite[Th. 1]{sturmfels1995vector}. Writing $K(M;t) = \sum_j c_j t^{h_j}$ with $c_j \in \mathbb{Z}$ and $h_j \in G$, we obtain
    \begin{equation}
        q(g) = \sum_j c_j \, q_S(g - h_j)~.
    \end{equation}
    A finite $\mathbb{Z}$-linear combination of piecewise quasipolynomial functions is piecewise quasipolynomial: each term contributes a shifted copy of the domains of $q'$, and a common refinement of these finitely many collections serves as the collection of domains for $q$.
\end{proof}
\begin{corollary}
    \label{cor:gen_func_torsion}
    The Hilbert function of $M$ is piecewise quasipolynomial for an arbitrary finitely generated grading group $G$.
\end{corollary}
\begin{proof}
    Write $G \cong \mathbb{Z}^d \oplus H$ with $H$ finite of order $m$ and set $R := \mathbb{C}[x_1^m, \dots, x_n^m] \subset S$. Applying \cref{prop:gen_func_piecewise_quasi} to each of the $m$ $R$-modules $M^{(\eta)} := \bigoplus_{\deg(g) \in \mathbb{Z}^d \oplus \{\eta\}} M_g$ gives polynomials indexed by the cosets of a finite-index sublattice of $\mathbb{Z}^d$ and by $\eta \in H$ --- precisely the data of \cref{def:quasipoly}.
\end{proof}
\begin{remark}
    Sturmfels' result from~\cite{sturmfels1995vector} that underpinned \cref{prop:gen_func_piecewise_quasi} can also be viewed geometrically as the statement that, on a complete toric variety, the function $D \mapsto h^0(X,\mathcal O_X(D))$ (equivalently, the Hilbert function of the Cox ring of that toric variety) is piecewise quasipolynomial.
\end{remark}

\subsection{Mori Dream Spaces \& Fano type varieties} \label{sec:mori}

Instead of proceeding directly to the most general class of varieties we will consider --- those of Calabi--Yau type --- we will begin with the simpler case of Fano type varieties. But the relevant aspects of the $D$-MMP and Mori chamber holds more generally for all Mori dream spaces, hence we lose nothing by focusing on such varieties in this section. Mori dream spaces were originally introduced by Hu and Keel~\cite{HuKeel}, who we follow.
\begin{definition}
    A normal projective $\mathbb Q$-factorial variety $X$ is a \emph{Mori dream space} if $\Pic(X)_{\mathbb Q}\cong N^1(X)_{\mathbb Q}$, the nef cone $\Nef(X)$ is rational polyhedral and generated by finitely many semiample divisors, and there exist finitely many SQMs $\varphi_i\colon X\dashrightarrow X_i$ such that each $X_i$ has the same properties and
    \begin{equation*}
        \Mov(X)=\bigcup_i \varphi_i^*\Nef(X_i).
    \end{equation*}
\end{definition}
An equivalent characterization of Mori dream spaces is that their Cox rings are finitely generated~\cite{HuKeel}. Roughly speaking, the Cox ring of a variety $X$ is the $\mathrm{Cl}(X)$-graded ring obtained by assembling the spaces $H^0(X,\mathcal O_X(D))$ for all divisor classes $[D]\in\mathrm{Cl}(X)$ into a single algebra. Assuming, for simplicity, that ${\rm Cl}(X)$ is torsion free, this is 
\begin{equation*}
    {\rm Cox}(X)~=\bigoplus_{[D]\in {\rm Cl}(X)} 
    H^0(X,\mathcal O_X(D)),
\end{equation*}
with multiplication induced by multiplication of rational functions. Finite generation of the Cox ring immediately imposes strong algebraic structure on the spaces of global sections and, in particular, implies piecewise quasipolynomiality.

As mentioned in the previous section, applying \cref{cor:gen_func_torsion} to the Cox ring immediately yields the following.
\begin{proposition}
    \label{prop:cox_ring_quasipoly}
    Let $X$ be a Mori dream space. Then the function $D \mapsto h^0(X,\mathcal O_X(D))$ on the class group ${\rm Cl}(X)$ is piecewise quasipolynomial.
\end{proposition}
\begin{proof}
    The map $D \mapsto h^0(X, \mathcal{O}_X(D))$ is the Hilbert function of the Cox ring of $X$, so it suffices to apply \cref{cor:gen_func_torsion}.
\end{proof}
Thus the existence of piecewise quasipolynomial formulae on Mori dream spaces is already a consequence of finite generation of the Cox ring. However, this perspective provides little information about the geometry of the corresponding domains and is often impractical computationally, since Cox rings and their Hilbert functions can be difficult to determine explicitly. 
Our aim is not to prove the existence of piecewise quasipolynomial
formulae, which is already implicit in finite generation of the Cox
ring, but rather to identify their domains with birational chambers and
to compute them directly from birational data.

Mori dream spaces are distinct in that the $D$-MMP can be run for every effective divisor, and there are finitely many $D$-minimal models. This means the Mori chamber decomposition is a finite chamber decomposition of the entire effective cone.
\begin{proposition}[{\negthinspace\negthinspace\negthinspace\cite[Prop. 1.11]{HuKeel}, \cite[Def.-Prop. 2.9]{Okawa}}]\label{prop:MDS_chambers}
    Let $X$ be a Mori dream space. Then $\Eff(X)$ admits a finite rational polyhedral decomposition into a fan $\Fan(X)$, whose maximal cones are Mori chambers $\mathcal{C}$ associated to birational contractions $f_{\mathcal C} : X \dashrightarrow X_{\mathcal{C}}$ for $X_{\mathcal{C}}$ also a Mori dream space.
    \begin{equation}
        \label{eq:mori_chamber}
        \mathcal C = f_{\mathcal C}^* \Nef(X_{\mathcal{C}}) + \mathrm{ex}(f_{\mathcal C})
    \end{equation}
    In particular, $D$-minimal models exist for every $D \in \Eff(X)$.
\end{proposition}
\begin{remark}
    In comparing with \cref{def:mori_chamber}, we have that $f_{\mathcal C}^* \Nef(X_{\mathcal{C}}) = f_{\mathcal C}^* (\Nef(X_{\mathcal{C}}) \cap \Eff(X_{\mathcal{C}}))$ because $X_\mathcal{C}$ is also a Mori dream space and thus its nef divisors are semiample and in particular have some effective multiple, so $\Nef(X_{\mathcal{C}}) \subset \Eff(X_\mathcal{C})$.
\end{remark}
\begin{proposition}[{\negthinspace\negthinspace\negthinspace\cite[Lem. 1.7 and Prop. 1.11]{HuKeel}}]
    \label{prop:MDS_bir_contract_factor}
    Let $X$ be a Mori dream space and let $f : X \dashrightarrow Z$ be a birational contraction. Then $f$ factors as $f = h \circ g$ with $g : X \dashrightarrow Y$ an SQM and $h : Y \to Z$ a morphism.
\end{proposition}

It now suffices to comment that Fano type varieties are a special case of Mori dream spaces, meaning they enjoy a finite Mori chamber decomposition of the effective cone. Moreover, they are amenable to particular applications of Kawamata--Viehweg vanishing that we require for global section formulae. 
\begin{proposition}[{\negthinspace\negthinspace\cite[Cor. 1.3.1]{BCHM10}}]
    \label{prop:fano_type_is_mds}
    Varieties of Fano type are Mori dream spaces.
\end{proposition}
\begin{proposition}[{\negthinspace\negthinspace\cite[Lem.~2.8]{prokhorov2009towards}}]
    \label{prop:dmmp_of_fano_type_is_fano_type}
    $D$-minimal models of Fano type varieties are Fano type.
\end{proposition}
\begin{proof}
    It suffices to apply \cite[Lem.~2.8]{prokhorov2009towards} to the factorization ensured by \cref{prop:MDS_bir_contract_factor}.
\end{proof}
As mentioned in the introduction, the Fano type case is a special case of the Calabi--Yau type case (see also \cite[\S1.4]{moraga2023coregularity}).
\begin{proposition}[{\negthinspace\negthinspace\cite[Lem.-Def.~2.6(iii)]{prokhorov2009towards}}]
    \label{prop:fano_type_is_cy_type}
    Fano type varieties are Calabi--Yau type
\end{proposition}
We also note that toric geometry furnishes many examples of Fano type varieties.
\begin{proposition}[{\negthinspace\negthinspace\cite[Ex. 11.4.26]{cls}}]
    \label{prop:toric_is_fano_type}
    Simplicial projective toric varieties are Fano type
\end{proposition}

\subsection{Calabi--Yau Type Varieties}\label{sec:cy}

Calabi--Yau varieties --- Calabi--Yau type varieties $X$ with $K_X \sim 0$ --- play a distinguished role in compactifications of superstring theory. Consequently, advancing understanding of their topology is of the utmost importance. Fortunately for our purposes, not just Calabi--Yau varieties but actually the more general class of varieties of Calabi--Yau type have excellent birational geometry. This is because the Calabi--Yau type condition is preserved by the $D$-MMP and moreover relates the $D$-MMP to a log MMP, for which there are quite general results (namely those of \cite{BCHM10}) .

First, though, we must state an analog of \cref{prop:MDS_bir_contract_factor} for Calabi--Yau type varieties.
\begin{prop}[{\negthinspace\negthinspace\negthinspace\cite[Lem. 4.10]{gachet2024effective}}]
    \label{prop:CY_type_bir_contract_factor}
    Let $X$ be a Calabi--Yau type variety and let $f : X \dashrightarrow Z$ be a birational contraction. Then $f$ factors as $f = h \circ g$ with $g : X \dashrightarrow Y$ an SQM and $h : Y \to Z$ a morphism.
\end{prop}
This enables our desired birational geometry results.
\begin{proposition}
    \label{prop:dmmp_preserves_cy}
    If $f : X \dashrightarrow Y$ is a birational contraction for $X$ Calabi--Yau type, with $\Delta$ from \cref{def:CY_fano_type}, and $Y$ is $\mathbb{Q}$-factorial, then $Y$ is Calabi--Yau type. In particular, $D$-minimal models of $X$ are Calabi--Yau type.
\end{proposition}
\begin{proof}
    For any birational contraction, the pushforward preserves numerical equivalence, and by the standard sequence $\sum_i \mathbb{Z} E_i \to \Cl(X) \to \Cl(Y) \to 0$ (a consequence of, e.g., \cite[Th. 4.0.20]{cls}) $\Cl(Y)$ is a quotient of $\Cl(X)$ meaning $\Cl(Y)$ is finitely generated and in particular features trivial $\Pic^0$. It therefore suffices to show that $(Y, f_* \Delta)$ is klt, and also crepant --- i.e., on a resolution $(\alpha, \beta) : W \to X \times Y$, $\alpha^*(K_X + \Delta) = \beta^* (K_Y + f_* \Delta)$. By \cref{prop:CY_type_bir_contract_factor}, $f$ factors into the composition of an SQM and a morphism, so it suffices to handle these separately. The result holds for SQMs by \cite[Prop. 3.51]{KM}, using the nefness of $K_X + \Delta$ and its pushforward, and holds for morphisms by \cite[Lem. 2.14]{bernasconi2022semiampleness}.
\end{proof}
\begin{proposition}
    \label{prop:cy_mmp}
    Let $X$ be a Calabi--Yau type variety with divisor $D \in \Eff(X)$. Then either of the following conditions implies the existence of a $D$-minimal model.
        \begin{enumerate}
            \item $D \in \Bigc(X)$ 
            \item $\dim X \leq 3$
        \end{enumerate}
\end{proposition}
\begin{proof}
    Let $\Delta$ be as in \cref{def:CY_fano_type} and choose $\varepsilon > 0$ such that $(X, \Delta + \varepsilon D)$ is klt. Existence of a log minimal model for this pair is ensured for $D \in \Bigc(X)$ by \cite[Th.~1.2]{BCHM10} and for $\dim X \leq 3$ by the results of \cite{shokurov19963} (see, e.g., the introduction of \cite{birkar2010existence} for further discussion). It then suffices to note that a log minimal model for the klt pair $(X, \Delta + \varepsilon D)$ is a $D$-minimal model.
\end{proof}
We will elect to remain general in dimension and restrict to the big cone, as Kawamata--Viehweg becomes difficult to apply for non-big divisors on Calabi--Yau type varieties (but see \cref{sec:non-big}). The existence of $D$-minimal models for big divisors entails a Mori chamber decomposition of the big cone, as was elucidated in \cite{gachet2024effective}, which can be understood as building on \cite[Th. 2.3]{kawamata1997cone}.
\begin{proposition}[{\negthinspace\negthinspace\negthinspace\cite[Prop. 4.11, Lem. 4.2]{gachet2024effective}}]
    \label{prop:cy_preq}
    Let $X$ be a Calabi--Yau type variety.
    \begin{enumerate}
        \item Let $I$ index marked SQMs $\varphi_i : X \dashrightarrow X_i$ of $X$. Then,
        \begin{equation}
            \Mov(X) \cap \Bigc(X) \subset \bigcup_{i \in I} \varphi_i^* (\Nef(X_i) \cap \Eff(X_i))~.
        \end{equation}
        \item Let $J$ index marked $D$-minimal models $f_j : X \dashrightarrow X_j$ of $X$. Then,
        \begin{equation}
            \Bigc(X) \subset \bigcup_{j \in J} \mathcal{C}_j \qquad\qquad \mathcal{C}_j = f_j^* (\Nef(X_j) \cap \Eff(X_j)) + \mathrm{ex}(f_j)
        \end{equation}
        That is, divisors in $\mathcal{C}_j$ have $D$-minimal model $X_j$.
    \end{enumerate}
    That is, $\Mov(X) \cap \Bigc(X)$ and $\Bigc(X)$ are contained in unions of pairwise disjoint Mori chambers of SQMs and birational contractions, respectively.
\end{proposition}
\begin{remark}
    The first statement in \cref{prop:cy_preq} is stated in \cite{gachet2024effective} for the interior of the moving cone, and classically for Calabi--Yau varieties rather than pairs in \cite[Th. 2.3]{kawamata1997cone}, \cite[Props. 4.6 and 4.7]{wang2022remarks}. It holds on the generally larger cone $\Mov(X) \cap \Bigc(X)$, in every dimension, because for big divisors the minimal model required by \cite[Lem. 4.4]{gachet2024effective} exists unconditionally by \cref{prop:cy_mmp}.
\end{remark}
In spite of the similarities of this result with \cref{prop:MDS_chambers}, it is worth emphasizing that while the Mori chamber decomposition for Mori dream spaces is finite and consists of finitely generated cones, on Calabi--Yau type varieties the Mori chamber decomposition can consist of countably many chambers and the chambers can be generated by countably many divisors. 
\begin{remark}
    \label{rem:KM}
    The Kawamata--Morrison cone conjecture \cite{morrison1993compactifications, kawamata1997cone} (extended from the movable and nef cones to the effective cone in \cite{gachet2024effective}) can roughly be understood as the claim that the Mori chamber decomposition of the effective cone of a Calabi--Yau type variety is acted upon by its birational automorphism group such that any fundamental domain consists of finitely many chambers, each generated by finitely many divisors. In this sense, the difference between Calabi--Yau type varieties and Mori dream spaces is conjecturally a group action on Mori chambers. However, we will not pursue or discuss this further, as our results do not depend on it --- it suffices for us that we can subtract fixed components from a divisor or perform a $D$-MMP on a divisor in finitely many steps, which is true even in the presence of infinitely many chambers, including if some are infinitely generated.
\end{remark}
We conclude this section by recalling a few specific results that hold specifically for Calabi--Yau varieties. First, a criterion for low-dimensional smooth Calabi--Yau varieties to be Mori dream spaces.
\begin{proposition}[{\negthinspace\negthinspace\negthinspace\cite[Cor. 4.5]{McKernan2010}}]
    \label{prop:when_cy_is_mds}
    Smooth Calabi--Yau varieties $X$ of dimension $\leq 3$ are Mori dream spaces if and only if they have rational polyhedral effective cones.
\end{proposition}
Second, the following result is convenient for our purposes as smoothness ensures the applicability of Hirzebruch--Riemann--Roch for convenient evaluation of the Euler characteristic.
\begin{theorem}[{\negthinspace\negthinspace\negthinspace\cite[Th. 2.4]{kollar1989flops}}]
    \label{th:flop_of_smooth_cy3}
    If $X$ is a smooth Calabi--Yau threefold, so are its SQMs.
\end{theorem}

\subsection{Zariski Decomposition} \label{sec:zariski}

Zariski decomposition is a general concept which takes different forms in different contexts but broadly refers to a splitting of a divisor (or its class) into two pieces: a ``positive'' or ``movable'' component and a ``negative'' or ``fixed'' component, in direct analogy to the decomposition of a linear system into its fixed and movable components \cite{Nakayama04}. Zariski originally developed the notion for surfaces \cite{zariski} but the decomposition is more challenging to implement in higher dimensions. However, for any collection of divisors which admit a Mori chamber decomposition, there is a natural notion of Zariski decomposition. This is well understood for Mori dream spaces \cite{HuKeel, Okawa}, and naturally extends to Calabi--Yau type varieties, thanks to the results of \cite{gachet2024effective} regarding Mori chamber decompositions, as stated in \cref{prop:cy_preq}.

Consider an effective $\mathbb{Q}$-divisor $D$ on a variety $X$ belonging to a Mori chamber $\mathcal{C}$ associated to some birational contraction $f : X \dashrightarrow Y$ with $r$ prime exceptional divisors $E^{\mathcal C}_1, \dots, E^{\mathcal C}_{r}$. 
\begin{equation}
    \mathcal{C} = f^* (\Nef(Y) \cap \Eff(Y)) + \mathrm{ex}(f)~.
\end{equation}
Under our standing hypotheses $\Cl(X)_{\mathbb{R}} \cong N^1(X)_{\mathbb{R}}$, and there is a canonical decomposition of $[D]$ induced by the above expression of $\mathcal{C}$. 
\begin{equation}
    [D] = [P_\mathcal{C}(D)] + [N_\mathcal{C}(D)]~, \qquad [P_\mathcal{C}(D)] \in f^* (\Nef(Y) \cap \Eff(Y)), \;[N_\mathcal{C}(D)] \in \mathrm{ex}(f)~.
\end{equation}
Since $\mathrm{ex}(f)$ is the simplicial cone generated by the classes of the exceptional divisors $E_i^{\mathcal C}$, every class $[N_\mathcal{C}(D)] \in \mathrm{ex}(f_{\mathcal C})$ admits a unique expression
\begin{equation}
    [N_\mathcal{C}(D)] = \sum_{i=1}^{r} \lambda^{\mathcal C}_i(D) [E^{\mathcal C}_i]~.  
\end{equation}
Moreover, since each $[E^{\mathcal C}_i]$ has a unique effective representative $E^{\mathcal C}_i$, the decomposition of classes lifts uniquely to a decomposition of divisors, enabling the following definition.
\begin{definition}
    \label{def:zariski_decomp}
    In the setup described above, let $D \in \mathcal C$. The following is called the \emph{Zariski decomposition} of $D$: 
    \begin{equation*}
        D = P_\mathcal{C}(D) + N_\mathcal{C}(D) \qquad \left( N_\mathcal{C}(D) = \sum_{i=1}^{r} \lambda^{\mathcal C}_i(D) E^{\mathcal C}_i, \,\, P_\mathcal{C}(D) = D - N_\mathcal{C}(D) \right) ~.   
    \end{equation*}
    The divisors $P_\mathcal{C}(D)$ and $N_\mathcal{C}(D)$ are called the \emph{positive} (or movable) and \emph{negative} (or fixed) parts, respectively.
\end{definition}
Note that the Zariski decomposition is determined by the functions $\lambda^{\mathcal C}_i$ mapping divisors $D \in \mathcal{C}$ to $\mathbb{R}$. This function is linear because exceptional divisors are not torsion and have linearly independent classes.
\begin{corollary}
    \label{cor:zariski_decomp_lin_and_mov}
    $P_\mathcal{C}$, $N_\mathcal{C}$, and $\lambda^{\mathcal C}_i$ extend to linear functions on $\Cl(X)$.
\end{corollary}
Additionally, the support of $N_\mathcal{C}(D)$ encodes information about both the $D$-MMP as well as the stable base locus of $D$. In particular, the prime components of that support are both the prime exceptional divisors of $f$ and the prime divisors in the stable base locus of $D$. The former is true essentially by construction, but the latter is more non-trivial: in the Mori Dream Space case, it is similar to \cite[Prop 1.11 (5)]{HuKeel} and \cite[Prop. 2.15]{Okawa}.
\begin{proposition}
    \label{prop:zariski}
    Let $X$ be a Calabi--Yau type variety (Mori dream space) with a big (effective) divisor $D$ belonging to the Mori chamber $\mathcal{C} = f^* (\Nef(Y) \cap \Eff(Y)) + \mathrm{ex}(f)$ associated to a birational contraction $f : X \dashrightarrow Y$ with prime exceptional divisors $E_i$. Then,
    \begin{enumerate}
        \item $P_\mathcal{C}(D)$ is big (effective).
        \item $\Bdiv(D) = \mathrm{supp} \, N_\mathcal{C}(D)$.
        \item We have the following inequality.
        \begin{equation}
            \label{eq:bound_on_asymptotic_vanishing}
            \sigma_{E_i}(D) \geq \lceil \lambda^{\mathcal C}_i(D) \rceil
        \end{equation}
        In particular $\lceil \lambda^{\mathcal C}_i(D) \rceil > 0$ if and only if $E_i$ belongs to the stable base locus of $D$.
    \end{enumerate}
\end{proposition}
\begin{proof}
    First, due to the factorization of $f$ into a composition of an SQM and morphism (\cref{prop:CY_type_bir_contract_factor} for Calabi--Yau type varieties and \cref{prop:MDS_bir_contract_factor} for Mori dream spaces), we can freely perform an SQM to render $f$ a morphism and $P_\mathcal{C}(D)$ nef, while preserving the linear system and global sections of $D$. This also preserves either the Calabi--Yau type or Mori dream space property, by \cref{prop:dmmp_preserves_cy} and \cref{prop:MDS_chambers}, respectively. 
    
    For the first item, in the Calabi--Yau type (Mori dream space) case, $P_\mathcal{C}(D)$ is big (effective) due to the bigness (effectivity) of $D$, as $mP_\mathcal{C}(D)$ is a subtraction of $mD$ for all positive integers $m$ rendering both divisors integral, so the global sections of these families agree.
    
    From here, we show the third item of the proposition. Let $D' \in |D|$ and decompose 
    \begin{equation}
        D' = P + \sum_i \nu_{E_i}(D') E_i
    \end{equation}
    Then if we define the divisor
    \begin{equation}
        \Theta := N_\mathcal{C}(D) - \sum_i \nu_{E_i}(D') E_i = \sum_i (\lambda^{\mathcal C}_i(D) - \nu_{E_i}(D')) E_i
    \end{equation}
    then $P - \Theta$ is numerically equivalent to $P_\mathcal{C}(D)$, which in turn is numerically trivial over $Y$ (i.e., considering only curves contracted by $f$). Thus $- (P - \Theta)$ is numerically nef over $Y$ and $f_*(P - \Theta) = f_* P$ is effective. Hence, by the negativity lemma (\cref{lem:negativity}), $P - \Theta$ is effective. Moreover, $\Theta$ is $f$-exceptional while $P \geq 0$ has no exceptional component by construction, so $-\Theta$ by itself is effective. Finally, $D'$ was arbitrary, so we can replace $\nu_{E^{\mathcal C}_i}(D)$ by $\sigma_{E^{\mathcal C}_i}(D)$ in $\Theta$ and deduce 
    \begin{equation}
        \sigma_{E^{\mathcal C}_i}(D) \geq \lambda^{\mathcal C}_i(D)~.
    \end{equation}
    Applying the ceiling to both sides yields the desired result.
    
    The third item implies $\Bdiv(D) \supset \mathrm{supp} \, N_\mathcal{C}(D)$: an $E^{\mathcal C}_i$ belongs to the latter precisely when $\lambda^{\mathcal C}_i(D) > 0$ meaning $\sigma_{E^{\mathcal C}_i}(D) > 0$, and one can repeat this argument for all multiples of $D$. Let us now show the opposite inclusion. It suffices to show that the restriction of $D$ to $U := X \setminus \mathrm{supp} \, N_\mathcal{C}(D)$ has no stable fixed divisorial component. This is just $P_\mathcal{C}(D)|_U$, so it certainly suffices for $P_\mathcal{C}(D)$ to be semiample. If $X$ is a Mori dream space, then because $P_\mathcal{C}(D)$ is nef, it is immediately semiample. For $X$ Calabi--Yau type, semiamplitude isn't known to be implied by nefness (see, e.g., \cite{lazic2016morrison}). However, $P_\mathcal{C}(D)$ is big by the first item. Consequently, we can apply the classical basepoint free theorem from the theory of minimal models (see, e.g., \cite[Th. 3.3]{KM}). Namely, letting $(X, \Delta)$ denote the klt pair assured by \cref{def:CY_fano_type} and letting $mP_\mathcal{C}(D)$ be Cartier for $m > 0$, then $mP_\mathcal{C}(D) - (K_X + \Delta) \equiv mP_\mathcal{C}(D)$ is nef and big, hence $P_\mathcal{C}(D)$ is semiample.
\end{proof}
\begin{corollary}
    \label{cor:chamber_movable_part}
    Let $X$ be Calabi--Yau type (a Mori dream space). Then for every Mori chamber $\mathcal{C}$, $P_\mathcal{C}(D) \in \Mov(X)$ for every big (effective) divisor $D$ and
    \begin{equation}
        f_{\mathcal C}^* \bigl( \Nef(X_{\mathcal C}) \cap \Bigc(X_{\mathcal C}) \bigr) = \mathcal C \cap \Mov(X) \cap \Bigc(X),
    \end{equation}
    with $\Bigc$ replaced by $\Eff$ throughout in the Mori dream space case. 
\end{corollary}
\begin{proof}
    In the Calabi--Yau type (Mori dream space) case, applying \cref{prop:zariski} yields that $P_\mathcal{C}(D)$ is big (effective). Because $N_\mathcal{C}(P_\mathcal{C}(D)) = 0$, a further application gives $\Bdiv(P_\mathcal{C}(D)) = \emptyset$, so $P_\mathcal{C}(D)$ is movable. For the displayed equality, if $A$ is nef and big on $X_{\mathcal C}$ then $f_{\mathcal C}^* A$ is big and lies in $\mathcal{C}$ with $N_\mathcal{C}(f_{\mathcal C}^* A) = 0$, hence is movable by the same item. Conversely, a big movable class $D \in \mathcal{C}$ has $\Bdiv(D) = \emptyset$, so $N_\mathcal{C}(D) = 0$ and $D = f_{\mathcal C}^*(f_{\mathcal C *} D)$ with $f_{\mathcal C *} D$ nef and big.
\end{proof}

The important result here is that $\lambda^{\mathcal C}_i(D)$ characterizes important aspects of the stable base locus of $D$. Hence, for a Calabi--Yau type variety or Mori dream space, the Zariski decomposition --- enabled by the geography of models --- thus exhibits that the support of the stable base locus (and a lower bound on the degree of vanishing along components of that locus) is given piecewise on the effective cone by the Mori chambers. This will prove essential for our subtraction method, as otherwise, computing aspects of the fixed component of a linear system requires knowledge of its global sections, which are exactly what we are aspiring to compute.

It is worth stressing that for integral divisor classes, the negative part of the Zariski decomposition is generally rational, and the construction is sensitive only to the $\mathbb{Q}$-linear equivalence class; in particular, it is blind to torsion in the class group. It will eventually be preferable to modify this decomposition to preserve integrality and maintain awareness of torsion. For this, we will have to relax linearity to quasilinearity and sacrifice the movability of the positive part.
\begin{remark}
    Intuitively, the Zariski decomposition is rational for the following reason. Consider the Mori dream space case for simplicity: then, if $D \in \mathcal{C}$, $P_{\mathcal{C}}(D)$ is $f^* f_* D$, and while $f_* D$ is always $\mathbb{Q}$-Cartier in our setup --- indeed, this is necessary to define the pullback --- if it is not Cartier there is no reason it should pullback to anything nice. Indeed, there is no need for the pullback to be Weil: there is no need for it to be anything other than $\mathbb{Q}$-Weil.
\end{remark}

\subsection{Cohomology Theorems}

Our methods hinge on the reduction of global sections to Euler characteristics on suitable birational models with controlled singularities. This is furnished by Kawamata--Viehweg vanishing, which we now recall.
\begin{theorem}[{Kawamata--Viehweg vanishing (e.g., \cite[Cor. 2.49]{Fujino09})}]\label{th:KV}
    Let $(X, \Delta)$ be a projective klt pair and let $D$ be a $\mathbb{Q}$-Cartier Weil divisor on $X$. If $D - (K_X + \Delta)$ is big and nef, then
    \begin{equation*}
        H^i\bigl(X,\mathcal O_X(D)\bigr)=0
    \qquad\text{for all } i>0.
    \end{equation*}
\end{theorem}
The Calabi--Yau type (Fano type) property ensures the existence of precisely the klt pair needed to apply Kawamata--Viehweg vanishing to big nef divisors (nef divisors).
\begin{corollary}
    \label{cor:MAIN-VANISHING}
    Let a variety $X$ be Calabi--Yau type (Fano type) and let $D$ be a big nef divisor (nef divisor). Then $H^i\bigl(X,\mathcal O_X(D)\bigr) = 0$ for $i > 0$, so in particular $h^0\bigl(X,\mathcal O_X(D)\bigr) = \chi(X, \mathcal O_X(D))$.
\end{corollary}
\begin{proof}
    For $X$ Calabi--Yau type (Fano type), we apply \cref{th:KV} for $\Delta$ as in \cref{def:CY_fano_type}, and note that $D - (K_X + \Delta)$ is the sum of a big nef divisor and a nef divisor (a nef divisor and an ample divisor), which in particular must be nef and big.
\end{proof}
We additionally state Hirzebruch--Riemann--Roch --- a simple polynomial expression for the Euler characteristic in the smooth case --- as well as Snapper's theorem, which more generally ensures the quasipolynomial structure of Euler characteristics.
\begin{theorem}[Hirzebruch--Riemann--Roch \cite{o1981hirzebruch}]\label{th:HRR}
    Let $X$ be a smooth projective variety and let $L$ be a line bundle on $X$. Then
    \begin{equation*}
    \chi\bigl(X,\mathcal O_X(L)\bigr)=\int_X \mathrm{ch}(L)\,\mathrm{td}(T_X).
    \end{equation*}
    In particular, if $\dim X=n$, then $\chi(X,\mathcal O_X(L))$ is a polynomial of degree $n$ in the numerical class of $L$.
\end{theorem}
Hirzebruch--Riemann--Roch shows that Euler characteristics are polynomial on smooth varieties. For the singular varieties that arise as $D$-minimal models, one loses an explicit intersection-theoretic formula, but a polynomiality statement survives through a theorem of Snapper.
\begin{theorem}[Snapper \cite{snapper1959multiples, snapper1960polynomials, kleiman1966toward}]
    \label{th:snapper}
    Let $X$ be a complete algebraic scheme with a coherent sheaf $\mathcal{F}$, and let $\mathcal{L}_1, \dots, \mathcal{L}_r$ be a finite set of Cartier divisors. Then $\chi(X, \mathcal{F} \otimes (\otimes_{i=1}^r \mathcal{L}_i^{\otimes n_i}))$ is a polynomial in the $n_i$ of degree at most the dimension of the support of $\mathcal{F}$.
\end{theorem}
\begin{corollary}
    \label{cor:euler_char_is_quasipoly}
    The Euler characteristic $\chi$ on a complete $\mathbb{Q}$-factorial normal variety $X$ with finitely generated $\Cl(X)$ is a quasipolynomial on $\Cl(X)$, with each constituent polynomial of total degree at most $\dim X$.
\end{corollary}
\begin{proof}
    Let $D_1, \dots, D_\ell$ denote representatives of $\mathrm{Cl}(X) / \Pic(X)$, which is a finite group. Upon picking a set of Cartier divisors $\mathcal{L}_1, \dots, \mathcal{L}_r$ which generate $\Pic(X)$, we can apply \cref{th:snapper} to conclude that $\chi$ is polynomial on the coset $D_i + \Pic(X)$ in $\mathrm{Cl}(X)$ for $1 \leq i \leq \ell$. Together, these polynomials form a quasipolynomial for $\chi$ on the entire class group. Since ${\mathcal O}_X(D)$ is a rank-one reflexive sheaf, its support is all of $X$, so Snapper's polynomial has degree at most $\dim X$.
\end{proof}

\section{Cohomology Formulae}

\medskip
\subsection{Contraction}\label{sec:contraction}

\cref{prop:mmp_preserves_sections}, the result that the $D$-MMP preserves global sections, essentially immediately implies our main result for contraction.
\begin{proof}[Proof of \cref{MAIN:contract}]
    A $D$-minimal model exists by \cref{prop:cy_mmp} in the Calabi--Yau type case with $D \in \Bigc(X)$ and by \cref{prop:MDS_chambers} in the Fano type case with $D \in \Eff(X)$. \cref{prop:mmp_preserves_sections} then allows us to pass to such a $D$-minimal model (maintaining the Calabi--Yau type and Fano type properties, respectively, by \cref{prop:dmmp_preserves_cy} and \cref{prop:dmmp_of_fano_type_is_fano_type}) while preserving global sections. \cref{cor:MAIN-VANISHING} relates those global sections to an Euler characteristic and \cref{cor:euler_char_is_quasipoly} subsequently ensures quasipolynomiality.
\end{proof}

Even if contraction is frequently impractical, this existence of the approach can conceptually clarify the origin of global section formulae --- the divisors outside of the movable cone are controlled by Euler characteristics on divisorial contractions of the original variety $X$ and its SQMs. 
For example, the non-polynomiality (that is, strict quasipolynomiality) of global sections in a given Mori chamber arises from non-polynomiality of the Euler characteristic on the associated birational model due to the Picard group having finite index in the class group. As this only occurs for singular varieties, non-polynomiality can be understood as encoding the singularities of that model (indeed, if a $D$-minimal model is smooth, global sections in that chamber are polynomial by \cref{th:HRR}, so strict quasipolynomiality indicates singularities). We will also see in the next section how the guarantee that global sections must be quasipolynomial on each chamber --- even if that quasipolynomial is naively difficult to compute --- is quite valuable. 

\subsection{Subtraction} \label{sec:subtraction}

While the contraction method achieves our goal --- closed-form expressions for global sections --- it is frequently difficult to use in practice. First, it requires significant computation --- one must construct $D$-minimal models and compute their Euler characteristics. Additionally, divisorial contractions famously introduce complicated singularities, meaning the Euler characteristic on a generic $D$-minimal model will no longer be computable using \cref{th:HRR} --- for such models, the Euler characteristics are strictly quasipolynomial and computable only through intimate knowledge of their singularities. 

This is motivation for studying the subtraction approach, which reduces the global sections of a general effective (or big, in the Calabi--Yau case) divisor to those of a divisor contained in the movable cone, which can be handled by passing to an SQM rather than a general divisorial contraction. Subtraction thus only requires knowledge of SQMs rather than all $D$-minimal models. As discussed, subtraction will turn out to be naively less powerful than contraction, but the result for contraction actually strengthens subtraction and enables it to furnish a comprehensive piecewise quasipolynomial function for global sections (\cref{MAIN:together}).

Our first set of results formalises the subtraction method.
\begin{definition}
    \label{def:subtraction}
    Let $X$ be a normal projective variety with effective divisor $D$. Let $E_1,\ldots, E_r$ be prime fixed components of $|D|$. If integers $0 \leq a_i \leq \sigma_{E_i}(D)$ are chosen, not all zero, then we call $D - \sum_i a_i E_i$ a \emph{subtraction} of $D$. 
\end{definition}
\begin{corollary}
    \label{cor:subtraction_algo_preserves}
    Let $X$ be a normal projective variety with effective divisor $D$. If $D'$ is a subtraction of $D$, then 
    \(
        h^0(X, \mathcal{O}_X(D)) = h^0(X, \mathcal{O}_X(D')).
    \)
    Moreover, every sequence of successive subtractions starting from $D$ results in a divisor $D^\downarrow \in \Mov(X)$ after finitely many steps --- we call any such $D^\downarrow$ a movable subtraction of $D$.
\end{corollary}
\begin{proof}
    The first statement follows from \cref{lem:exceptional_preserve} because subtraction consists of successively removing prime divisors that are fixed components of the complete linear system $|D|$. Let $D = \sum_i d_i D_i$ for prime divisors $D_i$ and let $D^{(i+1)}$ denote the result of applying subtraction to $D^{(i)}$, with $D^{(0)} := D$. The second statement follows because the integer sequence $\sum_j \nu_{D_j}(D^{(i)})$ is strictly monotonically decreasing as a function of $i$ but must remain non-negative, ensuring that the subtraction procedure terminates after finitely many steps. Since the procedure terminates precisely when no further subtraction is possible, the terminal divisor has no fixed component and hence is movable.
\end{proof}
Once a movable subtraction is reached by a sequence of subtractions, vanishing theorems reduce the computation of global sections to an Euler characteristic. We note that movable subtractions in general are not unique.
\begin{proposition}
    \label{prop:subtraction_algo_chi}
    Let a variety $X$ be Calabi--Yau type with divisor $D \in \Bigc(X)$ admitting a big movable subtraction $D^\downarrow$, or let $X$ be Fano type with divisor $D \in \Eff(X)$. Then there exists a small birational modification $\varphi \colon X \dashrightarrow Y$ such that $\varphi_*(D^\downarrow)$ is nef and
    \begin{equation*}
        h^0\bigl(X,\mathcal O_X(D)\bigr) = \chi\bigl(Y,\mathcal O_{Y}(\varphi_* D^\downarrow)\bigr)~.
    \end{equation*}
\end{proposition}
\begin{proof}
    The existence of $Y$ follows from the existence of $D$-minimal models for Calabi--Yau type varieties (\cref{prop:cy_mmp}) and Mori dream spaces (\cref{prop:MDS_chambers}). In particular, $\varphi$ is an SQM because $D^\downarrow$ is movable (by \cref{prop:cy_preq} in the Calabi--Yau type case and \cref{prop:MDS_chambers} in the Fano type case). The equality then follows from \cref{cor:MAIN-VANISHING}, because varieties of Calabi--Yau type and of Fano type are each closed under taking $D$-minimal models, by \cref{prop:dmmp_preserves_cy} and \cref{prop:dmmp_of_fano_type_is_fano_type}, respectively.
\end{proof}
Thus, if the hypotheses of \cref{prop:subtraction_algo_chi} are satisfied, computing global sections of an effective divisor reduces to computing an Euler characteristic provided one understands both the birational geometry of $X$ and how to compute $D^\downarrow$ using subtractions. The Euler characteristics are quasipolynomial in general, by \cref{cor:euler_char_is_quasipoly}, and for the case of smooth SQMs --- e.g., if $X$ is a smooth Calabi--Yau threefold, by \cref{th:flop_of_smooth_cy3} --- they are polynomial and easy to compute by \cref{th:HRR}. 

We now turn our attention to movable subtractions $D^\downarrow$. These are generally difficult to compute directly, however the more general notion of a subtraction of a divisor allows the computation of movable subtractions to be performed recursively, rather than all at once. Algorithmically, then, it therefore suffices to just have access to some systematic and efficient way to perform subtractions. This will be furnished by the Zariski decomposition.

First, we introduce a modification of the Zariski decomposition which is integral.
\begin{definition}
    \label{def:integral-zariski}
    Let $X$ be a Calabi--Yau type variety (Mori dream space) with a big (effective) divisor $D$ contained in a Mori chamber $\mathcal C$. In the notation of \cref{sec:zariski}, its integral Zariski decomposition is defined as $D = P_{\mathbb{Z},\mathcal{C}}(D) + N_{\mathbb{Z},\mathcal{C}}(D)$ with
    \begin{equation}
        N_{\mathbb{Z},\mathcal{C}}(D) = \sum_{i=1}^r \lceil \lambda^{\mathcal C}_i(D) \rceil E^{\mathcal C}_i~,\quad  P_{\mathbb{Z},\mathcal{C}}(D) = D - N_{\mathbb{Z},\mathcal{C}}(D)~.
    \end{equation}
\end{definition}
From here we can straightforwardly show the following.
\begin{corollary}
    \label{cor:alg}
    In the setting of \cref{def:integral-zariski}, $P_{\mathbb{Z},\mathcal{C}}(D)$ is a subtraction of $D$. In particular, if $X$ is Calabi--Yau type (Fano type), for any big (respectively, effective) divisor $D$ there is a canonical movable subtraction achieved by iteratively performing $D \mapsto P_{\mathbb{Z},\mathcal{C}}(D)$ for $\mathcal{C}$ a Mori chamber containing $D$, unless at some iteration a non-big divisor is achieved (respectively, unconditionally).
\end{corollary}
\begin{proof}
    It suffices to note that \cref{prop:zariski} ensures that subtracting $N_{\mathbb{Z},\mathcal{C}}(D)$ from $D$ only involves the removal of prime divisors in the stable base locus.
\end{proof}
\begin{definition}
    \label{def:zariski_movable_subtraction}
    We refer to the subtraction of \cref{cor:alg} as the Zariski subtraction; in particular, we refer to the distinguished movable subtraction of an effective divisor $D$ achieved by iteratively performing Zariski subtractions until a movable divisor is first achieved as the Zariski movable subtraction $D^\downarrow_\mathrm{zms}$. 
\end{definition}
Unlike the positive part of the Zariski decomposition, a single Zariski subtraction need not be movable. Rather, it involves removal of some divisors supported on the stable divisorial fixed locus as prescribed by the chamberwise Zariski decomposition, after rounding coefficients up to the nearest integer. In particular, $P_{\mathbb{Z},\mathcal{C}}(D)$ need not lie in $\mathcal C$; indeed, in general it lies on the boundary of $\mathcal C$ or in a neighbouring chamber.

Before more explicitly describing how these subtractions can be performed algorithmically to compute global sections, we explain how the discussion extends from effective divisors to divisor classes in the effective cone. In practice, one often would like to compute the global sections of the isomorphism class of sheaves associated to some divisor class, rather than for one specific representative effective divisor. This is no problem, as by construction, a subtraction of $D$ will be a subtraction of any $D' \in |D|$, so our discussion lifts to classes.

If a divisor class has no effective representative, then its complete linear system is empty and hence $h^0(X,\mathcal O_X(D))=0$. In this case one should not interpret the prime exceptional divisors associated with the relevant Mori chamber as fixed components of $|D|$ in the usual sense. Nevertheless, the chamberwise integral Zariski map remains well defined at the level of divisor classes. Its subtraction step may then carry the class outside the effective cone. When this occurs, it certifies that the original class had no effective representative, and the algorithm terminates with output zero. Such ``holes'' in the effective cone are well known to occur: divisor classes in $\Eff(X)$ need not themselves be effective \cite{Gendler:2026uux}.

The location of the holes arising from such overshoots is also transparent from the convex geometry. For a fixed Mori chamber, integral Zariski subtraction rounds the coefficients of the exceptional divisors upward, so one can detect the holes simply by checking when the resulting class violates one of the supporting inequalities of $\Eff(X)$. Since the rounding errors are bounded and are unchanged by translation by integral exceptional classes, these holes occur in bounded-width strips parallel to the relevant exceptional directions, with infinite arithmetic families obtained by translation along those directions. 

The algorithm requires only the Mori chamber decomposition as well as the partitioning of each chamber into the pullback of the (effective) nef cone and its exceptional part, which we package as follows.
\begin{definition}
    \label{def:birational_data}
    The Mori chamber data of a Calabi--Yau type variety or Mori dream space $X$ consists of the Mori chambers $\mathcal{C} = \varphi^* \Nef(X_\mathcal{C} \cap \Eff(X_\mathcal{C}))$ of all SQMs $\varphi : X \dashrightarrow X_\mathcal{C}$ together with the prime exceptional divisors of the extremal contractions associated to each facet of $\Mov(X)$ not contained in $\partial \Eff(X)$.
\end{definition}
This data suffices to determine the Mori chamber decomposition. Indeed, for each SQM $\varphi : X \dashrightarrow X_\mathcal{C}$ one can study the faces $F$ of $\varphi^* (\Nef(X_\mathcal{C}) \cap \Eff(X_\mathcal{C}))$ which are additionally contained in $\partial \Mov(X) \setminus \partial \Eff(X)$ (recall \cref{cor:chamber_movable_part}). Each such $F$ could satisfy $F = \mathcal{C}' \cap \Mov(X)$ for some Mori chamber $\mathcal{C}'$, and in fact this will be true if $F$ has codimension $r$ and is contained in the intersection of $r$ facets of $\Mov(X)$ associated to $r$ distinct prime exceptional divisors $E_1, \dots, E_r$. In this case, $\mathcal{C}' = F + \sum_{i=1}^r \mathbb{R}_{\geq 0} E_i$, with $F = f_{\mathcal{C}'}^*(\Nef(X_{\mathcal{C}'}) \cap \Eff(X_{\mathcal{C}'}))$.
\begin{remark}
    As noted in \cref{rem:KM}, the Mori chamber decomposition of a Calabi--Yau type variety may not have finitely many chambers, and individual chambers may not have finitely many generators. While this renders the birational data of \cref{def:birational_data} somewhat unwieldy from a practical standpoint, the Kawamata--Morrison cone conjecture does entail that the birational data of \cref{def:birational_data} is determined by finitely many chambers and extremal rays. We do stress that none of our results depend on such conjectures.
\end{remark}

\begin{algorithm}
\caption[Computing $h^0(X,\mathcal{O}_X(D))$ from Mori chamber data]{%
  Given the Mori chamber data of $X$ and Euler characteristics $\chi(X_\mathcal{C}, -)$ of SQMs of $X$, attempt to compute $h^0(X,\mathcal{O}_X(D))$ for either
  \par
  \begin{minipage}[t]{\dimexpr\linewidth-1em}
    \vspace{-0.6\baselineskip}
    \begin{enumerate}
      \item $X$ Fano type and $D \in \Eff(X)$
      \item $X$ Calabi--Yau type with $D \in \Bigc(X)$ 
    \end{enumerate}
  \end{minipage}}
\label{alg}
\begin{algorithmic}
    \State Set $i = 0$
    \State Set $[D^{(i)}] = [D]$
    \While{$[D^{(i)}] \notin \Mov(X)$} \Comment{Perform integral Zariski subtractions until $[D^{(i)}] = [D_{\mathrm{zms}}^\downarrow]$}
        \State Increment $i$ by $1$.
        \State Set $\mathcal{C} =$ Mori chamber such that $[D^{(i-1)}] \in \mathcal{C}$
        \State Solve for the $\lambda^{\mathcal C}_j(D^{(i-1)})$ \Comment{see \cref{def:zariski_decomp}}
        \State Set $[D^{(i)}] = [P_{\mathbb{Z},\mathcal{C}}(D^{(i-1)})]$ \Comment{see \cref{cor:alg} and \cref{def:integral-zariski}}
        \If{$[D^{(i)}] \notin \Eff(X)$}
            \State \Return 0
        \EndIf
        \If{$X$ Calabi--Yau type and $[D^{(i)}] \notin \Bigc(X)$} 
            \State \Return N/A 
        \EndIf
    \EndWhile
    \State Set $[D_{\mathrm{zms}}^\downarrow] = [D^{(i)}]$
    \State Set $\mathcal{C} =$ Mori chamber with $[D^\downarrow_{\mathrm{zms}}] \in f_{\mathcal{C}}^* (\Nef(X_\mathcal{C}) \cap \Eff(X_\mathcal{C}))$
    \State \Return $\chi(X_\mathcal{C}, \mathcal{O}_{X_\mathcal{C}}(f_{\mathcal{C}*} [D^\downarrow_{\mathrm{zms}}]))$
\end{algorithmic}
\end{algorithm}

At this point, we have achieved a proof of the second main result \cref{MAIN:subtract}.

\begin{proof}[Proof of \cref{MAIN:subtract}]
    Combining \cref{cor:subtraction_algo_preserves}, \cref{prop:subtraction_algo_chi}, and \cref{cor:alg} yields \cref{alg} computing $h^0(X,\mathcal O_X(D))$.
\end{proof}

Now we turn our attention to \cref{MAIN:together}.
In particular, we will now illustrate how $P_{\mathbb{Z},\mathcal{C}}$ and $N_{\mathbb{Z},\mathcal{C}}$ are quasilinear and that asymptotically, the Zariski movable subtraction $D^\downarrow_{\mathrm{zms}}$ of \cref{def:zariski_movable_subtraction} varies quasilinearly as a function of the original divisor on the interior of Mori chambers.

\begin{proposition}
    \label{prop:integerZariski_quasilinear}
    Let $X$ be Calabi--Yau type or a Mori dream space, and let $\mathcal C$ be a Mori chamber with birational map $f : X \dashrightarrow X_\mathcal{C}$. Then $P_{\mathbb Z,\mathcal C}$ and $N_{\mathbb Z,\mathcal C}$ extend to quasilinear maps on $\Cl(X)$, with the associated lattice being
    \begin{equation*}
        \Lambda_{\mathcal C} = \sum_i\mathbb Z E_i^{\mathcal C} \oplus (f^* \Pic(X_\mathcal{C}))_\mathrm{sat},
    \end{equation*}
    with saturation taken in $\Cl(X)$.
\end{proposition}
\begin{proof}
    Let $D$ be a divisor. It suffices to study \(N_{\mathbb Z,\mathcal C}\), since $P_{\mathbb Z,\mathcal C}(D) = D-N_{\mathbb Z,\mathcal C}(D)$. On the lattice $\Lambda_{\mathcal C}$ we have $N_{\mathbb Z,\mathcal C}=N_{\mathcal C}$, since all $\lambda_i^{\mathcal C}$ take integral values there. Define the rounding error $ R_{\mathcal C} := N_{\mathbb Z,\mathcal C}-N_{\mathcal C}$.  It suffices to show that $R_{\mathcal C}(D)$ depends only on the coset of $D$ modulo $\Lambda_{\mathcal C}$. Indeed, choose representatives for the cosets of $\Lambda_{\mathcal C}$ in $\Cl(X)$. Write $ D=A+B$, with $A\in\Lambda_{\mathcal C}$ and $B$ one of these representatives. Then, using the linearity of $N_{\mathcal C}$ and the integrality of the $\lambda_i^{\mathcal C}(A)$ on $\Lambda_{\mathcal{C}}$, we have
    \begin{equation*}
        N_{\mathbb Z,\mathcal C}(D) = N_{\mathbb Z,\mathcal C}(A+B)  = N_{\mathcal C}(A)+N_{\mathbb Z,\mathcal C}(B).
    \end{equation*}
    Therefore
    \begin{equation*}
        R_{\mathcal C}(D) = N_{\mathbb Z,\mathcal C}(D)-N_{\mathcal C}(D) = N_{\mathbb Z,\mathcal C}(B)-N_{\mathcal C}(B)  = R_{\mathcal C}(B).
    \end{equation*}
    Thus the rounding error is constant on each coset of $\Lambda_{\mathcal C}$, and consequently $N_{\mathbb Z,\mathcal C}$ is affine linear on each such coset. Hence $N_{\mathbb Z,\mathcal C}$ is quasilinear, and so is $P_{\mathbb Z,\mathcal C}$.
\end{proof}
As noted in \cref{sec:contraction}, quasipolynomiality of global section formulae is related to the index of the Picard group of $D$-minimal models inside the class group, and thus encodes information about singularities. This same quasipolynomiality must be reproduced by the subtraction method, and this is often achieved through the quasipolynomiality of $P_{\mathbb Z,\mathcal C}(D)$. In this way, the manner in which the Mori chamber decomposition intersects the lattice of divisor classes captures important information about $D$-minimal models of a variety.

Having understood the lattice associated to the integral Zariski decomposition, we can demonstrate that away from boundaries of a Mori chamber, the Zariski movable subtraction behaves quite nicely. We will achieve this through a series of lemmas, culminating in \cref{prop:asymptotic_subtraction}. First, we prove a technical lemma showing that certain shifts of faces of Mori chambers can always be ``corrected'' such that they fall entirely in a single (distinct) Mori chamber. This will be used to show that on for coset $H$ of $\Lambda_{\mathcal{C}}$ of a Mori chamber $\mathcal{C}$, the Zariski subtraction maps ``most'' divisors --- in a sense we will shortly make precise --- to a common new Mori chamber $\mathcal{C}'$.
\begin{lemma}
    \label{lem:corrected_shifted_face}
    Let a variety $X$ be Calabi--Yau type or a Mori dream space, let $\mathcal C$ be a Mori chamber, and let $T$ be an effective $\mathbb{Q}$-divisor. We can then select an $S_T \in \relint(\mathcal C\cap\Mov(X))$ and a Mori chamber $\mathcal{C}_T$ such that
    \begin{equation}
        \begin{aligned}
            \mathcal{C} \cap \Mov(X) &\subset \mathcal{C}_T \cap \Mov(X)~,\\
            \ex(f_{\mathcal C_T}) &\subseteq \ex(f_{\mathcal C})~, \\
            -T + S_T + \mathcal{C} \cap \Mov(X) &\subset \mathcal{C}_T \cap \Bigc(X)~. \\
        \end{aligned}
    \end{equation}
\end{lemma}
\begin{proof}
    Fix an ample divisor $A'$ on $X_{\mathcal C}$ and put $P = f_{\mathcal C}^* A'$, which is big and lies in $\relint (\mathcal{C} \cap \Mov(X)) = f_{\mathcal C}^* \operatorname{Amp}(X_{\mathcal C})$, the equality holding by \cref{cor:chamber_movable_part} and injectivity of $f_{\mathcal C}^*$. We begin by constructing $\mathcal C_T$.

    To this end, fix $\alpha > 0$ such that the segment $I_\alpha = \{ P - \delta T \; | \; \delta \in [0, \alpha] \}$ is contained in $\Bigc(X)$. We will choose $\mathcal{C}_T$ to be a Mori chamber containing $I_\alpha$ --- to show such a chamber exists, we first show $I_\alpha$ is covered by finitely many chambers. This is clear for Mori dream spaces, and for varieties of Calabi--Yau type, this will be a technical application of \cite[Cor. 1.1.5]{BCHM10}, which shows that certain regions $\mathcal{E}_{A,\pi}(V)$ of $\mathbb{R}$-divisors are covered by Mori chambers of finitely many log minimal models. To this end, write $P - \alpha T \sim_\mathbb{Q} A + B$ for $A$ and $B$ effective $\mathbb{Q}$-divisors and $A$ ample, and let $\Delta$ be as in \cref{def:CY_fano_type}. Then $J_\alpha = \{ A + B + (\alpha - \delta) T \; | \; \delta \in [0, \alpha] \}$ is the same line segment as $I_\alpha$ up to linear equivalence, just represented by a different family of divisors. We can choose $\epsilon > 0$ sufficiently small such that $(X, \Delta + \epsilon D)$ is a klt pair for all $D \in J_\alpha$ --- it then suffices to show that the segment $\Delta + \epsilon J_\alpha$ is contained in one of the $\mathcal{E}_{A,\pi}(V)$, as $(X, \Delta + \epsilon D)$-log minimal models of divisors are $D$-minimal models. Following \cite[Def. 1.1.4]{BCHM10}, one can then check that because
    \begin{equation}
        \Delta + \epsilon J_{\alpha} 
        = \left\{ \epsilon A + \Big( \Delta + \epsilon \Big[ B + (\alpha - \delta) T \Big] \Big) \; \Big| \; \delta \in [0, \alpha] \right\}
    \end{equation}
    we have that $\Delta + \epsilon J_{\alpha} \subset \mathcal{E}_{\epsilon A,\pi}(V)$ for $V$ the $\mathbb{R}$-span of $\Delta$, $B$, and $T$ as divisors and $\pi$ the morphism from $X$ to a point. Because $I_\alpha$ and $J_\alpha$ correspond to the same divisor classes, the same is true for $I_\alpha$.

    Now that we know $I_\alpha$ is covered by finitely many Mori chambers, we note that some intersecting Mori chamber must contain an open neighborhood of $P$ in $I_\alpha$. Indeed by convexity each Mori chamber intersects $I_\alpha$ on a subsegment closed in $I_\alpha$; if all chambers containing $P$ contained no other points, an open neighborhood of $P$ in $I_\alpha$ would be uncovered, a contradiction. Letting $\mathcal{C}_T$ denote a Mori chamber containing an open neighborhood of $P$ in $I_\alpha$, by convexity we can rescale $\alpha$ such that $I_\alpha \subset \mathcal{C}_T$.
    
    With $\mathcal{C}_T$ in hand, we can argue the first inclusions asserted by the lemma. One can write the Zariski decomposition of $P$ with respect to both $\mathcal{C}$ and $\mathcal{C}_T$ and apply \cite[Lem. 1.7]{HuKeel} to conclude that $g = f_{\mathcal{C}} \circ f_{\mathcal{C}_T}^{-1}$ is a morphism. Thus, as a birational morphism pulls nef classes back to nef classes and effective classes back to effective classes, $g^* \big( \Nef(X_{\mathcal{C}}) \cap \Eff(X_{\mathcal{C}}) \big) \subset \Nef(X_{\mathcal{C}_T}) \cap \Eff(X_{\mathcal{C}_T})$. Additionally using \cref{cor:chamber_movable_part}, we have
    \begin{equation}
        \begin{aligned}
            \mathcal{C} \cap \Mov(X) 
            &= f_{\mathcal{C}}^* \big( \Nef(X_{\mathcal{C}}) \cap \Eff(X_{\mathcal{C}}) \big) 
            = f_{\mathcal{C}_T}^* \Big( g^* \big( \Nef(X_{\mathcal{C}}) \cap \Eff(X_{\mathcal{C}}) \big) \Big) \\
            &\subset f_{\mathcal{C}_T}^* \big( \Nef(X_{\mathcal{C}_T}) \cap \Eff(X_{\mathcal{C}_T}) \big) 
            = \mathcal{C}_T \cap \Mov(X)~.
        \end{aligned}
    \end{equation}
    The second inclusion of this lemma follows from the equality $f_{\mathcal C} = g \circ f_{\mathcal{C}_T}$.

    Now we finish the final inclusion. As $\mathcal{C}_T$ is a cone containing $P - \alpha T$, taking $S_T := \alpha^{-1} P \in \relint (\mathcal{C}_T \cap \Mov(X))$ gives $-T + S_T = \alpha^{-1}(P - \alpha T) \in \mathcal{C}_T$, and this class is big because $P - \alpha T \in I_\alpha \subset \Bigc(X)$. The first inclusion and convexity then give
    \begin{equation*}
        -T + S_T + \mathcal{C} \cap \Mov(X) \subset \mathcal{C}_T + \mathcal{C}_T = \mathcal{C}_T~,
    \end{equation*}
    while $\Bigc(X) + \Eff(X) \subset \Bigc(X)$ places the same set in $\Bigc(X)$.
\end{proof}
With this lemma in hand, we can now determine how the image of a single Zariski subtraction $P_{\mathbb{Z}, \mathcal{C}}$ on a Mori chamber interplays with the Mori chamber decomposition.
\begin{lemma}
    \label{lem:zariski_sub_chamber}
    In the setting of \cref{lem:corrected_shifted_face}, let $R_H$ be the restriction of $N_{\mathbb Z,\mathcal C} - N_{\mathcal C}$ to a coset $H$ of $\Lambda_\mathcal{C}$, which is a constant effective $\mathbb{Q}$-divisor by \cref{prop:integerZariski_quasilinear}. Let $\mathcal{C}_H$ and $S_H$ be the Mori chamber and class furnished by \cref{lem:corrected_shifted_face} for $T = R_H$. Then there is a coset $H'$ of $\Lambda_{\mathcal{C}_H}$ such that
    \begin{equation*}
        P_{\mathbb Z,\mathcal C}\Big( \, (S_H + \mathcal{C}) \cap H \, \Big) \subset \mathcal{C}_{H} \cap H' \cap \Bigc(X)~.
    \end{equation*}
\end{lemma}
\begin{proof}
    Now, because on $H$ we have 
    \begin{equation}
        P_{\mathbb Z,\mathcal C} = P_{\mathcal C} - R_H~,
    \end{equation}
    we conclude from \cref{lem:corrected_shifted_face} that
    \begin{equation}
        P_{\mathbb Z,\mathcal C}\Big( \, (S_H + \mathcal{C}) \cap H \, \Big) = -R_H + P_{\mathcal C}\Big( \, (S_H + \mathcal{C}) \cap H \, \Big) \subset -R_H + S_H + \mathcal{C} \cap \Mov(X) \subset \mathcal{C}_H \cap \Bigc(X) ~.
    \end{equation}
    Finally, to conclude that $P_{\mathbb Z,\mathcal C}$ maps $(S_H + \mathcal{C}) \cap H$ to a common coset, let $D, D' \in H$ and note
    \begin{equation}
        P_{\mathbb Z,\mathcal C}(D) - P_{\mathbb Z,\mathcal C}(D') = P_{\mathcal C}(D) - P_{\mathcal C}(D')
    \end{equation}
    is an integral divisor and lies in the span of $\mathcal C \cap \Mov(X) \subset \mathcal{C}_H \cap \Mov(X)$, meaning it is contained in the lattice $\Lambda_{\mathcal{C}_H}$, so $P_{\mathbb Z,\mathcal C}(D), P_{\mathbb Z,\mathcal C}(D')$ belong to a common coset of $\Lambda_{\mathcal{C}_H}$.
\end{proof}
Iterating this lemma, we can extend the result to the Zariski movable subtraction.
\begin{lemma}
    \label{lem:movable_zariski_sub_SQM}
    In the setting of \cref{lem:zariski_sub_chamber}, there exists $S \in \relint (\mathcal C \cap \Mov(X))$ and, for each coset $H$ of $\Lambda_\mathcal{C}$, an SQM $\varphi_H \colon X \dashrightarrow X_H$ such that for $D \in S + \relint \mathcal{C}$,
    \begin{equation*}
        D^\downarrow_{\mathrm{zms}} \in \varphi_H^*\left( \Nef(X_H) \cap \Eff(X_H) \right) \cap \Bigc(X)
    \end{equation*}
    where $H$ is the coset containing $D$ and $D^\downarrow_{\mathrm{zms}}$ differs from $P_{\mathcal C}(D)$ by a class in $\ex(f_{\mathcal C})$ depending only on $H$.
\end{lemma}
\begin{proof}
    Fix $D \in (S_H + \relint \mathcal{C}) \cap H$ for $S_H$ from \cref{lem:corrected_shifted_face} and define $A = f_{\mathcal C}^* A' \in \relint (\mathcal{C} \cap \Mov(X))$ for $A'$ ample on $X_\mathcal{C}$. We would like to construct a sequence $D = D_0, D_1, \dots$ of divisors along with Mori chambers $\mathcal{C} =: \mathcal{C}_0, \mathcal{C}_1, \dots$ and shifts $S_H =: S_1, S_2, \dots$ such that $D_i = P_{\mathbb{Z}, \mathcal{C}_{i-1}}(D_{i-1})$ is the Zariski subtraction of $D_{i-1} \in \mathcal{C}_{i-1}$, and $\mathcal{C}_i$, $S_i$ are the Mori chamber and shift from \cref{lem:corrected_shifted_face} derived from the remainder $T_i = P_{\mathcal{C}}(D) - D_i$ and the Mori chamber $\mathcal{C}$.\footnote{We note that while it may seem natural to apply \cref{lem:corrected_shifted_face} to $\mathcal{C}_{i-1}$ rather than $\mathcal{C}$, because we are constructing a formula on $\mathcal{C}$, not $\mathcal{C}_{i-1}$, it is more relevant to construct a new Mori chamber containing the image of divisors in $\mathcal{C}$ under Zariski subtractions, rather than divisors in $\mathcal{C}_{i-1}$.}
    
    If we can achieve this, then from \cref{lem:corrected_shifted_face} we have the inclusions 
    \begin{align}
        \mathcal{C} \cap \Mov(X) &\subset \mathcal{C}_{i} \cap \Mov(X) \label{eq:moving_containment} \\
        \ex(f_i) &\subseteq \ex(f_{\mathcal{C}}) \label{eq:ex_containment}
    \end{align}
    for $f_i$ the birational map associated to $\mathcal{C}_i$. By \cref{lem:corrected_shifted_face}, it suffices at each step for $P_{\mathcal{C}}(D) \in S_{i+1} + \mathcal{C} \cap \Mov(X)$. We proceed inductively --- we can clearly do this for $i = 0$, so we assume we can do it for $i < j$. Note that
    \begin{equation}
        D_j = P_\mathcal{C}(D) + (P_{\mathbb{Z},\mathcal{C}_{j-1}} \circ \dots \circ P_{\mathbb{Z},\mathcal{C}_1})(-R_H) = P_\mathcal{C}(D) - T_j
    \end{equation}
    as $D_1 = P_\mathcal{C}(D) - R_H$ and by \cref{eq:moving_containment} each subsequent $P_{\mathcal{C}_i}(D_i)$ preserves $P_\mathcal{C}(D)$. In particular, each $T_j$ is actually independent of $D$, and depends only on $H$. Now, if $P_{\mathcal{C}}(D) \notin S_{j+1} + \mathcal{C} \cap \Mov(X)$, we can simply replace the original $D$ by $D + mA$ for any sufficiently large (coset-preserving) $m$ such that $P_{\mathcal{C}}(D) \in S_{j+1} + \mathcal{C} \cap \Mov(X)$ is achieved. This affects each $D_i$ only by changing $P_\mathcal{C}(D)$, not $T_i$, so all of the earlier Zariski subtractions take the same form, and now we can safely construct $D_{j+1}$, finishing the inductive step. 

    Now it suffices to show that this process results in a movable $D_i$ in finitely many steps. In particular, while iterative application of Zariski subtraction certainly terminates in a movable divisor by \cref{cor:subtraction_algo_preserves}, in this procedure we are allowing $D$ to change, so we must show that our process here terminates. 

    To this end, we may assume $h^0(X, \mathcal{O}_X(D)) > 0$ by replacing $D$ by $D + mA$ once more, and we may take the representative of $A$ to satisfy $\nu_{E^{\mathcal C}_i}(A) = 0$ for every $i$, as $A$ is movable. Each subtraction removes a divisor supported on the $E^{\mathcal C}_i$ by \cref{eq:ex_containment}, so the argument of \cref{cor:subtraction_algo_preserves} applies to $\sum_i \nu_{E^{\mathcal C}_i}(D_j)$: this sum strictly decreases along $D_0, D_1, \dots$ while remaining non-negative, and it is unaffected by the addition of $A$, so the process terminates.

    Finally, let $D_N$ be the terminal, movable divisor of the sequence, so that $D^\downarrow_{\mathrm{zms}} = D_N$ differs from $P_\mathcal{C}(D)$ by the class $T_N \in \ex(f_\mathcal{C})$, which depends only on $H$, and set $S' = S_1 + \dots + S_N$, so that the single condition $P_\mathcal{C}(D) \in S' + \mathcal{C} \cap \Mov(X)$ suffices for every step at once. Factoring $f_N = h \circ \varphi_H$ into an SQM $\varphi_H : X \dashrightarrow X_H$ followed by a morphism (\cref{prop:CY_type_bir_contract_factor} for Calabi--Yau type varieties and \cref{prop:MDS_bir_contract_factor} for Mori dream spaces) then gives $D^\downarrow_{\mathrm{zms}} \in \mathcal{C}_N \cap \Mov(X) \subset \varphi_H^*\left( \Nef(X_H) \cap \Eff(X_H) \right)$, while $D^\downarrow_{\mathrm{zms}}$ is big by \cref{lem:corrected_shifted_face}. As $\Lambda_\mathcal{C}$ has finite index in $\Cl(X)$ there are finitely many such $S'$, so we can take their sum $S$ and conclude that the result holds for divisors $D \in S + \relint \mathcal{C}$ (as the image of this cone under $P_\mathcal{C}$ is $S + \relint \mathcal{C} \cap \Mov(X)$).
\end{proof}
\begin{proposition}
    \label{prop:asymptotic_subtraction}
    In the setting of \cref{lem:movable_zariski_sub_SQM}, the Zariski movable subtraction $D \mapsto D^\downarrow_{\mathrm{zms}}$ is quasilinear with lattice $\Lambda_{\mathcal C}$ and big for all $D \in S + \relint \mathcal{C}$ --- in particular, these properties hold asymptotically on $\relint \mathcal{C}$.
\end{proposition}
\begin{proof}
    By \cref{lem:movable_zariski_sub_SQM}, on $S + \relint \mathcal{C}$ the divisor $D^\downarrow_{\mathrm{zms}}$ is big and differs from $P_\mathcal{C}(D)$ by a class depending only on the coset $H$ of $\Lambda_\mathcal{C}$ containing $D$. As $P_\mathcal{C}$ is linear and that class is constant on $H$, the map $D \mapsto D^\downarrow_{\mathrm{zms}}$ is affine linear on each coset of $\Lambda_\mathcal{C}$, and hence quasilinear with lattice $\Lambda_\mathcal{C}$.
\end{proof}
\begin{definition}
    \label{def:quasilinear_zariski_subtraction}
    Let $X$ be a Calabi--Yau type variety or Mori dream space and let $\mathcal C$ be a Mori chamber. We denote by
    $$D\longmapsto D^\downarrow_{\mathcal C}$$
    the quasilinear map that asymptotically agrees with the Zariski movable subtraction on $\relint \mathcal C$.
\end{definition}
At this point we have introduced many variants of subtractions --- for convenience, we summarize these in \cref{tab:subtractions}.
\begin{table}
\centering
\small
\setlength{\tabcolsep}{4pt}
\renewcommand{\arraystretch}{1.1}
\begin{tabular}{|L{0.3\linewidth}|L{0.38\linewidth}|L{0.085\linewidth}|L{0.14\linewidth}|}
\hline
\textbf{Notion} & \textbf{Description} & \textbf{Source}
  & \textbf{Needs Zariski decomp.?} \\
\hline
Subtraction $D - \sum_i a_i E_i$
  & Requires $0 \leq a_i \leq \sigma_{E_i}(D)$ for integral $a_i$; leaves global sections unchanged.
  & \cref{def:subtraction}
  & No \\
\hline
Movable subtraction $D^\downarrow$
  & Any iteration of subtractions yielding a movable divisor, always achieved with finite iterations.
  & \cref{cor:subtraction_algo_preserves}
  & No \\
\hline
Zariski subtraction $P_{\mathbb{Z},\mathcal{C}}(D)$
  & A canonical quasilinear subtraction induced by the Zariski decomposition on a Mori chamber $\mathcal{C}$.
  & \cref{def:zariski_movable_subtraction}
  & Yes \\
\hline
Zariski movable subtraction $D^\downarrow_{\mathrm{zms}}$
  & A canonical movable subtraction induced by iterating Zariski subtractions.
  & \cref{def:zariski_movable_subtraction}
  & Yes\\
\hline
Quasilinear asymptotic Zariski movable subtraction $D^\downarrow_{\mathcal{C}}$
  & A quasilinear map that agrees asymptotically with the Zariski movable subtraction on a fixed Mori chamber $\mathcal{C}$.
  & \cref{def:quasilinear_zariski_subtraction}
  & Yes \\
\hline
\end{tabular}
\caption{The notions of subtraction introduced in \cref{sec:subtraction} for $D$ an effective (integral) divisor on $X$ with fixed components~$E_i$.}
\label{tab:subtractions}
\end{table}

We are now prepared to present the proof of our final main result, \cref{MAIN:together}. The asymptotic quasilinearity of \cref{prop:asymptotic_subtraction} gives rise to an asymptotic piecewise quasipolynomial of global sections, which must hold everywhere because \cref{MAIN:contract} fixes global sections to be quasipolynomial on each Mori chamber (intersected with $\Bigc(X)$ in the Calabi--Yau type case).
\begin{proof}[Proof of \cref{MAIN:together}]
    Fix a Mori chamber $\mathcal{C}$. We begin by arguing for the asymptotic quasipolynomiality of \cref{alg} on $\relint \mathcal{C}$. \cref{alg} computes global sections as
    \begin{equation}
        h^0(X, \mathcal{O}_X(D)) = \chi\Big(Y, \mathcal{O}_{Y}( \varphi_* D^\downarrow_{\mathrm{zms}} ) \Big)
    \end{equation}
    for $\varphi : X \dashrightarrow Y$ an SQM of $X$ satisfying $D^\downarrow_{\mathrm{zms}} \in \varphi^* (\Nef(Y) \cap \Eff(Y))$, as long as $D^\downarrow_{\mathrm{zms}}$ is big (or $X$ is Fano type). By \cref{prop:asymptotic_subtraction}, asymptotically $D \mapsto D^\downarrow_{\mathrm{zms}}$ is quasilinear on $\relint \mathcal{C}$ and has image in $\Bigc(X)$. Hence, asymptotically, for any coset $H$ of $\Lambda_\mathcal{C}$, global sections on $\relint \mathcal{C}$ are the composition of a linear map and a quasipolynomial Euler characteristic. Collectively this yields a quasipolynomial $Q_\text{sub}$ for divisor classes $D \in \mathcal{C}$. In particular, letting $\varphi_H : X \dashrightarrow X_H$ denote the SQM from \cref{lem:movable_zariski_sub_SQM} associated to $H$, the lattice of $Q_\text{sub}$ on $\relint \mathcal{C}$ is
    \begin{equation}
        \Lambda_\text{sub} = \Lambda_\mathcal{C} \bigcap \big( \bigcap_H P_\mathcal{C}^{-1} \varphi_H^* \Pic(X_H) \big)~.
    \end{equation}
    
    Now we argue why this asymptotic formula holds everywhere in $\mathcal{C}$ (understood throughout this proof as $\mathcal{C} \cap \Bigc(X)$ when $X$ is Calabi--Yau type). By \cref{MAIN:contract}, we know that $h^0(X, -)$ is described everywhere on $\mathcal{C}$ by some quasipolynomial $Q_\text{con}$ with lattice $\Lambda_\text{con}$. On any coset $J$ of $\Lambda_\text{sub} \cap \Lambda_\text{con}$, $Q_\text{sub}$ and $Q_\text{con}$ are both polynomials. Moreover, for $S$ as in \cref{prop:asymptotic_subtraction}, $Q_\text{sub} = h^0(X, -) = Q_\text{con}$ on all divisor classes in $(S + \relint \mathcal{C}) \cap J$. Because $\mathcal{C}$ is full dimensional, these polynomials must consequently coincide. Because $J$ was arbitrary, we achieve $Q_\text{sub} = Q_\text{con}$, so in particular $Q_\text{sub} = h^0(X, -)$ everywhere on $\mathcal{C}$.
\end{proof}

Having now formally demonstrated the existence of piecewise quasipolynomial formulae, we conclude this section with a brief discussion on how such formulae can be computed in practice. Given the Mori chamber data of \cref{def:birational_data}, for every Mori chamber $\mathcal{C}$ and every coset $H$ of $\Lambda_{\mathcal{C}}$, it suffices to understand the quasilinear map that is asymptotically $D \mapsto D^\downarrow_{\mathrm{zms}}$. To do this, following the proof of \cref{lem:movable_zariski_sub_SQM}, we set $T_1 = R_H$ and from here define $-T_{i+1} = P_{\mathbb{Z}, \mathcal{C}_{i}}(-T_{i})$ for $\mathcal{C}_i$ a Mori chamber of \cref{lem:corrected_shifted_face} which asymptotically contains $-T_i + \mathcal{C} \cap \Mov(X)$. Such a chamber can be computed for example by letting $P \in \relint(\mathcal{C} \cap \Mov(X))$ and identifying chambers containing $P - \epsilon T_i$ for increasingly small $\epsilon > 0$ until one also contains $P$. Repeating this process gives a sequence $T_1, T_2, \dots$ which terminates with some $T_N$ obeying $T_N = T_{N+1}$, at which point we know that on $\mathcal{C} \cap H$, $D^{\downarrow}_{\mathcal{C}} = P_{\mathcal{C}}(D) - T_N$. Applying the logic of \cref{MAIN:together}, we conclude
\begin{equation}
    \label{eq:systematic_recipe}
    h^0(X, \mathcal{O}_X(D)) = \chi\Big(Y, \mathcal{O}_{Y}(\varphi_*(P_{\mathcal{C}}(D) - T_N))\Big)
\end{equation}
for $\varphi : X \dashrightarrow Y$ any SQM such that $P_{\mathcal{C}}(D) - T_N \in \varphi^* (\Nef(Y) \cap \Eff(Y))$. Combining the results for the different cosets of $\Lambda_{\mathcal{C}}$ gives a quasipolynomial formula for $h^0(X, \mathcal{O}_X(D))$ on $\mathcal{C}$, and combining the results for the different Mori chambers gives the global piecewise quasipolynomial formula.

Our examples of \cref{sec:examples} will be basic enough that we can just comprehensively perform \cref{alg} for all relevant divisors and witness the formulae emerge as the asymptotic behavior of the algorithm. In particular, we will not employ the approach just described for directly computing the formulae. However, we will make some brief comments on how the examples agree with results just discussed, such as \cref{prop:asymptotic_subtraction}.

\begin{remark}
    From the perspective of the contraction method, the lattice associated to the global sections quasipolynomial on a $D$-minimal model $X_\mathcal{C}$ --- i.e., the lattice associated to the quasipolynomial Euler characteristic for nef divisors on $X_\mathcal{C}$ --- should be no smaller than $\Pic(X_\mathcal{C})$, by \cref{cor:euler_char_is_quasipoly}. Pulling this back to $X$ and noting the irrelevance of exceptional divisors, the lattice for the associated Mori chamber should be no smaller than $L_\mathcal{C} := f_\mathcal{C}^*\Pic(X_\mathcal{C}) + \sum_i \mathbb{Z} E_i^\mathcal{C}$. 
    
    Now consider the perspective of the subtraction method --- for simplicity, let the SQMs of $X$ be smooth. Because global sections are achieved as a composition of the quasilinear Zariski movable subtraction, with lattice $\Lambda_\mathcal{C}$ by \cref{prop:integerZariski_quasilinear}, and the polynomial Euler characteristic on some SQM, the associated lattice of the global sections quasipolynomial is again just $\Lambda_\mathcal{C}$. This lattice contains $L_\mathcal{C}$, meaning the two methods are consistent. In particular, if $f_\mathcal{C}^*\Pic(X_\mathcal{C})$ is not saturated in $\Cl(X)$, then $\Lambda_\mathcal{C}$ is strictly larger than $L_\mathcal{C}$. Pushing back down to $X_\mathcal{C}$, we find that the subtraction method actually ensures that the sublattice of $\Cl(X_\mathcal{C})$ on which global sections/the Euler characteristic are polynomial is larger than just $\Pic(X_\mathcal{C})$ --- some non-Cartier, Weil divisors are included as well. It would be interesting to additionally understand the case of singular SQMs, as then the subtraction lattice depends on the Picard groups of the relevant SQMs.
\end{remark}

\begin{remark}
    In the subtraction approach described here, we achieve integrality by rounding up the functions $\lambda^\mathcal{C}_i$. It is worth mentioning the alternative --- rounding them down. When some $\lambda^\mathcal{C}_i(D)$ is non-integral, $D - \lfloor \lambda^\mathcal{C}_i(D) \rfloor E_i$ is then still in $\mathcal{C}$, and in particular cannot be movable, recalling the notation of \cref{sec:zariski}, if one sets
    \begin{equation}
        \Delta' = \sum_i \left( \lambda^\mathcal{C}_i(D) - \lfloor \lambda^\mathcal{C}_i(D) \rfloor \right) E^\mathcal{C}_i,
    \end{equation}
    and $(X, \Delta')$ is a klt pair, then in spite of the fact that $P_{\mathcal{C}} + \Delta'$ wouldn't be movable, the Kawamata--Viehweg theorem would still apply to it, as $P_{\mathcal{C}} + \Delta - (K_X + \Delta')$ would be a movable $\mathbb{Q}$-divisor. It would be interesting to understand the most general setting in which such $\mathbb{Q}$-divisors $\Delta'$ --- linear combinations of exceptional divisors of a fixed birational morphism with coefficients in $[0, 1)$ --- are guaranteed to give rise to klt pairs. For example, it would suffice for them to have normal crossing singularities (see, e.g., \cite[Cor. 3.12]{kollar1997singularities}).
\end{remark}

\begin{remark}
    In many cases, one would like to not only compute the number of global sections of a given divisor but also the ring structure of these sections --- i.e., the structure of the Cox ring more generally. While our results are largely unhelpful for this endeavor, we can make the following elementary observation. If an effective divisor $D$ satisfies $D^\downarrow = D - \sum_i a_i E_i$ for $E_i$ exceptional divisors of various birational maps, and we let $s_i$ denote the unique global section of $E_i$, then all global sections of $D$ are of the form $s \prod_i s_i^{a_i}$ for $s$ a global section of $D^\downarrow$. More generally, it would be worthwhile to understand the hypotheses under which the Cox ring can be reconstructed (or guessed) from just the knowledge of its dimension in various degrees.
\end{remark}

\medskip
\subsection{The Nef and Non-big Case}
\label{sec:non-big}

To apply Kawamata--Viehweg vanishing, we require that the sum $D - (K_X + \Delta)$ be big, where $D$ is either the output of the subtraction method on a suitable SQM or the output of a pushforward to a $D$-minimal model in the contraction method, and $(X,\Delta)$ is a klt pair (e.g., for Calabi--Yau or Fano type varieties, $\Delta$ can be as in \cref{def:CY_fano_type}). For some Mori dream spaces and Calabi--Yau type varieties it is the case that this sum is not big --- that is, it may lie on the boundary of the closure of the effective cone. This is a reason why our results for Calabi--Yau type varieties have been posed only for the big cone. We do not have a systematic treatment for this case, but in this subsection we will briefly comment on one strategy which can apply. 

If the relevant non-big divisor $D$ is semiample, then it defines a non-trivial fibration $\psi : X \to B$ for some base $B$ of strictly smaller dimension than $X$, which we denote the Iitaka fibration. It is a standard result that the global sections of any $\psi^* D$ for any integral $D \in \Nef(B)$ are merely given by the global sections of $D$. This can be understood as an explanation for why $\psi^* D$ isn't big --- the global sections of its integral multiples scale asymptotically as a polynomial of degree at most $\dim B < \dim X$.
\begin{proposition}[{\negthinspace\negthinspace\negthinspace\cite[Lem. 2.1.13]{Lazarsfeld}}]
    \label{prop:nonbig_reduction}
    Let $X$ be a normal, reduced, and irreducible variety with a semiample, non-big Cartier divisor $D$. Let
    \begin{equation*}
        \psi \colon X \longrightarrow B := \Proj\!\Big(\bigoplus_{m\ge 0} H^0(X,\mathcal O_{X}(mD))\Big)
    \end{equation*}
    denote the map defined by $D$, which is a morphism by the semiamplitude of $D$. If $D = \psi^* A$ for a line bundle $A$ on $B$, then
    \begin{equation*}
        h^0\bigl(X,\mathcal O_X(D)\bigr) \;=\; h^0\bigl(B,\mathcal O_B(A)\bigr).
    \end{equation*}
\end{proposition}
One can imagine computing global sections on the base by again applying the methods of this paper --- for example, many Calabi--Yau varieties admit fibrations over Fano type bases. We emphasize that semiamplitude is a subtle property: for example, if $X$ is a Calabi--Yau threefold, then if $D$ is nef and effective, it is semiample, by \cite[Th. 2.1]{oguiso1993algebraic}, but it is unknown if semiamplitude follows from nefness in general. For a review of semiamplitude for Calabi--Yau varieties, see \cite{lazic2016morrison}. On the other hand, if $X$ is a Mori dream space, then if $D$ is nef, it is semiample. Additionally, the image of the nef integral classes of $B$ under $\psi^*$ need not be saturated in the Picard group of $X$ (or in the class group of $X$, if $\Cl(X)/\Pic(X)$ is non-trivial). In such cases, only a finite index sublattice of nef, non-big classes will be computable through this method.

\subsection{Calabi--Yau Hypersurfaces}\label{sec:cy_hyp}

Let us now consider a special case of Calabi--Yau varieties for which more is possible. In particular, an important class of Calabi--Yau varieties arise as hypersurfaces in toric varieties. The structure of the Koszul sequence used for computing hypersurface line bundle cohomology is conveniently compatible with Kawamata--Viehweg vanishing. 
\begin{proposition}
    \label{prop:toric_hyp_cy}
    Let $X$ be a quasismooth Calabi--Yau hypersurface in a simplicial projective toric variety $V$ such that $\pi : \mathrm{Cl}(V) \to \mathrm{Cl}(X)$ is an isomorphism. Then, identifying $\Cl(V) \cong \Cl(X)$, for $D \in \Bigc(V) \cap \Nef(V)$ we have
    \begin{equation}
        h^0(X, \mathcal{O}_{X}(D)) 
        = h^0(V, \mathcal{O}_{V}(D)) - h^0(V, \mathcal{O}_{V}(D + K_{V})) 
        = \chi(V, \mathcal{O}_{V}(D)) - \chi(V, \mathcal{O}_{V}(D + K_{V}))
    \end{equation}
    In particular, $h^0\bigl(X,\mathcal O_X(D)\bigr)$ is quasipolynomial for such $D$.
\end{proposition}
\begin{proof}
    Let $D \in \Bigc(V) \cap \Nef(V)$. By the standard Koszul sequence,
    \begin{equation}
        \begin{aligned}
            0 \to H^0(V, \mathcal{O}_{V}(D + K_{V})) &\stackrel{A}{\to} H^0(V, \mathcal{O}_{V}(D)) \to H^0(X, \mathcal{O}_{X}(\pi(D))) \\
            \to H^1(V, \mathcal{O}_{V}(D + K_{V})) &\stackrel{B}{\to} H^1(V, \mathcal{O}_{V}(D)) \to \dots
        \end{aligned}
    \end{equation}
    we can conclude that $H^0(X, \mathcal{O}_{X}(\pi(D))) = \mathrm{coker} \, A \oplus \mathrm{ker} \, B$. By Kawamata--Viehweg (with $\Delta = 0$), $H^1(V, \mathcal{O}_{V}(D + K_{V})) = 0$. Thus, only the $\mathrm{coker} \, A$ contribution survives, directly yielding the first equality. The second equality follows from the vanishing of higher cohomology for $D + K$ on $V$ (as already shown) as well as $D$ on $V$ (e.g., by applying \cref{prop:toric_is_fano_type} and \cref{cor:MAIN-VANISHING}). Recalling \cref{cor:euler_char_is_quasipoly}, it suffices to note that a sum of quasipolynomials is quasipolynomial.
\end{proof}
\begin{corollary}
    \label{cor:toric_hyp_cy_hilbert_function}
    Let $X$ be a quasismooth Calabi--Yau hypersurface in a simplicial projective toric variety $V$ given by the zero locus of $g \in H^0(V, -K_V)$. If $\pi : \mathrm{Cl}(V) \to \mathrm{Cl}(X)$ is an isomorphism, then, identifying $\Cl(V) \cong \Cl(X)$, for $D \in \Bigc(V) \cap \Nef(V)$ we have that $h^0(X, \mathcal{O}_{X}(-))$ agrees with the Hilbert function of $\mathrm{Cox}(V) / \langle g \rangle$.
\end{corollary}

\subsection{Birational tomography from finite cohomological data}
\label{sec:birational_tomography}

The preceding results admit a stronger interpretation than the statement that birational geometry determines global sections. As captured in \cref{def:birational_data}, the data required by the subtraction approach are comparatively limited: one needs the Mori chamber decomposition of the movable cone, the prime exceptional divisors governing the chambers outside the moving cone and the Euler characteristics on the SQMs of $X$. In particular, one need not construct all $D$-minimal models, nor compute their Euler characteristics independently. Nevertheless, by \cref{MAIN:together}, the resulting formula for global sections agrees on every Mori chamber with the Euler-characteristic quasipolynomial of the corresponding $D$-minimal model. Thus the global-section formulae encode birational information which was not supplied explicitly in their construction. They retain information about the entire geography of $D$-minimal models, including numerical and discrete information attached to birational contractions. 

This observation suggests reading our results in the opposite direction. Rather than using the birational geometry of $X$ to compute global sections, one may ask how much of the birational geometry can be inferred from finitely many values of $[D]\mapsto h^0(X,\mathcal O_X(D))$. We refer informally to this inverse problem as \emph{birational tomography}. Our results constrain the form that the cohomology data can take. In dimension~$d$,  the values in a fixed birational regime are governed by a quasipolynomial of degree at most~$d$. One may therefore search the data for regions on which a single polynomial or quasipolynomial fit holds, and regard changes in the resulting formula as evidence for walls separating different Mori chambers. This experimental, data-driven approach to uncovering birational structure from cohomology computations has already been explored in a number of examples in~\cite{Brodie:2019dfx, Brodie:2020fiq, Brodie:2021nit}. 

Once a candidate region has been identified, its formula is determined by finitely many values. Indeed, if $\rho=\rank\Cl(X)$, a polynomial of total degree at most $d$ in $\rho$ variables has $\binom{\rho+d}{d}$ coefficients. Thus $\binom{\rho+d}{d}$ suitably chosen evaluations determine a polynomial constituent. If the data are given by a quasipolynomial expression with some non-trivial lattice, the same interpolation can be performed separately on the cosets of the lattice. This does not provide an a priori bound on the total number of computations needed to discover the complete piecewise quasipolynomial structure, since neither the walls nor the lattice are assumed known in advance, but it shows that each candidate constituent can be tested and reconstructed from finitely many values.

The resulting formulae can contain considerably more information than the location of candidate chambers. For example, if the fitted formula in a region is invariant in the direction of a divisor class $[E]$, this is evidence that $[E]$ is exceptional. More information can sometimes be extracted by comparing the formulae in neighbouring candidate regions. Suppose quasipolynomials $P$ and $Q$ govern two such regions, with $P$ associated to a region in the movable cone. If the data reveal a quasilinear map $\Phi$ such that $Q(D)=P(\Phi(D))$, then $\Phi$ is a candidate for the asymptotic integral subtraction map carrying divisors toward the movable cone. In such cases the cohomological data can suggest not only which divisor directions are exceptional, but also the piecewise integral transformation by which their fixed components are removed.

Fitted formulae also carry numerical information. Whenever the relevant Euler characteristic is described by Hirzebruch--Riemann--Roch, its terms recover intersection-theoretic data on the corresponding birational model. For example, on a smooth Calabi--Yau threefold,
$
\chi(X,\mathcal O_X(D)) = \frac{1}{6}D^3+\frac{1}{12}c_2(X)\cdot D,
$
so the cubic and linear terms, respectively, elucidate the cubic intersection form and the pairing with~$c_2(X)$. 

Finally, quasipolynomial behaviour detects information which is invisible in the real vector space $N^1(X)_{\mathbb R}$. From the subtraction perspective, this arises from quasipolynomiality of Euler characteristics of SQMs but also crucially because the integral Zariski decomposition is itself quasilinear --- i.e., affine linear only on cosets of a finite-index lattice $\Lambda_{\mathcal C}\subset\Cl(X)$, with the rounding corrections depending on the corresponding coset. From the contraction perspective, periodicity can reflect the quasipolynomial Euler characteristic of a singular $D$-minimal model, including finite-index phenomena associated with non-Cartier Weil divisor classes and torsion in the class group. The cohomological data can therefore reveal finite integral information that is lost upon passage to $N^1(X)_{\mathbb R}$, even before its precise geometric origin is identified.

There are, however, limits to what can be recovered from finite cohomological data alone. A change in the quasipolynomial governing $h^0$ provides evidence for a birational wall, but finitely many evaluations cannot by themselves certify the complete chamber structure: without an a priori bound on the number and location of chambers, an additional wall may lie outside the sampled region. Moreover, distinct Mori chambers may in principle carry the same quasipolynomial, in which case the wall between them would be invisible from section counts alone. Likewise, periodic behaviour can reveal finite-index structure without uniquely determining its geometric origin. Birational tomography should therefore be understood as a method for extracting and testing birational information from cohomological data, rather than as an unconditional reconstruction theorem.

We note that there is some precedent for this kind of tomography from finite integer data. For example, in the special setting of Calabi--Yau threefolds, the authors of~\cite{Gendler:2022ztv} used finite computations of Gopakumar--Vafa invariants to reconstruct the extended K\"ahler cone.

Against this broader data-driven perspective, some of the examples below serve as a proof-of-principle for the cohomological approach described here in concrete terms. We illustrate how finite collections of global-section computations will be used to infer candidate chamber walls, exceptional directions, contraction maps, intersection-theoretic data, and finite-index phenomena; the resulting predictions will then be verified independently from the birational geometry and the $D$-MMP.

\section{Examples}
\label{sec:examples}

We now exhibit the methods described in the previous sections in some examples. We will consider three smooth Calabi--Yau threefolds arising as hypersurfaces in toric varieties as examples of Calabi--Yau type varieties, as well as a toric variety as an example of a Fano type variety. In particular, we will demonstrate the contraction method (\cref{MAIN:contract}) with some help from \cref{prop:toric_hyp_cy}, and the subtraction method (\cref{MAIN:subtract}) via \cref{alg}. By \cref{MAIN:together}, it suffices to determine the asymptotic quasipolynomial form of \cref{alg}, as this piecewise function characterizes global sections everywhere. However, for the purpose of illustration, we will show in these examples how \cref{alg} can be more complicated for divisors not sufficiently far out in the Mori chamber and then we will verify \cref{MAIN:together} by showing that the asymptotic formulae still apply to such divisors. In general, for Calabi--Yau type varieties a formula is guaranteed to exist only for big divisors, but in the examples studied below we are able to successfully employ the methods of \cref{sec:non-big} in conjunction with our knowledge that $h^0(X, \mathcal{O}_X) = 1$ to achieve a formula for the entire effective cone.

For the forward application of our methods in the examples below, we determine the relevant Mori chamber decompositions independently. For toric varieties this is simply the secondary fan decomposition of the effective cone (see, e.g., \cite[Ch. 14-15]{cls}), which is well-understood and can be performed combinatorially, for example using the \texttt{regfans} Python package \cite{regfans}, as discussed in \cite{MacFadden:2025ssx}. The situation is more subtle for Calabi--Yau hypersurfaces $X$, but our examples will actually themselves be Mori dream spaces as well. This can be verified, for example, by using the methods of \cite{herrera2024cox} to directly compute the Cox ring of the anticanonical hypersurface. In particular, all three of our Calabi--Yau examples actually have hypersurface Cox rings: if $f$ is a generic anticanonical section of the associated toric fourfold $V$, then the Cox ring of the hypersurface is exactly $\mathrm{Cox}(V)/f$ (this is of course not true in general for toric hypersurfaces). In particular, this means that the embedding of the hypersurface is the one described in \cite[Prop. 2.11]{HuKeel}, which entails that the Mori chamber decomposition of the Calabi--Yau hypersurface is refined by that of the toric variety, so the former is readily computed from the latter. In the first two examples, where we also illustrate birational tomography, this independent determination of the birational geometry allows us to verify directly the predictions inferred from the cohomological data.

We also recall the alternative, but more conjectural approach for constructing the extended K\"ahler cone of a Calabi--Yau threefold --- i.e., the Mori chamber decomposition of the moving cone --- presented in \cite{Gendler:2022ztv}, which we noted in noted in \cref{sec:birational_tomography}. This method uses Gopakumar--Vafa invariants computable using the Python packages \texttt{CYTools} \cite{Demirtas:2022hqf} and \texttt{cygv} \cite{cygv}, which implement the methods of \cite{Demirtas:2023als}.

We emphasize that while the examples studied here are simple enough that birational tomography is unnecessary to compute birational geometric data, our examples nevertheless allow us to usefully explain precisely how birational geometric data can be reconstructed from finite global section data. This is particularly true for the first two examples, where we explicitly illustrate how birational information can be extracted from cohomology data. The first example also illustrates how divisors with base loci that are not stable (i.e., ``transient'' base loci) result in \cref{alg} requiring multiple iterations. Our second example in \cref{sec:ex2} features a divisorial contraction which induces torsion in the class group of the target; this torsion leaves a distinctive parity dependence in the global-section formulae and can itself be detected from the cohomological data. Our third example in \cref{sec:ex3} has Picard number three and represents an application of our methods to a variety with a more complex collection of birational contractions. Finally, for our fourth example in \cref{sec:ex_4} we do not present a comprehensive formula, but consider a particular subset of the Mori chambers on which the Zariski movable subtraction is more complicated than in the first three examples. 

\subsection{Example 1 --- Transient Base Locus} \label{sec:ex1}

Consider the reflexive polytope $\Delta^\circ_{25}$ in the Kreuzer--Skarke classification
of four-dimensional reflexive polytopes \cite{KreuzerSkarke} with points 
\begin{equation*}
    \begin{aligned}
        v_1 &= (-7,\,-2,\,-2,\,-2),\quad v_2 = (0,\,0,\,0,\,1),\quad v_3 = (0,\,0,\,1,\,0),\\
        v_4 & = (0,\,1,\,0,\,0), \quad \quad \quad \quad \, v_5 = (1,\,0,\,0,\,0),\quad v_6 = (-3,\,-1,\,-1,\,-1).
    \end{aligned}
\end{equation*}  
Let $V$ be the unique simplicial toric fourfold whose fan has rays generated by the above points. Because $\Delta^\circ_{25}$ is reflexive, a generic anticanonical
hypersurface $X \;\subset\; V$ is a smooth Calabi--Yau threefold \cite{Batyrev:1993oya, MacFadden:2025ssx}. 
$X$ has non-trivial Hodge numbers $(h^{1,1}, h^{2,1}) = (2, 122)$, with its Picard group being inherited from the class group of $V$. We fix a basis for $\Cl(X) \cong \Pic(X)$ by specifying the following class group grading for the prime torus-invariant divisors, ordered in the same way as the above points:
\begin{equation}
    \label{eq:ex1_charge_matrix}
    \begin{pmatrix}
        1 & 2 & 2 & 2 & 7 & 0 \\
        0 & 1 & 1 & 1 & 3 & 1 \\
    \end{pmatrix}~.
\end{equation}
We let $x_i$ denote the homogeneous coordinate associated to the $i$th ray, $\hat{D}_i$ the prime torus-invariant divisor given by $x_i = 0$, and $D_i = \hat{D}_i \cap X$. We see that $\Cl(X)\cong\Pic(X)=\mathbb Z[D_1]\oplus\mathbb Z[D_6]$.

\subsubsection{Birational tomography.}

Before explicitly determining the birational geometry of~$X$, let us ask how much of it can be inferred if one only had access to a finite collection of cohomology computations. \cref{fig:25} displays values of $h^0\bigl(X,\mathcal O_X(aD_1+bD_6)\bigr)$ at lattice points in the effective cone. Since $X$ is a threefold, our general results lead us to search for piecewise quasipolynomial formulae of degree three. 

\begin{figure}
    \centering
    \includegraphics[width=.58\linewidth]{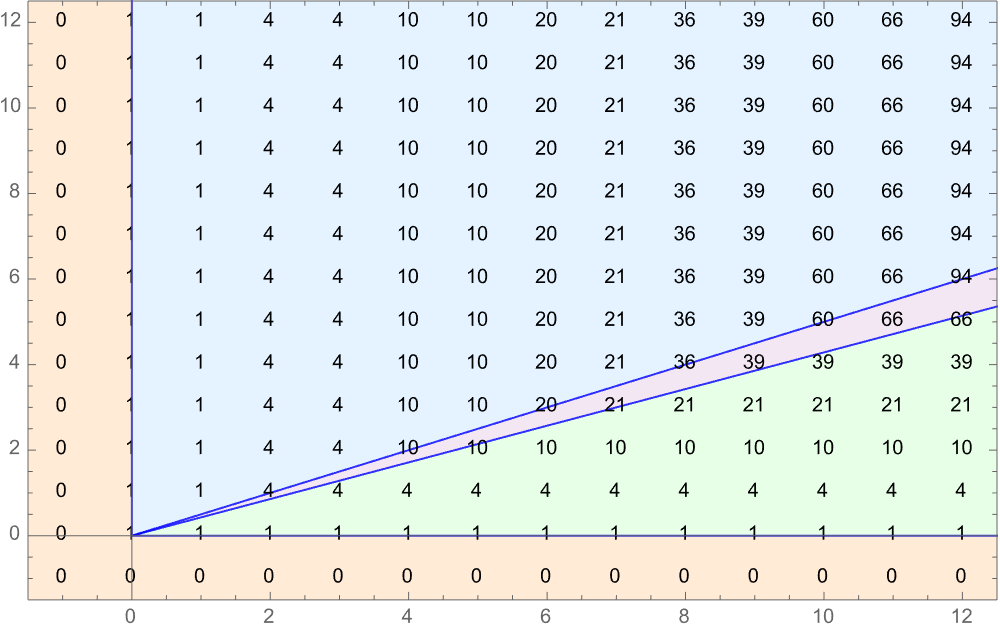}
    \caption{Cohomology data and chamber decomposition of ${\rm Eff}(X)$ for the smooth Calabi--Yau threefold associated with the reflexive polytope $\Delta_{25}^\circ$ in the Kreuzer--Skarke classification of four-dimensional reflexive polytopes. The numbers count global sections of line bundles on $X$.}
    \label{fig:25}
\end{figure}

The data separate naturally into three regimes. In the central region they are fit by the cubic polynomial
\begin{equation*}
P(a,b) = \frac{3}{2}a^{3} -\frac{21}{2}a^{2}b +\frac{49}{2}ab^{2} -\frac{1}{2}a -\frac{56}{3}b^{3} +\frac{14}{3}b .
\end{equation*}
In one exterior region, the data are instead fit by
\begin{equation*}
Q_{\ell}(a) = \frac{1}{24}a^3 +\frac{7}{16}(-1)^a a +\frac{67}{48}a,
\end{equation*}
while in the other they are fit by
\begin{equation*}
Q_r(b) = \frac{7}{18}b^3 +\frac{7}{2}b +\frac{1}{9}\bigl((b+1)\bmod 3-1\bigr).
\end{equation*}
While first formula admits no invariant directions and is polynomial, the two exterior formulae each are independent of a particular direction in the class group and are quasipolynomial with periods two and three, respectively. The change of formula indicates two candidate walls in the effective cone along the rays generated by $[D_2]=2[D_1]+[D_6]$, and $[D_5]=7[D_1]+3[D_6]$, that is $b=a/2$ and, respectively, $b=3a/7$. 

The invariant directions in the two exterior regions have a direct birational interpretation. In the left-hand region, $Q_\ell(a)$ is independent of $b$, indicating that $[D_6]$ is a fixed exceptional direction; its contribution can be subtracted without changing $h^0$, equivalently suggesting a contraction of $D_6$ on the associated $D$-minimal model. The data suggest that asymptotically a single subtraction removes the $D_6$-component. Similarly, in the right-hand region $Q_r(b)$ is independent of $a$, indicating that $[D_1]$ is the corresponding fixed exceptional direction and that a single asymptotic subtraction is needed. 

The periodic coefficients (i.e., quasipolynomiality) in the two exterior formulae indicate index-two and index-three lattice phenomena in the two exterior regions. From the subtraction perspective, such periodicity can arise from the relevant integral Zariski decomposition being linear only on cosets of a finite-index sublattice of $\Cl(X)$. From the contraction perspective, the same phenomenon can reflect non-Cartier divisor classes of finite index on the corresponding singular $D$-minimal model, and hence periodic corrections to its Euler characteristic.

The polynomial parts also carry intersection-theoretic information. The cubic and linear parts of $P(a,b)$ encode, respectively, the cubic intersection form and the pairing with $c_2(X)$ in the basis $\{[D_1],[D_6]\}$. By contrast, the cubic parts of the two exterior formulae depend on a single variable, which is consistent with the corresponding divisors being pushed forward to Picard-rank-one $D$-minimal models, where the Euler characteristic depends on a single numerical divisor coordinate.
We now determine the birational geometry directly and verify the predictions made from this finite cohomological data.

\subsubsection{Birational geometry and contraction.}
The class group $\Cl(X)$ is generated by $\{[D_1], [D_6]\}$; these basis elements turn out to be precisely the exceptional divisors of the divisorial contractions associated to~$X$.
$X$ is a Mori dream space and the effective cone of $X$ admits the Mori chamber decomposition shown in
\cref{fig:25}.
The small birational class of $X$ consists of just $X$, so $\Mov(X) = \Nef(X) = \mathrm{Cone}([D_2], [D_5])$, shown in purple, and there are two additional Mori chambers $Z_1  = \mathrm{Cone}([D_2], [D_6])$ and $Z_2 = \mathrm{Cone}([D_1], [D_5])$, shown in blue and, respectively, green. The ambient variety $V$ enjoys two divisorial contractions given by the blowdowns of the toric divisors $\hat{D}_6$ and $\hat{D}_1$, respectively:
\begin{equation*}
f^\ell : V \dashrightarrow \mathbb{P}_{1\,2\,2\,2\,7}, \qquad 
f^r : V \dashrightarrow \mathbb{P}_{1\,1\,1\,1\,3},
\end{equation*}
which, upon restriction to the hypersurface, induce contractions of $D_6$ and $D_1$, respectively:
\begin{equation*}
f^\ell : X \dashrightarrow Y_1, \qquad
f^r : X \dashrightarrow Y_2~.
\end{equation*}
These induce the two Mori chambers outside of $\Mov(X)$ in ${\rm Eff}(X)$. In particular, $Y_1$ and $Y_2$ are the anticanonical hypersurfaces in $\mathbb{P}_{1\,2\,2\,2\,7}$ and $\mathbb{P}_{1\,1\,1\,1\,3}$, respectively. 

We have identified the Mori chamber decomposition of $\Eff(X)$ as well as the birational models associated to each chamber --- namely, $X$ itself along with $Y_1$ and $Y_2$. Consequently, every effective divisor on~$X$ pushes forward to a nef divisor on one of these models. Thus we can compute $h^0(X, \mathcal{O}_X(D))$ for any effective $D$ by performing the computation on a $D$-minimal model. However, neither $Y_1$ nor $Y_2$ is smooth:
in fact $Y_2 \subset \mathbb{P}_{1\,1\,1\,1\,3}$ carries an isolated $\mathbb{Z}_3$ quotient singularity and $Y_1 \subset \mathbb{P}_{1\,2\,2\,2\,7}$ contains a curve of $\mathbb{Z}_2$ singularities, coming from a toric surface of $\mathbb{Z}_2$ singularities in $\mathbb{P}_{1\,2\,2\,2\,7}$ that meets $Y_1$ generically. This obstructs us from directly applying the smooth Hirzebruch--Riemann--Roch formula on either variety. Nevertheless, we have realized $Y_1$ and $Y_2$ as hypersurfaces in toric varieties, meaning we can compute the global sections of their big and nef divisors from the global sections of divisors on the ambient toric fourfold, by \cref{prop:toric_hyp_cy}. The ambient toric fourfolds in this case are also singular, meaning Hirzebruch--Riemann--Roch remains unavailable, but a convenient approach is furnished by \cref{cor:toric_hyp_cy_hilbert_function}. In particular, the hypersurface divisor classes descending from big nef classes are simply all non-trivial effective classes, and we can conclude from the corollary that their Hilbert series $HS(Y_i;t)$ are
\begin{equation}
    \begin{aligned}
        HS(Y_1;t) &= \frac{1 - t^{14}}{(1-t)(1-t^2)^3(1-t^7)}~, \\
        HS(Y_2;t) &= \frac{1 - t^7}{(1-t)^4(1-t^3)}~. 
    \end{aligned}
\end{equation}
A formula for the global sections of such effective divisors on these varieties can then be straightforwardly derived. In particular, this also computes the Euler characteristic for non-trivial effective classes $nH$, where $H$ is inherited from the weighted projective space hyperplane class and $n > 0$, as such classes are big and nef:
\begin{equation}
    \begin{aligned}
        \chi(Y_1, nH) = h^0(Y_1, \mathcal{O}_{Y_1}(nH)) &= \frac{1}{n!}\frac{d^n}{dt^n} HS(Y_1)\Big|_{t = 0} = \frac{1}{24}n^3 + \frac{7}{16} (-1)^{n} n + \frac{67}{48}n \\
        \chi(Y_2, nH) = h^0(Y_2, \mathcal{O}_{Y_2}(nH)) &= \frac{1}{n!}\frac{d^n}{dt^n} HS(Y_2)\Big|_{t = 0} = \frac{7}{18}n^3 + \frac{7}{2}n + \frac{1}{9}\Big((n + 1) \bmod 3 - 1\Big)
    \end{aligned}
\end{equation}
The divisors that push forward to the trivial class on $Y_1$ or $Y_2$ are not captured by this formula but evidently have one global section. Here, and for the remainder of this section, $\partial^\ell$ and $\partial^r$ refer to the left and right boundaries of a closed two-dimensional cone, respectively.
\begin{equation}
    \begin{aligned}
        h^0(X, \mathcal{O}_X(D)) 
        \; &= \;
        \begin{cases}
            h^0(X, \mathcal{O}_X)  & [D] \in \partial^\ell Z_1 \\
            \chi(Y_1, \mathcal{O}_{Y_1}(f^{\ell}_* D)) & [D] \in Z_1 \setminus \partial^\ell Z_1 \\
            \chi(X, \mathcal{O}_X(D)) & [D] \in \Nef(X) \setminus \{0\} \\
            \chi(Y_2, \mathcal{O}_{Y_2}(f^{r}_*D)) & [D] \in Z_2 \setminus \partial^r Z_2 \\
            h^0(X, \mathcal{O}_X)  & [D] \in \partial^r Z_2
        \end{cases}
    \end{aligned}
\end{equation}
It then suffices to note that $f^\ell_* (a[D_1] + b[D_6]) = a [H]$ and $f^r_* (a[D_1] + b[D_6]) = b [H]$ --- substitution then reproduces the formulae above.

From the contraction method, we see that the quasipolynomiality of the formulae arises from the structure of the Hilbert series of $Y_1$ and $Y_2$. As noted above, $Y_1$ and $Y_2$ admit $\mathbb{Z}_2$ and $\mathbb{Z}_3$ quotient singularities, and quotient singularities can give rise to periodic contributions to the Euler characteristic. Thus the periods two and three observed in the cohomological data admit a direct interpretation in terms of the singularities of the corresponding $D$-minimal models.

The subtraction method below will provide a complementary interpretation of the same periodicity. There, the relevant lattices $\Lambda_{Z_1}$ and $\Lambda_{Z_2}$ from \cref{prop:integerZariski_quasilinear} have indices two and three in $\Cl(X)$, respectively, so the integral Zariski decomposition depends on the corresponding residue classes. Thus the arithmetic structure first detected from finite cohomological data is reflected simultaneously in the lattice geometry of subtraction and in the singular birational models produced by contraction.

\subsubsection{Subtraction and \cref{alg}.} We now recover the same global-section formulae using subtraction, which additionally reveals the transient behaviour of \cref{alg}.  No subtraction is required for any divisor in $\Nef(X)$ and \cref{cor:MAIN-VANISHING} is immediately applicable. For divisors in $Z_1$, the relevant birational map is $f^\ell$, for which there is a single exceptional divisor, $E^{Z_1}_1 = D_6$. We recall that we can compute the Zariski decomposition by decomposing divisors in this chamber as a non-negative linear combination of exceptional classes and generators of the saturation of $f^{\ell *}\Pic(Y_1)$ in $\Cl(X)$. Here the exceptional class is $[D_6]$, while $f^{\ell *}\Pic(Y_1)=\mathbb Z[D_2]$ is already saturated in $\Cl(X)$. This determines the integral Zariski decomposition following \cref{def:integral-zariski}. In particular, $\lambda^{Z_1}_1(a[D_1] + b[D_6]) = b - a/2$, so in $Z_1$, 
\begin{equation}
    P_{\mathbb{Z},Z_1}([D]= a[D_1] + b[D_6]) = [D] - \lceil b - a/2 \rceil [D_6] = a [D_1] + \lfloor a/2 \rfloor [D_6]~.
\end{equation}
For $a = 0, 2, 4$ or $a \geq 6$, this is movable and the algorithm terminates. However, for $a = 1, 3, 5$, we have that $P_{\mathbb{Z},Z_1}([D]) \in Z_2$ and another iteration is required. For such divisors, the divisorial base locus is larger than the stable base locus --- there is an ``unstable'' or ``transient'' base locus --- such that removing the stable base locus does not result in a divisor in the movable cone: $\Mov(X)$ is ``passed through'' by subtraction. 

For divisors in $Z_2$, the relevant birational map is $f^r$, for which there is $E^{Z_2}_1 = D_1$. We can read off that 
$\lambda^{Z_2}_1(a[D_1] + b[D_6]) = a - 7b/3$ so in $Z_2$, 
\begin{equation}
    P_{\mathbb{Z},Z_2}([D]= a[D_1] + b[D_6]) = [D] - \lceil a - 7b/3 \rceil [D_1] = \lfloor 7b/3 \rfloor [D_1] + b [D_6]
\end{equation}
For all $[D] \in Z_2$, then, $P_{\mathbb{Z},Z_2}([D]) \in \Mov(X)$ and the algorithm terminates in one step in this chamber. 

Putting everything together, we see that the algorithm requires $0$, $1$, or $2$ steps, and requires $2$ steps only for the divisors in the set $B = \{ [D] = a[D_1] + b[D_6] \in \Eff(X) \; | \; a \in \{1, 3, 5\}, 2b > a\}$.    
This behaviour is illustrated in \cref{fig:ex1_subtraction}. Divisors in the movable cone require no subtraction, while a single Zariski subtraction suffices for divisors in $Z_2$ and for all but finitely many slices of $Z_1$. The exceptional set $B\subset Z_1$ consists precisely of those divisor classes for which the first subtraction crosses the movable cone and lands in $Z_2$, so that a second subtraction is required. Thus the transient behaviour is confined to finitely many lattice slices near the boundary of $Z_1$.  

\begin{figure}
    \centering
    \includegraphics[width=.58\linewidth]{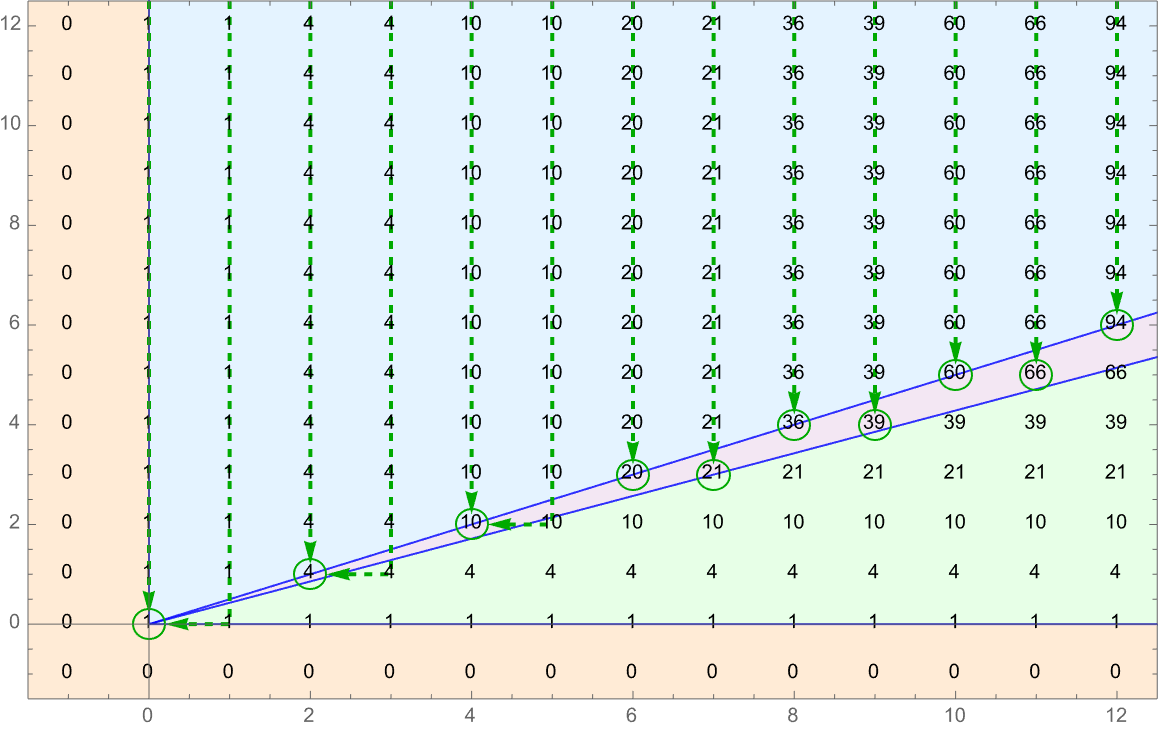}
    \caption{Behaviour of the Zariski subtraction algorithm on $\Eff(X)$. Divisors in $\Mov(X)=\Nef(X)$ (the purple region) require no subtraction. Divisors in $Z_2$ (the green region) are carried into $\Mov(X)$ after one subtraction, and the same holds for divisors in $Z_1$ (the blue region) outside the set
$B=\{a[D_1]+b[D_6]\in\Eff(X)\mid a\in\{1,3,5\},\;2b>a\}$. For classes in $B$, the first subtraction removes the stable fixed component but carries the divisor across $\Mov(X)$ into $Z_2$; a second subtraction is therefore required to reach the movable cone. This transient region accounts for the non-asymptotic behaviour of \cref{alg} in this example.}
    \label{fig:ex1_subtraction}
\end{figure}

As expected for Calabi--Yau varieties, for the non-big divisor classes --- the effective classes $a[D_1] + b[D_6]$ satisfying $a = 0$ or $b = 0$ --- the subtraction procedure culminates in a nef, non-big divisor outside of the scope of Kawamata--Viehweg. Additionally, the big divisors of the form $[D_1] + b[D_6]$ also subtract to a non-big class --- the existence of such examples requires the ``big Zariski movable subtraction'' hypothesis in \cref{MAIN:subtract}. In this example we are fortunate that the subtraction results in the trivial class, whose sections we know --- in general, big classes can subtract to non-trivial non-big classes, where the vanishing theorem does not apply and we would be forced to apply the results of \cref{sec:non-big} or attempt some other ad hoc strategy.

We have completely characterized how the subtraction algorithm proceeds for every class in the effective cone. We now note that for $[D] = a[D_1] + b[D_6] \in \Eff(X)$, 
\begin{equation}
    \chi(X, \mathcal{O}_X(D)) = \frac{3}{2}a^{3} - \frac{21}{2}a^{2} b + \frac{49}{2}a b^{2} - \frac{1}{2}a - \frac{56}{3}b^{3} + \frac{14}{3}b,
\end{equation}
which allows us to summarize the application of the algorithm as follows.
\begin{equation}
    \label{eq:ex_1}
    \begin{aligned}
        h^0(X, \mathcal{O}_X(D))
        \; &= \;
        \begin{cases}
            h^0(X, \mathcal{O}_X) & [D] \in \partial^\ell Z_1 \\
            \chi(X, \mathcal{O}_X(a D_1 + \lfloor a/2 \rfloor D_6)) & [D] \in Z_1 \setminus (B \cup \partial^\ell Z_1) \\
            h^0(X, \mathcal{O}_X) & [D] \in B, \; a = 1 \\
            \chi(X, \mathcal{O}_X(\lfloor 7 \lfloor a/2 \rfloor/3 \rfloor D_1 + \lfloor a/2 \rfloor D_6)) & [D] \in B, \; a \in \{3, 5\} \\
            \chi(X, \mathcal{O}_X(D)) & [D] \in \Nef(X) \setminus \{0\} \\
            \chi(X, \mathcal{O}_X(\lfloor 7b/3 \rfloor D_1 + b D_6)) & [D] \in Z_2 \setminus \partial^r Z_2 \\
            h^0(X, \mathcal{O}_X) & [D] \in \partial^r Z_2 \\
        \end{cases} \\
        &= \;
        \begin{cases}
            1 & (a,b) \in \partial^\ell Z_1 \\
            \frac{1}{24}a^3 + \frac{7}{16}(-1)^a\, a + \frac{67}{48}a & (a, b) \in Z_1 \setminus (B \cup \partial^\ell Z_1) \\
            1 & (a,b) \in B, \; a = 1 \\
            \frac{7}{18}\lfloor a/2\rfloor^{3} + \frac{7}{2}\lfloor a/2\rfloor + \frac{1}{9}\bigl((\lfloor a/2\rfloor + 1) \bmod 3 - 1\bigr) & (a,b) \in B, \; a \in \{3, 5\} \\
            \frac{3}{2}a^{3} - \frac{21}{2}a^{2} b + \frac{49}{2}a b^{2} - \frac{1}{2}a - \frac{56}{3}b^{3} + \frac{14}{3}b & (a, b) \in \Nef(X) \setminus \{0\} \\
            \frac{7}{18}b^{3} + \frac{7}{2}b + \frac{1}{9}\bigl((b+1) \bmod 3 - 1\bigr) & (a, b) \in Z_2 \setminus \partial^r Z_2 \\
            1 & (a,b) \in \partial^r Z_2 \\
        \end{cases}
    \end{aligned}
\end{equation}
The $a = 1$ case for $[D] \in B$ must be handled separately because its Zariski movable subtraction isn't big, as previously mentioned.

There is a useful connection between this calculation and the birational tomography discussed above. Recall the polynomial $P(a,b)$ describing the central chamber and the exterior quasipolynomials $Q_\ell(a)$ and $Q_r(b)$. The subtraction maps derived above are already encoded in these formulae: direct substitution gives
\begin{equation}
Q_\ell(a) = P\left(a,\left\lfloor\frac{a}{2}\right\rfloor\right),
\qquad
Q_r(b) = P\left(\left\lfloor\frac{7b}{3}\right\rfloor,b\right).
\end{equation}
Thus the exterior quasipolynomials are obtained from the central polynomial by the integral maps
\begin{equation}
(a,b)\longmapsto \left(a,\left\lfloor\frac{a}{2}\right\rfloor\right),
\qquad
(a,b)\longmapsto \left(\left\lfloor\frac{7b}{3}\right\rfloor,b\right),
\end{equation}
which are precisely the subtraction maps derived geometrically above. On $Z_2$, and on $Z_1\setminus B$, these maps carry a divisor into the movable cone after a single subtraction. In this sense, the finite cohomological data detect not only the exceptional directions, but also the integral subtraction maps governing the asymptotic behaviour of the algorithm.

\subsubsection{Subtraction and Asymptotic/Global Formulae.} We know from \cref{MAIN:together} that \cref{alg} results in asymptotic formulae on the Mori chambers. This can be directly read off from the previous section. The quasilinear maps $D \mapsto D^\downarrow_\mathcal{C}$ from \cref{def:quasilinear_zariski_subtraction} that asymptotically give the Zariski movable subtraction, for $[D] = a [D_1] + b [D_6]$, are given as follows:
\begin{equation}
    \begin{aligned}
        [D^\downarrow_{Z_1}] &= a [D_1] + \lfloor a/2 \rfloor [D_6]~, \qquad
        [D^\downarrow_{Z_2}] = \lfloor 7b/3 \rfloor [D_1] + b [D_6]~.
    \end{aligned}
\end{equation}
We note that $[D^\downarrow_{Z_1}]$ agrees with $[D^\downarrow_{\mathrm{zms}}]$ in the asymptotic region $Z_1 \setminus B$ but $[D^\downarrow_{Z_2}]$ happens to agree with $[D^\downarrow_{\mathrm{zms}}]$ everywhere on $Z_2$. The asymptotic formulae are then given by composing the quasilinear $D \mapsto D^\downarrow_\mathcal{C}$ with the quasipolynomial Euler characteristic of the relevant SQM. In this example, there is a single SQM, and it is smooth, so the Euler characteristic is simply a polynomial given by Hirzebruch--Riemann--Roch. 

Let us comment further on the lattice structure in this example. Applying \cref{prop:integerZariski_quasilinear}, we know that the lattice $\Lambda_{Z_1}$ associated to the asymptotic quasilinear map $D \mapsto D^\downarrow_{Z_1}$ is
\begin{equation}
    \Lambda_{Z_1} = \mathbb{Z}[D_2] + \mathbb{Z}[D_6],
\end{equation}
and we can read off from \cref{eq:ex1_charge_matrix} that this generates an index-$2$ sublattice of $\Cl(X)$ such that the projection $\Cl(X) \to \Cl(X) / \Lambda_{Z_1} \cong \mathbb{Z}_2$ is $a[D_1] + b[D_6] \mapsto a \bmod 2$. On $Z_1$, we therefore expect the asymptotic quasipolynomial formula to consist of two polynomials, with the parity of $a$ determining which one applies. This perfectly matches the formula presented for $Z_1 \setminus B$, which implements the quasipolynomiality through dependence on $\lfloor a/2 \rfloor$ or $(-1)^a$, which depend only on $a \bmod 2$. A similar analysis for $\Lambda_{Z_2} = \mathbb{Z}[D_1] + \mathbb{Z}[D_5]$ finds that $\Cl(X) \to \Cl(X) / \Lambda_{Z_2} \cong \mathbb{Z}_3$ is $a[D_1] + b[D_6] \mapsto b \bmod 3$ and indeed, the formula given for $Z_2$ is quasipolynomial precisely because it depends on $\lfloor 7b/3 \rfloor$ or $b \bmod 3$. 
We comment that the asymptotic formula for a fixed Mori chamber $\mathcal{C}$ can involve divisors in different cosets of $\Lambda_{\mathcal{C}}$ being subtracted to distinct Mori chambers in the moving cone --- see \cref{sec:ex_4} --- however, we only have one chamber in $\Mov(X)$, and moreover this cannot happen for Picard number $2$.

In order to make contact with the results at the end of \cref{sec:subtraction} regarding the asymptotic structure of the subtraction method and \cref{alg}, in particular with the asymptotic quasilinearity discussed in \cref{prop:asymptotic_subtraction}, we can observe that in $Z_1$ we can choose the class $S$ defined in \cref{prop:asymptotic_subtraction} to be $\tfrac{5}{2}[D_2]$ --- this is the smallest allowed choice. On the other hand, for $Z_2$ \cref{alg} terminates after a single step so any $S$ will do. 

As argued in \cref{MAIN:together} it must be true that the asymptotic formula for each Mori chamber holds for all (big) divisors in that chamber. Hence, in particular, the formula presented for $Z_1 \setminus B$ must also hold on~$B$, which it does. In particular,  for $[D] = a [D_1] + b [D_6]$, we can conclude
\begin{equation}
    \label{eq:ex_1_global}
    \begin{aligned}
        h^0(X, \mathcal{O}_X(D)) 
        \; &= \;
        \begin{cases}
            h^0(X, \mathcal{O}_X) = 1 & [D] \in \partial^\ell Z_1 \\
            \chi(X, \mathcal{O}_X(a D_1 + \lfloor a/2 \rfloor D_6))=Q_{\ell}(a) & [D] \in Z_1 \setminus \partial^\ell Z_1 \\
            \chi(X, \mathcal{O}_X(D)) = P(a,b) & [D] \in \Nef(X) \setminus \{0\} \\
            \chi(X, \mathcal{O}_X(\lfloor 7b/3 \rfloor D_1 + b D_6)) =Q_r(b) & [D] \in Z_2 \setminus \partial^r Z_2 \\
            h^0(X, \mathcal{O}_X)=1  & [D] \in \partial^r Z_2 \\
        \end{cases} 
    \end{aligned}
\end{equation}

\subsection{Example 2 --- Torsionful Contraction} \label{sec:ex2}

Consider the reflexive polytope $\Delta^\circ_{20}$ in the Kreuzer--Skarke classification
of four-dimensional reflexive polytopes \cite{KreuzerSkarke}. The rays of the fan relevant to our chosen simplicial toric variety $V$ are generated by the lattice points
\begin{equation*}
    \begin{aligned}
        v_1 &= (-5,\,-1,\,-1,\,-2),\quad v_2 = (0,\,0,\,1,\,0),\quad
        v_3 = (0,\,1,\,0,\,0),\\ v_4 & = (1,\,0,\,0,\,0),\quad\quad\quad\quad\,
        v_5 = (1,\,0,\,0,\,2),\quad v_6 = (1,\,0,\,0,\,1).
    \end{aligned}
\end{equation*}
Analogously to the previous example, a generic anticanonical
hypersurface $X \;\subset\; V$ is a smooth Calabi--Yau threefold. 
$X$ has non-trivial Hodge numbers $(h^{1,1}, h^{2,1}) = (2, 106)$, with its Picard group being inherited from the class group of $V$. We fix a basis for $\Cl(X) \cong \Pic(X)$ by specifying the following class group grading for the prime torus-invariant divisors, ordered in the same way as the above points:
\begin{equation}
    \begin{pmatrix}
        1 & 1 & 1 & 4 & 1 & 0 \\
        1 & 1 & 1 & 3 & 0 & 2
    \end{pmatrix}~.
\end{equation}
As before, let $x_i$ denote the homogeneous coordinate associated to the $i^{\rm th}$ ray, $\hat{D}_i$ the prime torus-invariant divisor given by $x_i = 0$, and $D_i = \hat{D}_i \cap X$. We see that $\Cl(X)$ is generated by $\{[D_5], [D_6/2]\}$ (with both of these being integral classes, in spite of the factor of $1/2$). These basis elements will turn out to be (at least proportional to) the exceptional divisors of the divisorial contractions associated to $X$.

\begin{figure}
    \centering
    \includegraphics[width=.58\linewidth]{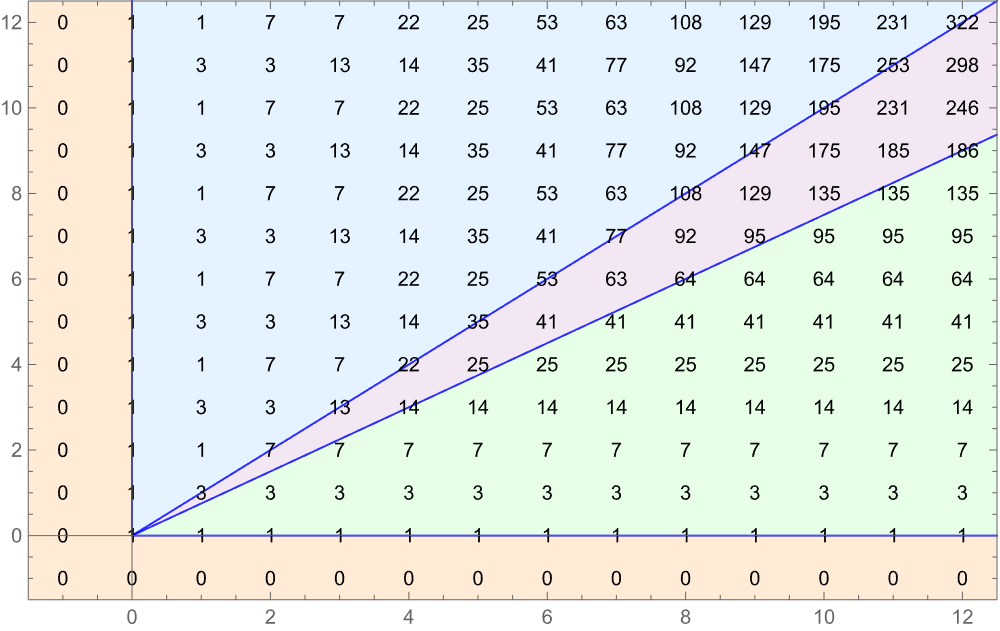}
    \caption{Cohomology data and chamber decomposition of ${\rm Eff}(X)$ for the Calabi--Yau threefold associated with the reflexive polytope $\Delta_{20}^\circ$ in the Kreuzer--Skarke classification of four-dimensional reflexive polytopes. The numbers count global sections of line bundles on $X$.}
    \label{fig:plot20}
\end{figure}

\subsubsection{Birational tomography.}

Before determining the birational geometry of $X$, let us again ask what can be inferred from a finite collection of global-section computations alone. \cref{fig:plot20} displays values of $h^0\bigl(X,\mathcal O_X(aD_5+b(D_6/2))\bigr)$ at lattice points in the effective cone. As in the previous example, the data separate naturally into three regions, on which they are fit by distinct cubic quasipolynomial expressions.

In the central region, the data are described by the cubic polynomial
\begin{equation}
P(a,b) = \frac{3}{2}a^3 -6a^2b +8ab^2 -\frac{10}{3}b^3 -\frac12a +\frac{10}{3}b.
\end{equation}
In one exterior region, they are instead described by
\begin{equation}
Q_\ell(a,b) = \frac16a^3 +(-1)^{a-b}a +\frac{11}{6}a,
\end{equation}
while in the other they are described by
\begin{equation}
Q_r(b) = \frac29b^3 +\frac83b +\frac19\bigl((b+1)\bmod 3-1\bigr).
\end{equation}
The changes of formula occur along the rays $[D_1]=[D_5]+[D_6/2]$ and $[D_4]=4[D_5]+3[D_6/2]$, corresponding in the $(a,b)$-coordinates to
$b=a$ and, respectively,  $b=3a/4$. Thus the cohomological data suggest two walls separating three birational regimes.

In the right-hand region, the formula depends only on $b$, suggesting that the $D_5$-direction is fixed or contracted on the birational model governing the region. The formula has a term depending on $b\bmod 3$; as we will see, this dependence is explained in the contraction method by the fact that the Picard group of the relevant $D$-minimal model has index $3$ in the class group --- yielding a strictly quasipolynomial Euler characteristic --- due to a quotient singularity. In the subtraction method, this dependence is captured by how the Mori chambers lie with respect to the lattice of divisor classes.

The left-hand region contains a still more informative phenomenon. The formula for $Q_\ell(a,b)$ has no ordinary polynomial dependence on $b$, but it does retain the parity of $a-b$. As such, motion in the $D_6/2$-direction is invisible to global sections up to a residual $\mathbb Z_2$-valued datum.
This is stronger than merely detecting a possible contracted divisor. The data suggest that the corresponding birational model has only one non-torsion divisor direction but retains an additional torsion class of order two. One possible explanation is therefore that the relevant contraction sends $a[D_5]+b[D_6/2]$ to a class of the form $aH+\epsilon(a,b)T$, with $\epsilon(a,b)=(a-b) \bmod 2$, $H$ a generator of the free part of the class group of the contracted model and $T$ a two-torsion class.

The polynomial parts provide further evidence for this picture. The central expression features no invariant directions, whereas the two exterior regions feature formulae which depend on only one direction in the class group, indicating that they are governed by birational models with Picard number one. 

Thus the finite cohomological data suggest the following birational picture: the effective cone should decompose into three regions separated by the rays $[D_1]$ and $[D_4]$; the two exterior regions should be governed by contractions of the $D_6$- and $D_5$-directions, respectively; the corresponding models should each have one remaining numerical divisor direction; and the left-hand model should carry a genuine order-two discrete structure, while the right-hand model should exhibit an order-three periodic phenomenon. We now construct the relevant birational models and test these predictions directly.

\subsubsection{Birational geometry and contraction.} The effective cone of $X$ admits the Mori chamber decomposition shown in \cref{fig:plot20}.
The small birational class of $X$ again consists of just $X$, so $\Mov(X) = \Nef(X) = \mathrm{Cone}([D_1], [D_4])$, shown in purple, and there are additionally two Mori chambers $Z_1  = \mathrm{Cone}([D_1], [D_6])$ and $Z_2 = \mathrm{Cone}([D_4], [D_5])$, shown in blue and, respectively, green. The ambient variety $V$ enjoys two divisorial contractions given by the blowdowns of the toric divisors $\hat{D}_6$ and $\hat{D}_5$, respectively:
\begin{equation*}
f^\ell : V \dashrightarrow \mathbb{P}_{1\,1\,1\,1\,4} / \mathbb{Z}_2, \qquad 
f^r : V \dashrightarrow \mathbb{P}_{1\,1\,1\,2\,3}~.
\end{equation*}
We stress that the five rays remaining after removing $v_6$ have the linear relation associated with $\mathbb{P}_{1\,1\,1\,1\,4}$ but these rays do not generate $\mathbb{Z}^4$ --- instead they generate an index-$2$ sublattice, meaning that the class group is $\mathbb{Z} \oplus \mathbb{Z}_2$. This torsion follows from the presence of a non-primitive exceptional divisor $\hat{D}_6$, as noted in \cite[Prop. 20]{Gendler:2026uux}. Upon restriction to the hypersurface, these toric contractions induce maps contracting $D_6$ and $D_5$,
$f^\ell : X \dashrightarrow Y_1$, and, respectively, 
$f^r : X \dashrightarrow Y_2$,
determining two Mori chambers in ${\rm Eff}(X)$. In particular, $Y_1$ and $Y_2$ are the anticanonical hypersurfaces in $\mathbb{P}_{1\,1\,1\,1\,4} / \mathbb{Z}_2$ and $\mathbb{P}_{1\,1\,1\,2\,3}$, respectively.

As in the previous example, we can compute $h^0(X, \mathcal{O}_X(D))$ for any effective $D$ on $X$ by performing the computation on a $D$-minimal model: one of $X, Y_1, Y_2$. While, again, neither $Y_1$ nor $Y_2$ is smooth, the hypersurface realizations continue to allow us to circumvent this. In particular, applying \cref{cor:toric_hyp_cy_hilbert_function} as in the previous section allows us to compute the following Hilbert series $HS(Y_i;s,t)$ which holds for all non-trivial effective classes.
\begin{equation}
    \begin{aligned}
        HS(Y_1;s,t) &= \frac{1 - t^8}{(1-t)^3(1-st)(1-st^4)}~, \\
        HS(Y_2;t) &= \frac{1 - t^8}{(1-t)^3(1-t^2)(1-t^3)}~.
    \end{aligned}
\end{equation}
Here, $s$ is understood to square to $1$, as it corresponds to the torsionful generator of $\mathrm{Cl}(\mathbb{P}_{1\,1\,1\,1\,4} / \mathbb{Z}_2)$.
It follows straightforwardly that, for $H$ the restriction of the relevant ambient weighted projective space hyperplane class to $Y_1$ and $Y_2$, respectively, and $T$ a torsion generator on $Y_1$,
\begin{equation}
    \begin{aligned}
        \chi(Y_1, \mathcal{O}_{Y_1}(nH + mT)) = h^0(Y_1, \mathcal{O}_{Y_1}(nH + mT)) &= 
        \frac{1}{n!}\frac{d^n}{dt^n}\left[\frac{HS(Y_1;s,t)|_{s=1} + (-1)^m HS(Y_1;s,t)|_{s=-1}}{2}\right]_{t=0}\\
        & = \frac{1}{6}n^3 + (-1)^m n + \frac{11}{6} n \\
        \chi(Y_2, \mathcal{O}_{Y_2}(nH)) = h^0(Y_2, \mathcal{O}_{Y_2}(nH)) &= \frac{1}{n!}\frac{d^n}{dt^n} HS(Y_2;t)\Big|_{t = 0} = \frac{2}{9}n^3 + \frac{8}{3}n + \frac{1}{9} ((n+1) \bmod 3-1)
    \end{aligned}
\end{equation}
where, as before, the formula holds for big divisors ($n > 0$). We then already know that the trivial class has one global section, and we can additionally note that the pure torsion class $T$ on $Y_1$ has no global sections. The bracketed expression implements a sum over coefficients in $HS(Y_1;s,t)$ with powers of~$s$ congruent to $m$ modulo $2$. We then arrive at the following formula.
\begin{equation}
    \begin{aligned}
        h^0(X, \mathcal{O}_X(D)) 
        \; &= \;
        \begin{cases}
            h^0(X, \mathcal{O}_X(-(b \bmod 2)[D_6])) & D \in \partial^\ell Z_1 \\
            \chi(Y_1, \mathcal{O}_{Y_1}(f^{\ell}_* D)) & D \in Z_1 \setminus \partial^\ell Z_1 \\
            \chi(X, \mathcal{O}_X(D)) & D \in \Nef(X) \setminus \{0\} \\
            \chi(Y_2, \mathcal{O}_{Y_2}(f^{r}_*D)) & D \in Z_2 \setminus \partial^r Z_2 \\
            h^0(X, \mathcal{O}_X) & D \in \partial^r Z_2
        \end{cases}
    \end{aligned}
\end{equation}
It then suffices to note that $f^\ell_* (a[D_5] + b[D_6/2]) = a [H] + ((a - b) \bmod 2) [T]$ and $f^r_* (a[D_5] + b[D_6/2]) = b [H]$ --- substitution then reproduces the formulae above.

\subsubsection{Subtraction and \cref{alg}.} Once more, no subtraction is required for any divisor in $\Nef(X)$. For divisors in $Z_1$, the relevant birational map is $f^\ell$, for which there is a single exceptional divisor, $E^{Z_1}_1 = D_6$. We can read off that 
\begin{equation}
    \lambda^{Z_1}_1(a[D_5] + b[D_6/2]) = b/2 - a/2,
\end{equation}
so for $Z_1$, 
\begin{equation}
    P_{\mathbb{Z},Z_1}([D]= a[D_5] + b[D_6/2]) = [D] - \lceil b/2 - a/2 \rceil [D_6] = a [D_5] + (b - 2 \lceil b/2 - a/2 \rceil)[D_6/2]
\end{equation}
For $a - b \equiv 0$ modulo $2$ or $a \geq 4$, this is movable and the algorithm terminates. Otherwise, if $a > 0$ then $P_{\mathbb{Z},Z_1}([D]) \in Z_2$ and another iteration is required, while if $a = 0$ the output is a non-effective class and the algorithm returns $0$, as discussed in \cref{sec:subtraction}.

It is worth emphasizing that because the exceptional divisor $E^{Z_1}_1 = D_6$ is not primitive in the Picard group, but rather generates an index two sublattice of its saturation, there is different behavior depending on the remainder of $a - b$ modulo $2$, whereas in the previous example there was no dependence on the coefficient of the exceptional divisor. But this is natural given that we expect different behavior for the different cosets of $\Lambda_{Z_1}$, one of whose generators is $E^{Z_1}_1 = D_6$, whose class is not primitive in the class group. This phenomenon is explored in \cite{Gendler:2026uux}, where it was connected to the presence of torsion in the birational model achieved by contracting the non-primitive class. We note that this torsion appears to not be preserved upon performing a generic deformation from this singular birational model.

Now, for divisors in $Z_2$, the relevant birational map is $f^r$, for which there is $E^{Z_2}_1 = D_5$. We can read off that
\begin{equation}
    \lambda^{Z_2}_1(a[D_5] + b [D_6/2]) = a - 4b/3
\end{equation}
so in $Z_2$, 
\begin{equation}
    P_{\mathbb{Z},Z_2}([D]= a[D_5] + b [D_6/2]) = [D] - \lceil a - 4b/3 \rceil [D_5] = \lfloor 4b/3 \rfloor [D_5] + b [D_6/2]
\end{equation}
For all $[D] \in Z_2$, then, $P_{\mathbb{Z},Z_2}([D]) \in \Mov(X)$ and the algorithm terminates. In conjunction, then, the algorithm requires $0$, $1$, or $2$ steps, and requires $2$ steps only for the divisors in the set 
\begin{equation}
    B = \{ D = a[D_5] + b[D_6/2] \in Z_1 \; | \; a - b \equiv 1 \bmod 2, \; 0 < a < 4 \}.
\end{equation}
We have now again completely characterized how the subtraction algorithm proceeds for every class in the effective cone. We note that for $[D] = a[D_5] + b[D_6/2] \in \Eff(X)$, $\chi(X,D)=P(a,b)$.
Before summarizing the result, let us also briefly comment on non-big divisors on the boundary of the effective cone. In the previous example, the subtraction method reduced all of these to the trivial class. Now, though, for $b [D_6/2]$, if $b \equiv 1$ modulo $2$, then subtraction actually reduces to the non-effective divisor $-[D_6/2]$. This is entirely consistent, as discussed in \cref{sec:subtraction}: this merely encodes that such divisors have no global sections. In particular, as $D_6$ still belongs to the stable base locus of these divisors, we are free to subtract it.  We then have the following formula.
\begin{equation}
    \begin{aligned}
        h^0(X, \mathcal{O}_X(D))
        \; &= \;
        \begin{cases}
            h^0(X, \mathcal{O}_X(-(b \bmod 2)[D_6])) & [D] \in \partial^\ell Z_1 \\
            \chi(X, \mathcal{O}_X(a D_5 + (b - 2 \lceil b/2 - a/2 \rceil)(D_6/2))) & [D] \in Z_1 \setminus (B \cup \partial^\ell Z_1) \\
            h^0(X, \mathcal{O}_X) & [D] \in B, \; a = 1 \\
            \chi(X, \mathcal{O}_X(\lfloor 4(a-1)/3 \rfloor D_5 + (a-1)(D_6/2))) & [D] \in B, \; a \in \{2, 3\} \\
            \chi(X, \mathcal{O}_X(D)) & [D] \in \Nef(X) \setminus \{0\} \\
            \chi(X, \mathcal{O}_X(\lfloor 4b/3 \rfloor D_5 + b (D_6/2))) & [D] \in Z_2 \setminus \partial^r Z_2  \\
            h^0(X, \mathcal{O}_X) & [D] \in \partial^r Z_2 \\
        \end{cases} \\[4pt]
        &= \;
        \begin{cases}
            (b+1) \bmod 2 & (a,b) \in \partial^\ell Z_1 \\
            \frac{1}{6}a^3 + (-1)^{a-b}\, a + \frac{11}{6}a & (a, b) \in Z_1 \setminus (B \cup \partial^\ell Z_1) \\
            1 & (a, b) \in B, \; a = 1 \\
            \frac{2}{9}(a-1)^3 + \frac{8}{3}(a-1) + \frac{1}{9}\bigl((a \bmod 3) - 1\bigr) & (a, b) \in B, \; a \in \{2, 3\} \\
            \frac{3}{2}a^3 - 6 a^2 b + 8 a b^2 - \frac{10}{3} b^3 - \tfrac{1}{2} a + \frac{10}{3} b & (a, b) \in \Nef(X) \setminus \{0\} \\
            \frac{2}{9}b^3 + \frac{8}{3}b + \frac{1}{9}\bigl((b+1)\bmod 3 - 1\bigr) & (a, b) \in Z_2 \setminus \partial^r Z_2 \\
            1 & (a,b) \in \partial^r Z_2 \\
        \end{cases}
    \end{aligned}
\end{equation}

\subsubsection{Subtraction and Asymptotic/Global Formulae.} The situation for the asymptotic formulae is largely analogous to the previous example, differing primarily because $E^{Z_1}_1 = D_6$ is not primitive. The quasilinear maps $D \mapsto D^\downarrow_\mathcal{C}$ that asymptotically give the Zariski movable subtraction, for $[D] = a [D_5] + b [D_6/2]$, are given as follows.
\begin{equation}
    \begin{aligned}
        [D^\downarrow_{Z_1}] &= a [D_5] + (b - 2 \lceil b/2 - a/2 \rceil)[D_6/2] \\
        [D^\downarrow_{Z_2}] &= \lfloor 4b/3 \rfloor [D_5] + b [D_6/2]
    \end{aligned}
\end{equation}
Once again, $[D^\downarrow_{Z_1}]$ agrees with $[D^\downarrow_{\mathrm{zms}}]$ in the asymptotic region $Z_1 \setminus B$ but $[D^\downarrow_{Z_2}]$ happens to agree with $[D^\downarrow_{\mathrm{zms}}]$ everywhere on $Z_2$. The asymptotic formulae are again given by composing the quasilinear $D \mapsto D^\downarrow_\mathcal{C}$ with the polynomial Euler characteristic of the unique smooth SQM.

We can also again study the lattices relevant for this example. The lattice associated to $Z_1$ is
\begin{equation}
    \Lambda_{Z_1} = \mathbb{Z}[D_1] + \mathbb{Z}[D_6]
\end{equation}
such that the projection $\Cl(X) \to \Cl(X) / \Lambda_{Z_1} \cong \mathbb{Z}_2$ is $a[D_5] + b[D_6/2] \mapsto a - b \bmod 2$, and our asymptotic formula indeed has quasipolynomiality arising exactly from $\lceil b/2 - a/2 \rceil$, whose periodic piece is $(-1)^{a-b}$. 

Making contact once again with \cref{prop:asymptotic_subtraction} as we did in the previous section, the class $S$ from that result can be chosen to be $3[D_1]$ for $Z_1$, which is again the smallest allowed choice. Again, on $Z_2$ any $S$ is appropriate because subtraction terminates after a single step.

Once again, the asymptotic formulae for each Mori chamber must actually hold for all big divisors in that Mori chamber, allowing us to simplify to
\begin{equation}
    \begin{aligned}
        h^0(X, \mathcal{O}_X(D)) 
        \; &= \;
        \begin{cases}
            h^0(X, \mathcal{O}_X(-(b \bmod 2)[D_6])) =(b+1) \bmod 2 & [D] \in \partial^\ell Z_1 \\
            \chi(X, \mathcal{O}_X(a D_5 + (b - 2 \lceil b/2 - a/2 \rceil)(D_6/2)))= Q_\ell(a,b) & [D] \in Z_1 \setminus \partial^\ell Z_1 \\
            \chi(X, \mathcal{O}_X(D)) =P(a,b) & [D] \in \Nef(X) \setminus \{0\} \\
            \chi(X, \mathcal{O}_X(\lfloor 4b/3 \rfloor D_5 + b (D_6/2))) = Q_r(b)& [D] \in Z_2 \setminus \partial^r Z_2  \\
            h^0(X, \mathcal{O}_X) = 1& [D] \in \partial^r Z_2 \\
        \end{cases} 
    \end{aligned}
\end{equation}

\subsection{Example 3 --- Picard number three} \label{sec:ex3}

In this example we will exhibit a geometry with a more complicated Mori chamber decomposition. We have already seen the lattice/convex geometry of subtraction and contraction in detail in the previous two examples, so we will not present such a thorough analysis for this geometry. Instead, we will merely present the chamber decomposition and the formulae that result for global sections from performing our methods (in a manner entirely analogous to the previous two examples). It is worth noting that this example was selected for its simplicity: in fact, there is a piecewise \textit{polynomial} global formula for global sections --- there is no quasipolynomiality. 

Consider the reflexive polytope $\Delta^\circ$ in the Kreuzer--Skarke classification
of four-dimensional reflexive polytopes \cite{KreuzerSkarke} with points 
\begin{equation*}
    \begin{aligned}
        v_1 &= (1,0,0,0),\quad v_2 = (-4,-1,-1,-1),\quad v_3 = (0,0,0,1), \\
        v_4 &= (0,0,1,0),\quad v_5 = (0,1,0,0),\quad\quad\quad\quad\, v_6 = (-2,-1,0,0),\quad
        v_7 = (-1,1,-1,0).
    \end{aligned}
\end{equation*}
Let $V$ denote one of the five weak-Fano simplicial toric fourfolds whose fan has rays generated by the above points --- i.e., the fans are associated to fine, regular, star triangulations of $\Delta^\circ$. Analogously to the previous examples, generic anticanonical
hypersurfaces $X \;\subset\; V$ are smooth Calabi--Yau threefolds. There are two other simplicial toric fourfolds with fans constructed from these rays for which the anticanonical class is not nef: such fans were studied under the name ``vex triangulations'' in \cite{MacFadden:2025ssx}, where it was shown that the associated toric varieties also have smooth anticanonical hypersurfaces. 
$X$ has non-trivial Hodge numbers $(h^{1,1}, h^{2,1}) = (3, 103)$, with its Picard group being inherited from the class group of $V$. We fix a basis by specifying the following class group grading for the prime torus-invariant divisors, ordered in the same way as the above points:
\begin{equation}
    \begin{pmatrix}
        7 & 1 & 1 & 2 & 1 & 1 & 1 \\
        4 & 1 & 1 & 1 & 1 & 0 & 0 \\
        2 & 0 & 0 & 0 & 1 & 1 & 0
    \end{pmatrix}~.
\end{equation}
As before, we let $x_i$ denote the homogeneous coordinate associated to the $i^{\rm th}$ ray, $\hat{D}_i$ the prime torus-invariant divisor given by $x_i = 0$, and $D_i = \hat{D}_i \cap X$. 

The small birational class of the anticanonical hypersurfaces in the seven simplicial toric fourfolds consists of two elements, which we label $X = X_1$ and $X_2$. The Mori chamber decomposition for $X$ is illustrated in \cref{fig:eff_cone_diagram}, as well as the Mori chamber decomposition for the ambient toric variety (which is just its secondary fan). In particular, the latter refines the former: we use solid lines to denote faces of the CY Mori chamber decomposition and dotted lines to denote faces of the ambient toric Mori chamber decomposition which do not descend to $X$. For example, the Mori chamber $\Nef(X)$ is the union of five toric Mori chambers (i.e., secondary cones). Intuitively, this refinement arises because the ambient toric varieties possess birational contractions which do not descend to the hypersurface because the exceptional locus doesn't intersect it.

\begin{figure}
\begin{center}
\scalebox{0.8}{
\begin{tikzpicture}[xscale=.5, yscale=.5]

\tkzDefPoint(-8, 8){D6}     
\tkzDefPoint( 8, 8){D5}     
\tkzDefPoint( 8,-8){BR}     
\tkzDefPoint(-8,-8){D7}     

\tkzDefPoint( 0, 8){TM}     
\tkzDefPoint(-8, 0){LM}     
\tkzDefPoint( 8, 0){RM}     
\tkzDefPoint( 0,-8){BM}     

\tkzDefPoint( 0           , 0           ){T}    
\tkzDefPoint( 2.6666666667,-2.6666666667){R}    
\tkzDefPoint(-2.6666666667,-2.6666666667){L}    
\tkzDefPoint( 0,-3.3333333333){C}                

\tkzDrawSegments(D6,D5  D5,BR  BR,BM  BM,D7  D7,D6)

\tkzDrawSegments(D6,T  D5,T  D6,L  D7,L  D5,R  BR,R)

\tkzDrawSegments(T,R  R,BM  BM,L  L,T)

\tikzset{dotline/.style={gray, line width=0.7pt,
        dash pattern=on 0.5pt off 2.5pt, line cap=round}}
\draw[dotline] (C) -- (D6);
\draw[dotline] (C) -- (D5);
\draw[dotline] (C) -- (D7);
\draw[dotline] (C) -- (BR);
\draw[dotline] (BM) -- (C);     

\tkzDrawPoints[size=2, color=black](D6, D5, BR, D7, T, R, L, BM, C)

\tkzLabelPoint[above left ](D6){$D_6$}
\tkzLabelPoint[above right](D5){$D_5$}
\tkzLabelPoint[below right](BR){$D_2, D_3$}        
\tkzLabelPoint[below left ](D7){$D_7$}

\tkzLabelPoint[above      ](T){$D_5 + D_7$}         
\tkzLabelPoint[right      ](R){$D_4 + D_5$}         
\tkzLabelPoint[left       ](L){$D_4 + D_6$}         
\tkzLabelPoint[below      ](BM){$D_4$}        

\tkzLabelPoint[below     ](C){$D_1$}

\node at ( 0  , 5.0) {$\mathcal{C}_3$};
\node at (-5.5, 0  ) {$\mathcal{C}_5$};
\node at (-3.5, 2.0) {$\mathcal{C}_4$};
\node at ( 3.5, 2.0) {$\mathcal{C}_2$};
\node at ( 5.5, 0  ) {$\mathcal{C}_1$};
\node at ( 0, -2.0) {$\Nef(X)$};
\node at (-3.5,-6  ) {$\mathcal{C}_6$};
\node at ( 3.5,-6  ) {$\varphi_2^*\Nef(X_2)$};

\node[above=4pt] at (TM)      {$B(\bullet_1)$};
\node[left =4pt] at (LM)      {$B(\bullet_2)$};
\node[right=4pt] at (RM)      {$B(\mathbb{P}^1)$};
\node[below=4pt] at (-4,-8)   {$B(\mathbb{P}^2)$};
\node[below=4pt] at ( 4,-8)   {$B(F_1)$};

\end{tikzpicture}
}
\end{center}
\caption{Mori chamber decomposition of $\Eff(X)$ for the Picard number three Calabi--Yau threefold of \cref{sec:ex3}, drawn as a codimension-one cross-section. Solid black lines denote walls of the Mori chamber decomposition of $X$, while gray dotted lines denote additional walls of the secondary fan of the ambient toric variety which do not descend to walls on $X$. The moving cone consists of the two chambers $\Nef(X)$ and $\varphi_2^*\Nef(X_2)$, while $\mathcal{C}_1,\ldots,\mathcal{C}_6$ correspond to divisorial contractions. The labels $B(\mathbb{P}^1)$, $B(\mathbb{P}^2)$, $B(F_1)$, and $B(\bullet_i)$ indicate boundary cones whose divisors induce Iitaka fibrations with the indicated bases, with $\bullet_i$ denoting a point.}
\label{fig:eff_cone_diagram}
\end{figure}

Let us now describe the Mori chamber decomposition of $X$ in some detail. We denote the SQMs of $X$ by $X_i$ and their Mori chambers by $\varphi_i^* \Nef(X_i)$, while the divisorial contractions of (SQMs of) $X$ are denoted $Y_i$ with Mori chambers $\mathcal{C}_i$ and birational maps $f_i : X \dashrightarrow Y_i$. The moving cone decomposes into two CY K\"ahler cones, $\Nef(X)$ and $\varphi_2^* \Nef(X_2)$ corresponding to two distinct Picard number three Calabi--Yau threefolds $X$ and $X_2$. Explicitly, the two K\"ahler cones are
\begin{equation}
    \begin{aligned}
        \Nef(X) &= \mathrm{Cone}(D_4,\ D_5+D_7,\ D_4+D_6,\ D_4+D_5)~, \\
        \varphi_2^* \Nef(X_2) &= \mathrm{Cone}(D_2,\ D_4,\ D_4+D_5)~.
    \end{aligned}
\end{equation}
The three facets of $\Mov(X)$ that don't lie in a facet of $\partial \Eff(X)$ correspond to divisorial contractions: in particular, the top right, top left, and bottom left facets of $\Mov(X)$ correspond to the contraction of $D_5$, $D_6$, and $D_7$, respectively. This results in four new Mori chambers intersecting $\Mov(X)$ at codimension one: $\mathcal{C}_1, \mathcal{C}_2, \mathcal{C}_4, \mathcal{C}_6$. In particular, $D_6$ and $D_7$ can be contracted only on $X$ (yielding $Y_4$ and $Y_6$, respectively), while $D_5$ can be contracted on either $X$ or $X_2$ (yielding $Y_2$ and $Y_1$, respectively).\footnote{We comment that removing the rays associated to $\hat{D}_5, \hat{D}_6, \hat{D}_7$ from $\Delta^\circ$ results in three new, distinct reflexive polytopes, and indeed contracting these divisors on $X$ yields Picard number two anticanonical hypersurfaces in the toric varieties associated to these new polytopes (albeit not generic hypersurfaces, but rather ones with sections inherited from the original toric variety).} Of course, as must be the case for Mori dream spaces, the Mori chambers are generated by their intersection with $\Mov(X)$ and the relevant exceptional divisors (see \cref{cor:chamber_movable_part} and the discussion around \cref{def:birational_data}).

Additionally, $Y_2, Y_4, Y_6$ admit further divisorial contractions. Contracting $f_{2*} D_6$ on $Y_2$ or $f_{4*} D_5$ on $Y_4$ yields a Picard number one variety $Y_3$, and contracting $f_{4*} D_7$ on $Y_4$ or $f_{6*} D_6$ on $Y_6$ yields a Picard number one variety $Y_5$. These have Mori chambers intersecting $\Mov(X)$ at codimension two. 

Explicitly, the six Mori chambers outside of $\Mov(X)$ are:
\begin{equation*}
    \begin{aligned}
        \mathcal{C}_1 &= \mathrm{Cone}(D_2,\ D_5,\ D_4{+}D_5), \quad~
        \mathcal{C}_2 = \mathrm{Cone}(D_5,\ D_5{+}D_7,\ D_4{+}D_5), \quad~
        \mathcal{C}_3 = \mathrm{Cone}(D_6,\ D_5,\ D_5{+}D_7), \\
        \mathcal{C}_4 &= \mathrm{Cone}(D_6,\ D_5+D_7,\ D_4{+}D_6), \quad~
        \mathcal{C}_5 = \mathrm{Cone}(D_7,\ D_6,\ D_4{+}D_6), \quad~
        \mathcal{C}_6 = \mathrm{Cone}(D_7,\ D_4,\ D_4{+}D_6).
    \end{aligned}
\end{equation*}
  
A crucial fact is that the lattices $\Lambda_{\mathcal{C}_i}$ for $1 \leq i \leq 6$ all coincide with $\Cl(X)$. This means that the Zariski decomposition itself --- not just our modified integral Zariski decomposition --- maps integral divisors to integral divisors on all Mori chambers, and the subtraction method applied to big divisors results in asymptotic polynomial formulae, not quasipolynomial formulae. We can then exploit our result for Calabi--Yau varieties to conclude that these are global polynomial formulae. Evidently, in spite of the birational contractions $Y_1, \dots, Y_6$ being singular, their Euler characteristics are all polynomial on the entirety of $f_{i*}\Cl(X)$, rather than just $f_{i*}\Cl(X) \cap \Pic(Y_i)$.

Finally, we consider the non-big divisors belonging to the faces of the effective cone. A priori, because $X$ is Calabi--Yau, our methods do not apply to $\partial \Eff(X)$. However, for this particular example, each face can be handled using the methods of \cref{sec:non-big}. We stress that in order to do so, one should formally prove that the pullback of the Picard group of the base of the relevant Iitaka fibration is saturated in the Picard group of the original variety. However, since in this work we place emphasis on the big cone, with this being the only example where we apply the methods of \cref{sec:non-big}, for brevity and simplicity we will not provide such a proof here. Explicit line bundle cohomology computations verify the formulae implied by \cref{sec:non-big} for non-big divisors assuming saturation and we content ourselves with this strong empirical evidence that the saturation holds. 

Having provided this disclaimer, we now proceed to an analysis of the Iitaka fibrations associated to each face. The effective cone of $X$ is rational polyhedral, so $X$ is a Mori dream space by \cref{prop:when_cy_is_mds} (in addition to being Calabi--Yau), and hence every nef divisor on $X$ is semiample. It will suffice to consider $\partial \Eff(X) \cap \Mov(X)$, because all other faces of $\Eff(X)$ can be reduced to these faces by subtractions. 
In this example, all fibrations are inherited from the ambient toric variety, whose fibrations are straightforward to compute \cite{Kreuzer:2000qv}. We find that $X_2$ is genus-one fibered over the Hirzebruch surface $F_1$, which is itself fibered, inducing a K3 fibration of $X_2$ as well. $F_1$ also admits a birational contraction to $\mathbb{P}^2$, meaning $Y_6$ is also a genus-one fibration over $\mathbb{P}^2$. The faces of $\varphi_2^* \Nef(X_2)$ whose divisors induce Iitaka fibrations with base $\mathbb{P}^2$, $F_1$, and $\mathbb{P}^1$ are $\Cone(D_4)$, $\Cone(D_2, D_4)$, and $\Cone(D_2)$, respectively. We can also consider the trivial class, a zero-dimensional face of $\Mov(X) \cap \partial \Eff(X)$, to correspond to the trivial Iitaka fibration of $X$ to a point. Extending these faces to include all of the divisors in $\partial \Eff(X)$ for which the subtraction method reduces to these faces yields the following cones.\footnote{Equivalently, any Mori chamber outside of $\Mov(X)$ containing one of the faces of $\Mov(X) \cap \partial \Eff(X)$ mentioned above will enjoy the same fibration, so instead of subtraction, we can imagine performing contraction and studying the Iitaka fibration on the $D$-minimal model of a non-big $D$.}
\begin{equation}
    \begin{aligned}
        B(F_1) &= \mathrm{Cone}(D_2,\ D_4), \quad
        B(\mathbb{P}^2) = \mathrm{Cone}(D_7,\ D_4), \quad
        B(\mathbb{P}^1) = \mathrm{Cone}(D_2,\ D_5), \\
        B(\bullet_1) &= \mathrm{Cone}(D_6,\ D_5), \quad
        B(\bullet_2) = \mathrm{Cone}(D_7,\ D_6).
    \end{aligned}
\end{equation}
Here $\bullet$ denotes a point: $B(\bullet_1)$ and $B(\bullet_2)$ are the cones of divisors whose Iitaka fibration contracts $X$ to a point, and the index serves only to distinguish them. We comment in passing that $X_1$ is of course also genus-one fibered over $\mathbb{P}^2$ because $\Nef(X)$ also contains $D_4$.

By applying \cref{prop:nonbig_reduction}, we can compute formulae for each of these faces of the effective cone because we understand the global sections of the bases of the relevant Iitaka fibrations --- when they are non-trivial, they are all toric varieties, and hence can be treated using our methods for Fano-type varieties. Of course, this is overkill for this simple example, as the global sections of simple varieties like $\mathbb{P}^1$, $\mathbb{P}^2$, and $F_1$ can be understood using much more elementary methods. 

Let $[D] = a [D_7] + b ([D_5] - [D_6]) + c ([D_6] - [D_7])$ with representative $D$, and let $L=\mathcal O_X(D)$. Then, either by performing subtraction to reduce to the Euler characteristic of $X_1$, $X_2$, or one of the bases of the fibrations of these, or by performing contraction to reduce to the Euler characteristic of a birational variety (or of the base of an Iitaka fibration of a birational variety), we achieve the following formula. In all cases, $\pi$ denotes the relevant birational contraction or Iitaka fibration,
\begin{equation}
    \label{eq:ex_3_global}
    \begin{aligned}
        h^0(X, L) &=
        \begin{cases}
            \begin{aligned}
                \chi(X, L) & = \tfrac{1}{6}a^{3} - a^{2} b - \tfrac{1}{2}a^{2} c + 2 a b^{2} + 2 a b c + \tfrac{1}{2}a c^{2} \\ 
                &~~+ \tfrac{5}{6}a - \tfrac{4}{3}b^{3} - b^{2} c - 2 b c^{2} + \tfrac{4}{3}b + \tfrac{1}{6}c^{3} - \tfrac{1}{6}c
            \end{aligned}  & D \in \Nef(X) \setminus B(\mathbb{P}^2) \\[4pt]
            \chi(X_2, L) = - a^{2} c + 4 a b c + a - 3 b^{2} c - b c^{2} + b & D \in \varphi_2^*\Nef(X_2) \setminus B(F_1) \\[4pt]
            \begin{aligned}
               \chi(Y_4, \pi(L))& = \tfrac{1}{6}a^{3} - a^{2} b - \tfrac{1}{2}a^{2} c + 2 a b^{2} + 2 a b c + \tfrac{1}{2}a c^{2} \\
                & ~~+ \tfrac{5}{6}a - b^{3} - 2 b^{2} c - b c^{2} + 2b - \tfrac{1}{6}c^{3} - \tfrac{5}{6}c
            \end{aligned} & D \in \mathcal{C}_4 \setminus B(\bullet_2) \\[4pt]
            \begin{aligned}
                \chi(Y_2, \pi(L)) & = \tfrac{1}{6}a^{3} - \tfrac{1}{2}a^{2} c + \tfrac{1}{2}a c^{2} + \tfrac{17}{6}a - \tfrac{1}{3}b^{3} \\
                &~~+ b^{2} c - b c^{2} - \tfrac{2}{3}b + \tfrac{1}{6}c^{3} - \tfrac{13}{6}c
            \end{aligned}
            & D \in \mathcal{C}_2 \setminus B(\mathbb{P}^1) \\[4pt]
            \begin{aligned}
            \chi(Y_1, \pi(L)) & = a^{2} b - a^{2} c - 2 a b^{2} + 2 a b c\\
            &~~ + 3 a + b^{3} - b^{2} c - b - 2 c
            \end{aligned}& D \in \mathcal{C}_1 \setminus B(\mathbb{P}^1) \\[4pt]
            \chi(Y_6, \pi(L)) = b^{2} c - b c^{2} + 3 b + \tfrac{1}{3}c^{3} + \tfrac{2}{3}c & D \in \mathcal{C}_6 \setminus B(\mathbb{P}^2) \\[4pt]
            \chi(Y_3, \pi(L)) = \tfrac{1}{6}a^{3} - \tfrac{1}{2}a^{2} c + \tfrac{1}{2}a c^{2} + \tfrac{17}{6}a - \tfrac{1}{6}c^{3} - \tfrac{17}{6}c & D \in \mathcal{C}_3 \setminus B(\bullet_1) \\[4pt]
            \chi(Y_5, \pi(L)) = \tfrac{1}{3}b^{3} + \tfrac{11}{3}b & D \in \mathcal{C}_5 \setminus B(\bullet_2) \\[4pt]
            \chi(F_1, \pi(L))=  - \tfrac{1}{2}a^{2} + 2 a b + \tfrac{1}{2}a - \tfrac{3}{2}b^{2} + \tfrac{1}{2}b + 1  & D \in B(F_1) \setminus (B(\mathbb{P}^2) \cup B(\mathbb{P}^1)) \\[4pt]
            \chi(\mathbb{P}^2, \pi(L))=\tfrac{1}{2}b^{2} + \tfrac{3}{2}b + 1 & D \in B(\mathbb{P}^2) \\[4pt]
            \chi(\mathbb{P}^1, \pi(L))= a - c + 1 & D \in B(\mathbb{P}^1) \\[4pt]
            1 & D \in B(\bullet_1) \\[4pt]
            1 & D \in B(\bullet_2)
        \end{cases}
    \end{aligned}
\end{equation}

\subsection{Example 4 --- Non-trivial subtraction} \label{sec:ex_4}

There are two features of the subtraction methods and the asymptotic formulae they induce which we spent some time discussing and accommodating in \cref{sec:subtraction} but which have not featured in the three examples discussed so far. First, we noted that even for sufficiently large divisors, subtraction need not terminate in a single step. Second, we noted that for sufficiently large divisors in a fixed Mori chamber $\mathcal{C}$, subtraction in general will place different cosets of the lattice $\Lambda_\mathcal{C}$ into different Mori chambers in the moving cone. We will now illustrate these phenomena in an example. 

We have already presented how one can both comprehensively apply \cref{alg} and in the process construct asymptotic formulae for the entire effective cone --- for this reason, we will not perform a thorough analysis for a new geometry, but rather we will merely discuss the minimal amount of birational information needed for a particular example to see the two aforementioned, yet-unseen behavior in practice.

Consider the reflexive polytope $\Delta^\circ$ in the Kreuzer--Skarke classification
of four-dimensional reflexive polytopes \cite{KreuzerSkarke} with points 
\begin{equation*}
    \begin{aligned}
        v_1 &= (0,\,0,\,1,\,0),\quad v_2 = (-3,\,-2,\,-1,\,-1),\quad
        v_3 = (0,\,0,\,0,\,1),\quad v_4 = (0,\,1,\,0,\,0),\\
        v_5 &= (-2,\,0,\,1,\,0),\quad v_6 = (-2,\,-1,\,0,\,0),\quad
        v_7 = (-1,\,0,\,1,\,0),\quad v_8 = (1,\,0,\,0,\,0).
    \end{aligned}
\end{equation*}
In particular, consider the four simplicial toric fourfolds whose fans have rays generated by the above points. These fourfolds have Picard number four. We will consider the toric varieties themselves in this example, which are of course Fano type due to their projectivity by \cref{prop:toric_is_fano_type}. We fix a basis for their class group by specifying the following class group grading for the prime torus-invariant divisors, ordered the same as the above points:
\begin{equation}
    \begin{pmatrix}
        1 & 1 & 1 & 2 & 0 & 0 & 0 & 3 \\
        0 & 1 & 1 & 2 & 0 & 0 & 1 & 4 \\
        0 & 1 & 1 & 2 & 1 & 0 & 0 & 5 \\
        0 & 0 & 0 & 1 & 0 & 1 & 0 & 2
    \end{pmatrix}~.
\end{equation}
As before, we let $x_i$ denote the homogeneous coordinate associated to the $i$th ray and $\hat{D}_i$ the prime torus-invariant divisor given by $x_i = 0$.
Now, let $X$ (which we will also denote $X_1$) and $X_2$ denote the two toric fourfolds with nef cones satisfying
\begin{equation}
    \begin{aligned}
        \Nef(X) &= \Cone(\hat{D}_8, \hat{D}_1 + \hat{D}_8, 2\hat{D}_1 + \hat{D}_7 + 2 \hat{D}_8, \hat{D}_2 + \hat{D}_8) \\
        \varphi_2^* \Nef(X_2) &= \Cone(\hat{D}_4, \hat{D}_8, \hat{D}_1 + \hat{D}_8, 2\hat{D}_1 + \hat{D}_7 + 2 \hat{D}_8)~.
    \end{aligned}
\end{equation}
The moving cone is
\begin{equation}
    \Mov(X) = \Cone(\hat{D}_2, \hat{D}_4, \hat{D}_8, \hat{D}_1 + \hat{D}_8) ~.
\end{equation}
Contracting $\hat{D}_5, \hat{D}_6, \hat{D}_7$ yields the weighted projective space $Y_1 = \mathbb{P}_{1\,1\,1\,2\,3}$. We will also be interested in another contraction $Y_2$ which results from contracting just $\hat{D}_5$. Their Mori chambers are
\begin{equation}
    \begin{aligned}
        \mathcal{C}_1 &= \Cone(\hat{D}_8, \hat{D}_5, \hat{D}_6, \hat{D}_7) \\
        \mathcal{C}_2 &= \Cone(\hat{D}_8, \hat{D}_1 + \hat{D}_8, \hat{D}_2 + \hat{D}_8, \hat{D}_5)
    \end{aligned}
\end{equation}
From the $\Cl(X)$-grading we can read off that $\Cl(X)/\Lambda_{\mathcal{C}_1} = \mathbb{Z}_3$ with $\Cl(X) \to \Cl(X)/\Lambda_{\mathcal{C}_1}$ being given by $a\hat{D}_1 + b\hat{D}_6 + c\hat{D}_5 + d\hat{D}_7 \mapsto a \bmod 3$. In terms of the effective class $[\hat{D}_0] := [\hat{D}_4] - [\hat{D}_1] \in \mathcal{C}_1$, representatives of the cosets of $\Lambda_{\mathcal{C}_1}$ in $\Cl(X) \cap \mathcal{C}_1$ are furnished by $(3m + 1)[\hat{D}_0]$, $(3m + 2)[\hat{D}_0]$, and $(3m + 3)[\hat{D}_0]$ for $m \geq 0$. In particular, these project to $1, 2, 0 \in \Cl(X)/\Lambda_{\mathcal{C}_1}$, respectively --- and we will use these numbers to label the cosets of $\Lambda_{\mathcal{C}_1}$ in $\Cl(X)$. The following two properties hold for these cosets.
\begin{itemize}
    \item[(1)] Asymptotically, divisors in the coset $2$ require two subtractions to reach their Zariski movable subtraction: in particular, the first subtraction step asymptotically results in a divisor in the non-trivial coset of $\Lambda_{\mathcal{C}_2}$ in $\mathcal{C}_2$ (noting $\Cl(X)/\Lambda_{\mathcal{C}_2} \cong \mathbb{Z}_2$).
    \item[(2)] Asymptotically, divisors in the cosets $0, 1, 2$ subtract to elements of $\Nef(X) \cap \varphi_2^* \Nef(X_2)$, $\Nef(X) \setminus \varphi_2^* \Nef(X_2)$, and $\varphi_2^* \Nef(X_2) \setminus \Nef(X)$, respectively.
\end{itemize}
To verify this, let us compute some integral Zariski decompositions. On the coset $0$, the Zariski composition of course coincides with the integral version and has image in $\mathcal{C}_1 \cap \Mov(X)$. Now consider the coset $1$: the (regular) Zariski decomposition of $(3m + 1)[\hat{D}_0]$ is 
\begin{equation}
    (3m + 1)[\hat{D}_0] = \left(m + \frac{1}{3}\right)[\hat{D}_8] + \left[ \left(m + \frac{1}{3} \right) [\hat{D}_5] + \left(m + \frac{1}{3} \right) [\hat{D}_6] + \left(2m + \frac{2}{3} \right) [\hat{D}_7] \right]
\end{equation}
from which one can compute that $P_{\mathbb{Z},\mathcal{C}_1}((3m + 1)[\hat{D}_0]) = [\hat{D}_2] + m[\hat{D}_8]$. One can verify for $m \geq 1$ that this belongs to $\Nef(X)$ and not $\varphi_2^* \Nef(X_2)$. The Zariski decomposition of $(3m + 2)[\hat{D}_0]$ is 
\begin{equation}
    (3m + 2)[\hat{D}_0] = \left(m + \frac{2}{3}\right)[\hat{D}_8] + \left[ \left(m + \frac{2}{3} \right) [\hat{D}_5] + \left(m + \frac{2}{3} \right) [\hat{D}_6] + \left(2m + \frac{4}{3} \right) [\hat{D}_7] \right]
\end{equation}
from which one can compute that $P_{\mathbb{Z},\mathcal{C}_1}((3m + 2)[\hat{D}_0]) = [\hat{D}_4] + [\hat{D}_5] + m[\hat{D}_8]$, which belongs to $\mathcal{C}_2$ and not $\Mov(X)$. Performing one more subtraction results in a class in $\varphi_2^* \Nef(X_2)$ and not in $\Nef(X)$.
Thus, for a general $[D] \in \mathcal{C}_1 \cap \Cl(X)$,
\begin{equation}
    \begin{aligned}
        h^0(X, \mathcal{O}_X(D))
        &= h^0(Y_1, \mathcal{O}_{Y_1}((f_1)_* D))
        = \chi(Y_1, \mathcal{O}_{Y_1}((f_1)_* D)) \\[2pt]
        &= \begin{cases}
            \chi(X, \mathcal{O}_X(D^\downarrow_{\mathcal{C}_1})) = \chi(X_2, \mathcal{O}_{X_2}((\varphi_2)_* D^\downarrow_{\mathcal{C}_1})) & [D] \in \Lambda_{\mathcal{C}_1} \\
            \chi(X, \mathcal{O}_X(D^\downarrow_{\mathcal{C}_1})) & [D] \in \Lambda_{\mathcal{C}_1} + [\hat{D}_0] \\
            \chi(X_2, \mathcal{O}_{X_2}((\varphi_2)_* D^\downarrow_{\mathcal{C}_1})) & [D] \in \Lambda_{\mathcal{C}_1} + 2[\hat{D}_0]
        \end{cases}
    \end{aligned}
\end{equation}
The first and second equalities are \cref{MAIN:contract}, while the third is \cref{MAIN:together}. Regrettably, in this particular example, $X$ and $X_2$ are both singular, so we cannot easily write down a closed-form expression for $h^0(X, \mathcal{O}_X(D))$ on $\mathcal{C}_1$ using Hirzebruch--Riemann--Roch; however, the purpose of this example was to exemplify non-trivial subtraction behavior.

\begin{remark}
    The failure of the Zariski subtraction to asymptotically agree with the Zariski movable subtraction, as in item (1) above, can be understood as the failure of the inequality in \cref{eq:bound_on_asymptotic_vanishing} to be saturated.
\end{remark}

\bibliographystyle{amsalpha-arxiv}
\bibliography{ref}

\end{document}